\documentclass[12pt, notitlepage]{amsart}
\usepackage{latexsym, amsfonts, amsmath, amssymb, amsthm, cite}
\usepackage{graphicx}
\usepackage[linktocpage=true,colorlinks,citecolor=blue,linkcolor=blue,urlcolor=blue]{hyperref}
\usepackage{amssymb}
\usepackage{cite}
\usepackage{bbm}
\usepackage{amsmath}
\usepackage{latexsym}
\usepackage{amscd}
\usepackage{tikz}
\usepackage{amsthm}
\usepackage{mathrsfs}
\usepackage{url}
\usepackage{bbm}
\usepackage[utf8]{inputenc}
\usepackage[english]{babel}
\usepackage{amsfonts}
\usepackage{mathtools}
\usepackage{comment}
\usepackage[autostyle]{csquotes}
\usepackage[colorinlistoftodos]{todonotes}
\usepackage[noabbrev,capitalize]{cleveref}
\usepackage{adjustbox}
\usepackage{mathrsfs}
\crefname{equation}{}{}
\usepackage{graphicx, tikz}
\usepackage[margin=1in]{geometry}
\usepackage{microtype}

\numberwithin{equation}{section}
\def\R{\mathbb R}

\def\C{\mathbb C}
\def\P{\mathbb P}
\def\N{\mathbb N}
\def\E{\mathbb E}

\def\TT{\widetilde{T}}

\def\CA{\mathcal{A}}

\def\rsum{\sideset{_{} ^{}} {_{} ^{\star}}{\sum }}

\newtheorem{theorem}{Theorem}[section]
\newtheorem{lemma}[theorem]{Lemma}
\newtheorem{proposition}[theorem]{Proposition}

\theoremstyle{remark}

\theoremstyle{definition}

\theoremstyle{remark}

\numberwithin{equation}{section}
\begin{document}
	\title[Distribution of random multiplicative functions in short intervals]{Distribution of random multiplicative functions in short intervals, with proper normalization }
	\author{Adam J Harper}
	\address{Mathematics Institute, Zeeman Building, University of Warwick, Coventry CV4
7AL, England}
	\email{A.Harper@warwick.ac.uk}
	\author{Kannan Soundararajan}
	\address{Department of Mathematics, Stanford University, Stanford, CA, USA}
	\email{ksound@stanford.edu}
\author{Max Wenqiang Xu} 
\address{Yau Mathematical Sciences Center, Tsinghua University, Beijing, China}
\email{maxxu1729@gmail.com}

	\begin{abstract}
 We determine the limiting distribution of partial sums of a Steinhaus random multiplicative function $\sum_{x\le n \le x+y} f(n)$ over short intervals $[x, x+y]$, where $y \rightarrow \infty$ but $y=o(x)$. We show that with appropriate normalization, the limiting distribution is Gaussian for all such $y$. A key new feature of our result is that the normalization factor is different from the standard deviation $\sqrt{y}$ when $y$ is very close to $x$. In contrast, when $y \asymp x$ there is no normalization for which the limiting distribution is a non-degenerate Gaussian.
	\end{abstract}
	\maketitle

 \section{Introduction}
 The study of random multiplicative functions has been very active in recent years, motivated by various connections and applications in number theory, probability, and analysis. Define a {\em Steinhaus random multiplicative function} $f : \N \rightarrow \C$, by taking $(f(p))_{p \; \text{prime}}$ to be independent random variables distributed uniformly on the complex unit circle, and setting $f(n) := \prod_{p^{a}\| n} f(p)^{a}$ for all natural numbers $n$ (where $p^a \| n$ means that $p^a$ is the highest power of the prime $p$ that divides $n$). These are often considered to be models for number-theoretic functions like Dirichlet characters $\chi(n)$.  A {\em Rademacher random multiplicative function} is defined by letting $(f(p))_{p \; \text{prime}}$ be independent, taking values $\pm 1$ with probability $\frac 12$ each, and setting $f(n) := \prod_{p |n} f(p)$ for all squarefree $n$, and $f(n) = 0$ when $n$ is not squarefree.  A fundamental question in the area is to study the statistical properties of partial sums $\sum_{1\le n \le x}f(n)$, and more generally, weighted partial sums $\sum_{1\le n \le x} a(n)f(n)$. 

One striking result is Harper's resolution \cite{HarperLow} of Helson's conjecture \cite{Helson}.  Harper proved (in both the Steinhaus and Rademacher settings) that for all large $x$
 \begin{equation}\label{eqn: Harper}
  \E\Big[\Big|\sum_{1\le n \le x} f(n)\Big|\Big] \asymp \frac{\sqrt{x}}{(\log \log x)^{\frac 14}}, 
 \end{equation}
 so that partial sums of random multiplicative functions exhibit ``better than square-root cancellation."   This implies that the normalized partial sum, with the natural normalization 
 $$
 \Big(\E\Big[\Big|\sum_{1\le n \le x} f(n)\Big|^2\Big]\Big)^{-\frac 12}  \asymp \frac{1}{\sqrt{x}}, 
 $$
 has a trivial limiting distribution. In earlier work \cite{Harper}, it had been proved (in the Rademacher case) that 
 $$ 
 \Big(\E\Big[\Big|\sum_{1\le n \le x} f(n)\Big|^2\Big] \Big)^{-\frac 12} \sum_{n \leq x} f(n)
 $$ 
 could not converge in distribution to a standard Gaussian $N(0,1)$ (though this did not exclude the more subtle possibility that it might converge to some Gaussian with variance different from $1$).

 A central limit theorem may emerge in cases with general weights $a(n)$ that somewhat disrupt the multiplicative structure.  One particular example that has attracted considerable attention is when $a(n)$ is an indicator function of a short interval, namely we study the limiting distribution of $\sum_{x\le n \le x+y}f(n)$ where $y \rightarrow \infty$, but $y=o(x)$. It was first proved by Chatterjee and Soundararajan in \cite{CS} (for Rademacher $f(n)$) that the limiting distribution, with usual normalization $1/\sqrt
 {y}$, is Gaussian as long as $Cx^{1/5} \log x \leq y = o(x/\log x)$. Note that the lower bound condition on $y$ is, in a sense, much less interesting here; it was imposed so that one knows unconditionally that all intervals $[x, x+y]$ contain roughly the expected quantity of squarefree numbers. This result was improved by Soundararajan and Xu in \cite{SoundXu} to the wider range $Cx^{1/5} \log x \leq y\le x/(\log x)^{\log4-1 + \epsilon}$, and also extended to allow Steinhaus $f(n)$ whenever $y \rightarrow \infty$ and $y\le x/(\log x)^{\log4-1 + \epsilon}$. On the other hand, by using the triangle inequality and Harper's theorem \eqref{eqn: Harper}, one can see that when $y$ is sufficiently close to $x$, $\frac{1}{\sqrt{y}} \sum_{x\le n \le x+y}f(n)$ has a trivial limiting distribution. In a recent work of Caich~\cite{caichshort}, it is proved that such trivial limiting distribution holds as long as $y\ge \frac{x}{\exp((\log \log x)^{1/2-\epsilon})}$. It has been an intriguing question to properly understand all these transitions in behavior. In particular, one would like to know whether $\frac{1}{\sqrt{y}} \sum_{x\le n \le x+y}f(n)$ actually has Gaussian behavior for $y$ larger than $x/(\log x)^{\log4-1 + \epsilon}$, and what can be said about the distribution for those very large $y$ where the limit becomes trivial.

 Perhaps surprisingly, 
we prove a Gaussian limit theorem for partial sums
$\sum_{x\le n \le x+y}f(n)$ of a Steinhaus random multiplicative function whenever $y \rightarrow \infty$ with $y=o(x)$. This completely determines the distribution over all short intervals, thereby completing the previous results in \cite{CS, SoundXu, caichshort}.

\begin{theorem}\label{thm: main}
Let $f(n)$ be a Steinhaus random multiplicative function. There exists a deterministic scaling factor $V(x,y)$ such that the quantity 
\[ \frac{1}{\sqrt{V(x, y)}} \sum_{x\le n \le x+y}f(n)  \]
converges in distribution to a standard complex Gaussian random variable with mean $0$ and variance $1$, as $y\to \infty$ with $y=o(x)$.  The scaling factor satisfies 
\begin{equation}\label{eqn: hVexplicit}
V(x,y) \sim \frac{y}{\sqrt{2\pi}} \int_{-\kappa}^{\kappa} e^{-u^{2}/2} du, \qquad \text{with} \qquad \kappa= \frac{\log (x/y)}{\sqrt{2\log \log x}}, 
\end{equation}
with the asymptotic holding as $x/y \to \infty$.   In particular 
\begin{equation}\label{eqn: V}
  V(x, y) \asymp  y \cdot \min \Big\{1, \frac{\log (x/y)}{\sqrt{\log \log x}}\Big\},  
\end{equation}
and $V(x, y) \sim y$ as $\frac{\log(x/y)}{\sqrt{\log \log x}} \to +\infty$.
\end{theorem}

The value of $V(x,y)$ that emerges naturally from our arguments is expressed in terms of a quantity $\mu(x,y)$, which is (up to scaling) a certain random walk probability. See Proposition \ref{prop: mu} and \eqref{4.3} below, and the discussion in sections \ref{hsubsecoutline} and \ref{subsec:introconc}. However, we are free to replace $V(x,y)$ in Theorem \ref{thm: main} by any quantity to which it is asymptotic (as $y\to \infty$ with $y=o(x)$), and our estimation of $\mu(x,y)$ (again see Proposition \ref{prop: mu}) shows that \eqref{eqn: hVexplicit} is such a choice.

This is the first time that partial sums of random multiplicative functions with non-obvious normalization have been found to have a Gaussian limiting distribution. There are a few other number-theoretic situations where somewhat similar (and very interesting) phenomena arise.
For example, Montgomery and Soundararajan~\cite{MS04} show, assuming various strong but plausible hypotheses, that the distribution of the von Mangoldt function in short intervals  $\sum_{x \leq n \leq x+y} \Lambda(n)$ as $x \leq N$ varies and with $y$ in the range $N^{\delta} \leq y \leq N^{1-\delta}$, should be roughly Gaussian, but with variance differing from the obvious guess by a constant factor. Gorodetsky, Mangerel and Rodgers~\cite{GMR} recently gave a beautiful proof that the analogous count of squarefree numbers in short intervals, where $x \leq N$ varies and $y \rightarrow \infty$ with $y = N^{o(1)}$, is roughly Gaussian with variance differing from the obvious guess by a power. In the squarefree case, the variance drop reflects the highly structured nature of the sequence of squarefree numbers. For the von Mangoldt function (equivalently primes) the situation is closer to our case here, but again the variance drop reflects a structural property of the sequence, namely its connection with the zeros of the Riemann zeta function. We are not aware of a similar phenomenon having been observed before in any random or deterministic context quite analogous to the present one (e.g. involving character sums). We also note that in both the primes and the squarefree examples, one sees the variance drop simply by computing the variance of the sequence itself (and the results on limiting distributions are obtained by a moment method). In contrast, the variance of our random sums has the obvious size one would expect (from orthogonality), {\em it is only in the limiting distribution that one sees a change of scaling}.

\vspace{12pt}
Theorem \ref{thm: main} contrasts significantly with the distributional behaviour of the long sums $\sum_{n \leq x} f(n)$, or $\sum_{x\le n \le x+y}f(n)$ where $y \asymp x$. In view of  \eqref{eqn: Harper} and Theorem \ref{thm: main}, one might initially suppose that if one rescaled these sums by something like $\frac{\sqrt{x}}{(\log \log x)^{1/4}}$, then a Gaussian limiting distribution should appear. However, the ideas in the proofs, and specifically the connection with so-called {\em multiplicative chaos} (on which we elaborate below), suggest otherwise. Harper~\cite{HarperLow} proposed, in both the Steinhaus and Rademacher cases, that $\frac{(\log \log x)^{1/4}}{\sqrt{x}} \sum_{n \leq x} f(n)$ should have limiting distributions related to the total mass of critical multiplicative chaos. Later, Gorodetsky and Wong made an explicit conjecture of the limiting distribution in \cite[Conjecture 1.6]{GW24}. In particular, this would imply that the tail probabilities $\P\left(|\sum_{n \leq x} f(n)| \geq \lambda \frac{\sqrt{x}}{(\log\log x)^{1/4}} \right)$ should be of size $\asymp 1/\lambda^2$ for fixed large $\lambda$, much heavier than a Gaussian. Recent breakthrough work of Gorodetsky and Wong \cite{GW-Final} determines, in the Steinhaus case (although not yet the Rademacher case), that the limiting distribution is indeed of this form. Their arguments also apply to $\sum_{x\le n \le x+y}f(n)$ when $y \asymp x$.

Exploiting the known blow-up of the low moments $(\frac{(\log\log x)^{1/4}}{\sqrt{x}})^{2q} \E[|\sum_{n \leq x} f(n)|^{2q}]$ as $q$ approaches $1$, we establish the 
following heavy tail bound.
\begin{proposition}\label{prop: tailsprop}
Let $f(n)$ be a Steinhaus or Rademacher random multiplicative function. There exists a constant $A$ such that, for all large $\lambda$ and $x$ sufficiently large in terms of $\lambda$, we have
$$ 
\P\Big(|\sum_{n \leq x} f(n)| \geq \lambda \frac{\sqrt{x}}{(\log\log x)^{1/4}} \Big) \geq \frac{1}{\lambda^A}. 
$$
Moreover, for any small $\delta > 0$ there exists $A(\delta)$ such that
$$ 
\P\Big(|\sum_{x \leq n \leq (1+\delta)x} f(n)| \geq \lambda \frac{\sqrt{\delta x}}{(\log\log x)^{1/4}} \Big) \geq \frac{1}{\lambda^{A(\delta)}}
$$
whenever $\lambda$ is large enough in terms of $\delta$, and $x$ is sufficiently large in terms of $\lambda$.
\end{proposition}
Proposition \ref{prop: tailsprop} will be quickly proved in Section \ref{SEC: LONG SUM}, essentially independently of the rest of the proofs in this paper. As we describe explicitly there, this implies that there is {\em no} choice of normalizing factor $V(x)$ for which $\frac{1}{\sqrt{V(x)}} \sum_{n \leq x} f(n)$ can converge in distribution to a non-degenerate Gaussian.

We shall comment more on the work of Gorodetsky and Wong~\cite{GW-Final} (and also S.~Hardy~\cite{hardy}) on the full-sum limiting distribution in Section \ref{subsec:furtherremarks}. Here the limiting distribution is non-Gaussian, or more precisely it may be thought of as a Gaussian whose variance is itself a genuine, heavy-tailed random variable. Perhaps it is then surprising that in our short interval case, there exists a \textit{deterministic normalization} which makes the random sum behave like a Gaussian. This new feature is due to a concentration phenomenon of the conditional variance, which we will try to explain below (see Section \ref{subsec:introconc}).

\subsection{Outline proof strategy for Theorem \ref{thm: main}}\label{hsubsecoutline}
The usual framework for proving a limiting distribution for partial sums of random multiplicative functions (originating in Harper's work~\cite{Harper} on sums with few prime factors) is by applying a suitable version of the {\em martingale central limit theorem}. This is the same framework we will use but with several new additional features, both conceptual and technical. Essentially, we require a fusion of martingale technique with some quite delicate {\em barrier arguments}, developing the previous work of e.g. Caich~\cite{caichshort}, Harper~\cite{HarperLow, Harperlargevalue} and Xu~\cite{Xu}.

We apply a version of McLeish's central limit theorem \cite{McLeish} developed in \cite{SoundXu}. A direct application would require showing that the fourth moment of the random sum  $\sum_{x\le n \le x+y}f(n)$ (after a suitable pruning of the terms) is dominated by the diagonal contribution.  However, in fact the fourth moment blows up once $y$ is large enough. This forces the applicable range of $y$ to be $\ll  x/(\log x)^{c}$ for certain $c>0$, exactly the limitation of the results in \cite{CS} and \cite{SoundXu}. 

To overcome the limitation, our strategy is to apply the martingale central limit theorem only after {\em conditioning on small primes} $f(p)$ for $p\le z$ (for a parameter $z$ that needs to be carefully chosen), instead of applying it to the full sum from the beginning.   For our choice of $z$, the set of integers in $[x,x+y]$ with all prime factors below $z$ has negligible size, and may be ignored.   By using the multiplicative property of $f$, the sum over the remaining integers may be decomposed as 
\[
\sum_{\substack{x \le n \le x+y \\ P(n) > z  }  } f(n) = \sum_{ \substack{ 1<m\le x+y \\ p|m \implies p>z}} f(m) \sum_{\substack{x/m\le n \le (x+y)/m \\ p|n \implies p\le z  }  } f(n),
\] 
where henceforth we write $P(n)$ for the largest prime factor of $n$.  
After the conditioning, the inner sum is fixed and can be viewed as a constant coefficient $a(m)$ indexed by $m$. Then one can apply the version of the martingale central limit theorem \cite{SoundXu} for general sums $\sum_{m}a(m)f(m)$ aiming to get a \textit{conditional} central limit theorem. As noted, to make this work we need to choose $z$ carefully.
Suppose such a choice of $z$ exists and the theorem in \cite{SoundXu} can be successfully applied. Then this implies that given a fixed choice of $f(p)$ for $p\le z$, the random sum is approximately Gaussian, but with a conditional variance $V_f(x, y) = \sum_{m}|a(m)|^2$ which depends on the values of $f(p)$ for $p\le z$. 

In order to get a limiting distribution for the original random sum, we need to ``reveal'' the conditioned choices of $f(p)$ with $p\le z$ and 
understand the behavior of the conditional random variance $V_f(x, y)$. It is not clear a priori whether $V_f(x, y)$ has genuinely nontrivial distribution (as the analogous object does for the long sum $\sum_{n \leq x} f(n)$) or typically behaves like a deterministic quantity, i.e., for asymptotically almost all realizations of $f(p)$ with $p\le z$, it is asymptotic to some deterministic quantity $V(x, y)$ (which doesn't depend on $f$, only on $x, y$).  One important feature of our proof is a concentration phenomenon which happens with the choice $z = x^{\frac{1}{\log \log \log x}}$ (say): namely, for a deterministic quantity $V(x,y)$, one has typically $V_f(x, y)\sim V(x, y)$. Moreover,  
at the same time, this choice of $z$ is eligible for passing through the conditions of the martingale central limit theorem in \cite{SoundXu}.

In the next couple of subsections, we will try to explain a little more of what is involved in the conditioned application of the martingale central limit theorem (which amounts to needing to understand something like the ``off-diagonal'' contribution to a conditional fourth moment), and in showing the crucial concentration of the conditional variance.

In the process of proving the concentration result, we see the phase transition in $V(x, y)$ when $y$ gets close to $x$. As found in previous work \cite{SoundXu, CS}, $\frac{1}{\sqrt{y}} \sum_{x\le n \le x+y}f(n)$ has the expected Gaussian limiting distribution when $y\ll x/(\log x)^{c}$ for suitable $c$, so our choice of $V(x, y)$ should satisfy $V(x, y)\sim  y$ (the obvious value, namely the variance of $\sum_{x\le n \le x+y}f(n)$) for such $y$.  
However, as $y$ approaches $x$, a ``better than square-root cancellation" phenomenon shows up, produced by multiplicative chaos effects. Indeed, the transition point of $y$ is closely connected to a ballot-type problem in the theory of random walks. The work of Caich \cite{caichshort} gives that $y\approx x/\exp((\log \log x)^{1/2})$ is a transition point for ``better than square-root cancellation" (working at the level of order of magnitude of low moments), which is also the transition point for the shape of $V(x, y)$ changing from $y$ to $ \asymp y \cdot \frac{\log(x/y)}{\sqrt{\log \log x}}$. We also refer readers to \cite{Chang, Xu} for related proofs exploring such transitions (again at the order of magnitude level) in other settings.

At a very high level, one might have in mind that a standard way of proving a central limit theorem (although generally not easy to apply to random multiplicative function problems) is the method of moments, computing all moments of the random variables of interest and showing they converge to Gaussian moments. More sophisticated methods, such as martingale techniques, can essentially reduce from needing to compute all moments to just the fourth moment. Chatterjee and Soundararajan's work~\cite{CS} on short interval sums, using Stein's method, operates on this level. To obtain their improvement in the range of $y$, Soundararajan and Xu~\cite{SoundXu} continue to work on the fourth moment level (now using martingales), but first remove a sparse subset of integers from the short interval sum so that the fourth moment behaves well on a wider range. In proving Theorem \ref{thm: main} on the full range of $y$, we work in a regime where not only the fourth moment, but even the second moment, does not behave well and reflect the genuine distributional behaviour. As we shall try to explain, to address this we make use of ``barriers'' editing our random variables (after conditioning) on the ``Fourier side'' of random Euler products, rather than removing terms from the original sums. The argument thus becomes somewhat indirect, but ultimately explains clearly the cause of the transition in $V(x,y)$, and the difference between short interval sums (where $y=o(x)$) and long sums. We are not aware of any way to understand this transition operating simply with $\sum_{x\le n \le x+y}f(n)$ in ``physical space''.

\subsection{Concentration of conditional variance}\label{subsec:introconc}
As a first step, one can see that the definition $\sum_{m}|a(m)|^2$ of $V_f(x, y)$ is a discrete mean square of certain random sums (see \eqref{eqn: Vf}, below). As usual in this area, we apply Parseval's identity to transform the sum to its Fourier side, i.e. to some continuous second moment $\int |F_z(s)|^{2}\,ds$
of the random Euler product $F_z(s) := \prod_{p\le z}(1-\frac{f(p)}{p^{s}})^{-1}$. This step is standard but more demanding than in many previous works (compare with e.g. Caich's paper~\cite{caichshort}), since here we need an exact \textit{asymptotic} expression for $V_f(x, y)$.  Even the loss of a multiplicative constant factor would break the proof of concentration and of a limiting distribution. This generates some technical issues, but these are manageable and we find (with high probability over the $(f(p))_{p \leq z}$) that
\[V_f(x, y) \sim  \frac{e^{-\gamma}}{4\pi} \frac{y}{\log z} \frac{1}{T} \int_{-T(\log T)^{100}}^{T(\log T)^{100} } |F_z(\tfrac{1}{2}+it)|^{2}|K_T(t)|^{2} dt, \]
where $K_T(t)$ is a certain real-valued kernel which decays rapidly for large $t$ and thus allows the integral to be truncated at $|t|\approx  T$ with $T \approx 2x/y$. Notice that going from an interval $[x,x+y]$ in physical space, to an integral of effective length $\asymp x/y$ on the Fourier side, is consistent with the usual numerology in multiplicative number theory, in particular shorter intervals require more frequencies $t$ on the Fourier side. This reformulating of $V_f(x, y)$ forms the bulk of Section~\ref{SEC: CONDVAR}.

We must proceed to analyze $V_f(x, y)$ further, needing to understand this not just in (typical) order of magnitude, but seeking a more delicate concentration result. One might hope that an estimate like 
\begin{equation}\label{dream}
   \E[|V_f(x, y) - V(x, y)|^{2}] = o(V(x, y)^{2})  
\end{equation}
holds for our candidate $V(x, y)$, which if true would establish that typically $V_f(x, y)\sim V(x, y)$. However, when $y$ is large this second moment type computation (which, since $V_f(x, y)$ is itself a mean square, is really like a fourth moment) is dominated by certain relatively rare events, and (as expected given the failure of a direct application of McLeish's theorem) we cannot achieve such concentration directly. To overcome this, 
the crucial idea is to impose barrier events $\mathcal{G}(t)$ on the growth of the partial Euler products $F_z(\frac 12+it)$ at all different ``scales''. We hope that $\mathcal{G}(t)$ holds, simultaneously at all relevant points $t$, with high probability, and then it would suffice to show that \eqref{dream} holds with the restriction $\mathcal{G}(t)$ inside our integral approximation to $V_f(x, y)$. The purpose of $\mathcal{G}(t)$ is to suppress any blow-up of \eqref{dream} created by rare extreme events. Setting barrier events is a well developed tool in the probabilistic study of so-called multiplicative chaos, and it first shows up in \cite{HarperLow} for the study of random multiplicative functions. 

However, here it turns out that inserting a high probability event $\mathcal{G}(t)$  is still not enough to establish the concentration estimate. We solve this by inserting a further refined constraint $\mathcal{H}(t)$ and show that $\E[\mathbbm{1}_{\mathcal{G}(t)}\mathbbm{1}_{\mathcal{H}(t)~\text{fails}}|F_z(\frac 12+it)|^{2} ] $ is very small. We emphasize that the combination of $\mathcal{G}(t)$ with $\mathcal{H}(t)$, and of analyzing what happens in high probability along with what happens in mean square, is crucial here.  We cannot show that the stronger barrier event $\mathcal{H}(t)$ holds at all points $t$ with high probability, and neither is it true that $\E[\mathbbm{1}_{\mathcal{H}(t)~\text{fails}}|F_z(1/2+it)|^{2} ] $ (without the initial barrier $\mathbbm{1}_{\mathcal{G}(t)}$) is small when $y$ is large (in fact it would be $\sim \E[|F_z(\frac 12+it)|^{2} ] $). The combination of the two barriers allows us further to assume $\mathcal{H}(t)$ holds, and then we aim for a concentration result like 
\[
\E\Big[\Big|\int_{-T(\log T)^{100}}^{T(\log T)^{100}}|F_z(1/2+it)|^{2} K_T(t)^{2} \mathbbm{1}_{\mathcal{H}(t)} dt - \text{Main Term}\Big|^{2}\Big] = o((\text{Main Term})^{2}). 
\]
This is accomplished in Section~\ref{SEC: CONCEN}. Here the Main Term is simply 
$$
\E \Big[ \int_{-T(\log T)^{100}}^{T(\log T)^{100}}|F_z(1/2+it)|^{2} K_T(t)^{2} \mathbbm{1}_{\mathcal{H}(t)} dt\Big],
$$
 and thus we take $V(x, y)$ to be a scaled version of this. The transition \eqref{eqn: V} in $V(x,y)$, as $y$ varies with $x$, precisely reflects the effect of $\mathbbm{1}_{\mathcal{H}(t)}$ (which translates into multiplication by a certain random walk probability) inside this expectation.

This is the point in the argument where we see the conceptual origin of $V(x, y)$, and we shall try to expand on what is happening. It turns out that the Euler products $F_z(\frac 12+it)$ behave more or less independently on points $t$ that are at least 1 apart, say. Thus our integral, which is  (recalling that $K_T(t)$ basically serves to truncate things at $|t|\approx  T$)
$$
\approx \sum_{|n| \leq T(\log T)^{100}} \int_{n-\frac 12}^{n+\frac 12}|F_z(\tfrac 12+it)|^{2} K_T(t)^{2} \mathbbm{1}_{\mathcal{H}(t)} dt \approx \sum_{|n| \leq T} \int_{n-\frac 12}^{n+\frac 12}|F_z(\frac 12+it)|^{2} \mathbbm{1}_{\mathcal{H}(t)} dt,
$$
has the shape of a sum of $\approx T$ roughly independent terms. Provided that $T \rightarrow \infty$ (which exactly means $y=o(x)$), results like the Law of Large Numbers lead one to expect this to concentrate around its mean. We find this is indeed the case, but since each piece $\int_{n-\frac 12}^{n+\frac 12}|F_z(\frac 12+it)|^{2} \mathbbm{1}_{\mathcal{H}(t)} dt$ is heavy-tailed it is a delicate matter to establish this.

In fact, when we expand the square we need to deal with the expectation of a product of two factors $\E[|F_{z}(\frac 12+it_1)|^{2}\mathbbm{1}_{\mathcal{H}(t_1)} |F_{z}(\frac 12+it_2)|^{2} \mathbbm{1}_{\mathcal{H}(t_2)}]$. When $|t_1 - t_2| \rightarrow \infty$ we can show almost perfect decorrelation; this requires a slightly different analysis than in previous works (see Lemma~\ref{lem: concentration}, and its application inside the proof of Proposition~\ref{prop: concentration} below). We can also prove a decorrelation result when $|t_1-t_2|$ is smaller (see Lemma~\ref{lem: slice}), approximating the expectation by the product of two expectations (each with the corresponding barrier events involved), but now this can only be done for the parts of the Euler products involving sufficiently large primes in terms of $|t_1-t_2|$. {\em The barrier conditions $\mathcal{H}(t_1), \mathcal{H}(t_2)$ are invoked to control the small prime contributions.} More precisely, to succeed we need the barriers to restrict the small prime contribution to be (up to some normalization) a bit less than $\sqrt{T}$, so when squared inside $|F_{z}(1/2+it_1)|^{2}$ we get a factor {\em a bit smaller than the total integral length $T$}. The high probability barrier $\mathcal{G}(t_1)$ can only impose a restriction a bit larger than $\sqrt{T}$, but the refined barrier $\mathcal{H}(t_1)$ enables us to go a bit below this, and so complete the concentration argument. The factor $\min \{1, \frac{\log (x/y)}{\sqrt{\log \log x}}\} \sim \min \{1, \frac{\log T}{\sqrt{\log \log x}}\}$ that arises in $V(x, y)$ reflects this size $\approx \sqrt{T}$, on taking a logarithm we get a restriction $\approx \frac 12\log T$ (up to some recentering) on the logarithms of the partial Euler products (which are random sums over primes), the denominator $\sqrt{\log \log x}$ reflects the ``number of scales'' in the sums.

In summary, one should think there is a sort of competition between the effect of the heavy-tailedness of each piece $\int_{n-\frac 12}^{n+\frac 12}|F_z(\frac 12+it)|^{2} \mathbbm{1}_{\mathcal{H}(t)} dt$ (kept somewhat under control by the barrier $\mathcal{H}(t)$), and the concentration effect of summing $\approx T$ pieces. Provided $T \rightarrow \infty$, the concentration effect wins out, but with the ``rate'' of concentration (i.e. the exceptional probability that $V_f(x, y)$ is not close to $V(x, y)$) becoming worse when $T$ grows more slowly. Once $y \asymp x$, and so $T \asymp 1$, the heavy-tailedness dominates and one no longer gets concentration.

\subsection{Asymptotic of the fourth moment}
In the application of the (conditioned) martingale central limit theorem, the main issue is a certain conditional fourth moment computation. We need to demonstrate that all ``off-diagonal'' contributions are negligible. To illustrate the problem, one major task is to show that, with high probability over the $(f(p))_{p \leq z}$ and as $x\to \infty$, 
\[ 
\Big| \sum_{\substack{ m_1, m_2, m_3, m_4 \in \CA \\ m_1 m_2= m_3 m_4 \\ m_1 \neq m_3, m_2\neq m_4 \\ P(m_1)=P(m_3) \\ P(m_2)=P(m_4)}} \ \ \ 
\sum_{\substack{n_1, n_2, n_3, n_4 \\ \frac{x}{m_j} \le n_j \le \frac{x+y}{m_j} \\ P(n_j)\le z}} f(n_1)f(n_2)\overline{f(n_3) f(n_4)}  \Big| =o(V_f(x, y)^2),
\]
where $\CA$ is the set of integers in $(1,x+y]$ with all their prime factors $> z$.
Our approach is to write each inner sum over $n_j$ as an integral involving $F_z(s)$ by a suitable smoothed Perron formula: 
\[ 
\sum_{\substack{n \in [\frac{x}{m}, \frac{x+y}{m}]\\P(n)\le z}} f(n) =  \frac{1}{2\pi }\int_{\infty}^{\infty} F_z(\tfrac 12 +it) \Big(\frac{x}{m}\Big)^{\frac 12+it} W(t)dt,
\]
where the smooth weight $W(t)$ effectively restricts the integral to $|t|\le (\log x)^{100}$.  
  Then the quantity we need to estimate is a quadruple integral involving the product of four copies of random Euler products $F_z(\frac 12+ it_j)$ ($1\le j \le 4$), and importantly a sort of harmonic sum that comes from the $1/m^{s}$ factors, namely 
\[
G(t_1, t_2, t_3, t_4):= \sum_{\substack{m_1, m_2, m_3, m_4 >1\\ p|m_j \implies p\in (z,x+y] \\ m_1m_2=m_3m_4\\ m_1\neq m_3, m_2\neq m_4 \\ P(m_1)= P(m_3) \\ P(m_2) =P(m_4) }} \frac{1}{m_1^{\frac 12 + it_1}m_2^{\frac 12 + it_2}} \cdot \frac{1}{m_3^{\frac 12 - it_3}m_4^{\frac 12 - it_4}} . 
\]
One can show (with some technical work, see Lemmas \ref{lem4.4} and \ref{lem: R} below) that the $G(t_1, t_2, t_3, t_4)$ factor
provides a good saving unless $t_1, t_2, t_3, t_4$ are all very close to each other. Thus we may restrict our attention to the case that all four $t_j$ are close, and the resulting quadruple product would behave like the product of two squares $|F_z(1/2+it_1)|^{2}$ and  $|F_z(1/2+it_2)|^{2}$ with $t_1$ and $t_2$ close. Indeed, we are left to deal with a quantity like 
(with a suitable weight $W$)  
\[
\int_{\substack{t_1\approx t_2 \\ |t_1|, |t_2|\le (\log x)^{100}}} |F_z(1/2+it_1)|^{2}|F_z(1/2+it_2)|^{2}  |W(t_1)|^2 |W(t_2)|^2 dt_1 dt_2.
 \]

It would suffice to show that the expectation of this integral is ``small'' (essentially compared with $V(x,y)^2$, up to some scaling factors that we have suppressed). But as with \eqref{dream}, or a direct computation of the fourth moment, the expectation would in fact blow up. Instead, we use the method of barrier events again, and just like before, we need to put some barriers $\mathcal{G}^{*}(t)$ and $\mathcal
{H}^{*}(t)$ on the growth of the Euler products. With $\mathcal{G}^{*}(t)$ and $\mathcal
{H}^{*}(t)$ set appropriately, this final analysis becomes relatively straightforward, see Propositions \ref{prop4.6}, \ref{prop: no t} and the surrounding calculations below. We note, however, that this is a point in the proof where the choice of $z$ is crucial. We lose factors of the shape $(\frac{\log x}{\log z})^{O(1)}$ and need the saving coming from $\mathcal{H}^{*}(t)$, which is very limited, to (more than) compensate for this. Thus although $z$ cannot be too close to $x$ (in order for earlier steps of the argument to work), it also cannot be too far from $x$.

One small difference between the barrier events here and the previous $\mathcal{G}(t)$ and $\mathcal{H}(t)$, is that our integral here has long length $(\log x)^{100}$. Thus our barrier events must take the size of $t$ into account, in fact when $|t| \geq T$ we need the barrier to increase a bit with $t$ (becoming weaker) so that the exceptional probability of exceeding the barrier decreases proportional to $|t|$, otherwise a union bound on the exceptional probabilities would blow up. But the decay factors involving $W(t_1)$ and $W(t_2)$ compensate for this weakening.

\subsection{Further remarks}\label{subsec:furtherremarks}
We already mentioned that, for the appropriately normalized long sum $\frac{(\log \log x)^{\frac 14}}{\sqrt{x}} \sum_{n \leq x} f(n)$ of a Steinhaus random multiplicative function, a recent breakthrough of Gorodetsky and Wong~\cite{GW-Final} establishes convergence in distribution to a non-Gaussian, heavy tailed limit. A little earlier, S. Hardy~\cite{hardy} established the analogous result for the sub-sum $\frac{(\log \log x)^{\frac 14}}{\sqrt{x}} \sum_{n \leq x, P(n) > \sqrt{x}} f(n)$ over integers having a large prime factor. The simplifying condition $P(n) > \sqrt{x}$ means that Hardy does not require any martingale techniques (after conditioning on $(f(p))_{p \leq \sqrt{x}}$, he has a classical weighted sum of independent random variables to work with). The bulk of his task is understanding the conditional variance, which he finds to converge to a heavy-tailed distribution closely related to the total mass of critical multiplicative chaos (unlike the concentration that we find in the short interval setting here). Hardy's proof of convergence ultimately relies on some existing results from the theory of multiplicative chaos, but to make these applicable he must bring the conditional variance into an appropriate form. This uses a barrier-weighted mean square argument, sharing some features with our Proposition~\ref{prop: concentration} below.

Gorodetsky and Wong~\cite{GW-Final}, handling the full sum, require martingale theory. A major task for them is the analysis of the arising ``bracket process'', which is the same type of object as the conditional variance investigated by Hardy~\cite{hardy} and by us. Gorodetsky and Wong's beautiful proof has many features in common, and also many differences from, the work of Hardy~\cite{hardy} and our work here, and we only mention a few points. Firstly, to bring the conditional variance into a nice form we work mostly on the ``Fourier analytic'' side, and Hardy does likewise. Gorodetsky and Wong's bracket process has a bit more complicated structure, and they require a non-trivial ``truncation'' procedure on the physical side to produce an object they can work with. Secondly, Gorodetsky and Wong eschew the use of barriers, which are crucial to our work and which Hardy also makes some use of. Instead, they develop and use a new penalized second moment method, see their paper~\cite{GW-Final} for a detailed discussion of this. (Although behind this lie some of the same tools, like an approximate Girsanov theorem, that lie behind our barrier calculations.) Thirdly, note that about a third of our work here (almost all of section~\ref{Sec: CLT}) is the analysis of a conditional fourth moment contribution, to make the martingale central limit theorem applicable. Because of the form of martingale central limit theorem that they use, Gorodetsky and Wong do not require this type of calculation, but instead they require (and establish) demanding information about the nature of the convergence of their bracket process.

Inspecting our proofs (see Theorem \ref{thm: concentration} below), the reader will see that we show $V_f(x, y)\sim V(x, y)$ holds with probability $1- O((\log T)^{-\frac 15})$. Recall that $T \approx 2x/y$. On first glance this rate estimate may look weak, but in fact it probably is not far from the truth, apart from the precise exponent of $\log T$. For examining our integral expression for $V_f(x, y)$, we see that if any of the integrals $\int_{n-\frac 12}^{n+\frac 12}|F_z(\frac 12+it)|^{2} dt$, for $|n| \leq T$, exceeds a large multiple of $T \log T \frac{\log z}{\sqrt{\log \log z}}$ then $V_f(x, y)$ will be significantly larger than $V(x, y)$. If $T$ isn't too big compared with $z$ (i.e. if $y$ is reasonably close to $x$), the probability of this for any given $n$ should be $\approx \frac{1}{T\log T}$, see e.g. the discussion of tail probabilities in the introduction of Harper's paper~\cite{HarperLow}. And since these integrals should behave more or less independently for different $n$, the probability of it happening for {\em some} $|n| \leq T$ will be $\approx \frac{T}{T\log T} = \frac{1}{\log T}$. This further suggests that our arguments, although involved, are capturing the true behavior of our objects of study.

Theorem \ref{thm: main} solves the short interval story completely for Steinhaus random multiplicative functions, but not the Rademacher case. Most steps in our proof should transfer to the Rademacher case, either immediately or with quite obvious modifications, and we certainly expect the analogous theorem to hold (now with a real Gaussian limit, and ignoring any issues with counting square-free numbers when $y$ is very small). However, there is at least one place where some less trivial work would be needed. In the proof of Proposition~\ref{prop: concentration}, after applying a decorrelation estimate to the ``large'' prime contributions, we factor these out exploiting the fact that the distribution of $\mathbbm{1}_{\mathcal{H}(t)} \prod_{ z^{e^{-\tau}}<p\le z} |1-\frac{f(p)}{p^{\frac 12+it}}|^{-2}$ is exactly the same for all $t \in \R$ (because the joint distribution of the sequence $(f(p)p^{-it})_{p}$ is exactly the same for all shifts $t \in \R$). This ``translation invariance in law'' does not hold in the Rademacher case. It should be possible to address this (e.g. although translation invariance in law does not hold, it should ``almost'' hold provided $t$ isn't too small); we hope that an interested reader will take up this task.

\subsection{Organization and notations}
It will suffice to prove Theorem \ref{thm: main} in the range $\frac{x}{\sqrt{\log x}} \leq y = o(x)$, since smaller $y$ are already handled by the work of Soundararajan and Xu~\cite{SoundXu}.  Although not essential, restricting to this range for $y$ will streamline the writing of several of our arguments.

We write the conditional variance in a more useful asymptotic form, in terms of the random Euler product $F_z(s)$, in the fairly short Section~\ref{SEC: CONDVAR}. The concentration of conditional variance is established in Section~\ref{SEC: CONCEN} and the fourth moment computation together with other verifications of conditions in the martingale central limit theorem is done in Section~\ref{Sec: CLT}. Thus, Theorem~\ref{thm: main} is proved by combining results in Section~\ref{SEC: CONDVAR}, Section~\ref{SEC: CONCEN} and Section~\ref{Sec: CLT}.  Finally, we prove Proposition~\ref{prop: tailsprop} in Section~\ref{SEC: LONG SUM}.

We use standard notations of analytic number theory. Write $f(x)\ll g(x)$ and $f(x) =O(g(x))$ to denote that there exists a constant $C>0$ such that $|f(x)|\le Cg(x)$ for all $x$. Lastly  $f(x)\asymp g(x)$ means that $g(x)\ll f(x)\ll g(x)$.

\subsection*{Acknowledgements}
A.J.H. was supported in part by the Engineering and Physical Sciences Research Council of the United Kingdom [grant EP/V055755/1]. Some of the research for this paper was conducted when A.J.H. visited K.S. and M.W.X. in April 2023, and he would like to thank Stanford University for their hospitality during this visit. Some of the final writing up was done with support from the Simons Foundation and the Centre de Recherches Math\'ematiques, Montr\'eal, while A.J.H. was in residence as Aisenstadt Chair during the 2026 Universal Statistics in Number Theory thematic semester. K.S. is partially supported through a grant from the National Science Foundation; some of this work was also done while he was supported by a Simons Investigator grant from the Simons Foundation.  M.W.X. was supported by a Simons Junior Fellowship from the Simons Foundation. He appreciates the hospitality of Warwick Mathematics Institute, and he thanks Ye Tian for hosting him at the Morningside Center of Mathematics, where some of the writing up was done.  

We would like to thank Ofir Gorodetsky, Seth Hardy, Youness Lamzouri and Mo Dick Wong for their helpful comments on a draft of this paper. We especially thank Andrew Granville, whose comments and questions led us to determine an explicit form for the normalization factor $V(x,y)$.

For the purpose of open access, the authors have applied a Creative Commons Attribution (CC-BY) licence to any Author Accepted Manuscript version arising from this submission.

\section{Conditional variance}\label{SEC: CONDVAR}
Recall that $z = x^{\frac{1}{\log \log \log x}}$, and that we say a number $n$ is {\em $z$-smooth} if its largest prime factor $P(n)$ is $\leq z$. As is usual, we write $\Psi(x,z)$ to mean the count of $z$-smooth numbers $\leq x$ (although we will not need this notation very much).

Since the $f(n)$ are orthogonal, 
$$ 
\E \Big| \sum_{\substack{x \le n \le x+y \\ P(n)\le z   }  } f(n) \Big|^{2} = \sum_{\substack{x \le n \le x+y \\ P(n)\le z  }  } 1 = \Psi(x+y, z) - \Psi(x, z) + O(1) \ll \frac{y}{(\log\log x)^{100}} ,  
$$
where the last inequality follows upon using the sub-additivity $\Psi(x+y, z) - \Psi(x, z) \le \Psi(y, z)$ proved in \cite{Hildebrand85} (for $x,y \geq z$) and a classical estimate on $\Psi(y, z)$, see for example \cite[(1.12)]{Gran08}.   Since this is negligible compared with $V(x,y)$ in Theorem \ref{thm: main}, we see that we may discard the sub-sum over $z$-smooth numbers, and it will suffice to prove the theorem for $\sum_{\substack{x \le n \le x+y \\ P(n) > z  }  } f(n)$. Using the multiplicativity of $f$, this sum may be rewritten as
$$ 
\sum_{\substack{x \le n \le x+y \\ P(n) > z  }  } f(n) = \sum_{ \substack{ 1<m\le x+y \\ p|m \implies p>z}} f(m) \sum_{\substack{x/m\le n \le (x+y)/m \\ p|n \implies p\le z  }  } f(n) , 
$$
which we shall analyze by first conditioning on the behavior of $(f(p))_{p \leq z}$.

In this section, we establish an asymptotic expression for the conditional variance $V_f(x, y)$, which holds with asymptotic probability 1 (over realizations of the $(f(p))_{p \leq z}$). By orthogonality, note that the conditional variance satisfies 
\begin{equation}\label{eqn: Vf}
 V_f(x, y) =  \sum_{ \substack{ 1<m\le x+y \\ p|m \implies p>z}} \Big| \sum_{\substack{x/m\le n \le (x+y)/m \\ p|n \implies p\le z  }  } f(n) \Big|^{2}.    
\end{equation}

\begin{proposition}\label{prop: asmpy}
Let $y\ge x/(\log x)^{1/2}$ and $T=2/\log(1+y/x)$ be large, and put $\TT = T(\log T)^{100}$ . With probability $1-O( (\log T)^{-10} )$, we have
\begin{equation}\label{eqn: variance}
    V_f(x,y) = \frac{e^{-\gamma}}{4\pi} \frac{y}{\log z} \frac{1}{T} \int_{-\TT}^{\TT} \Big|F_z(\tfrac{1}{2}+it)\Big|^{2} K_T(t)^{2} dt +O\Big(\frac{y}{\sqrt{\log \log x}(\log T)^{50}}\Big),
\end{equation}
where $\gamma$ is Euler's constant, and
\begin{equation}\label{eqn: F}
  F_z(s):=  \sum_{\substack{n\ge 1 \\ p|n \implies p\le z}} \frac{f(n)}{n^{s}} = \prod_{p \leq z}\Big(1 - \frac{f(p)}{p^s} \Big)^{-1} , \quad K_T(t):= T \cdot \Big|  \frac{e^{\frac{1/2+it}{T}} - e^{-\frac{1/2 + i t }{T} }}{ 1/2+it}\Big|.  
\end{equation}
\end{proposition}

 The specific value $T=2/\log(1+y/x)$ simply arises from a change of variables in the course of the proof; note that $T=2x/y + O(1)$.   In section \ref{SEC: CONCEN}, we will prove that with probability $1-O((\log T)^{-\frac 15})$, 
\[
V_f(x, y) \gg  \frac{y}{\sqrt{\log \log x}}, 
\]
so that the remainder term in \eqref{eqn: variance} is negligible with high probability. 
\bigskip

\begin{proof}[Proof of Proposition~\ref{prop: asmpy}]
We first show that the contributions from small $m$ in \eqref{eqn: Vf} are negligible, analogously to our initial discarding of the $z$-smooth sub-sum (which would correspond to $m=1$). By small we mean terms with $1< m < (x+y)/x^{0.1}$, and the expected contribution of such terms to $V_f(x,y)$ is  
\begin{align*}
\E\Big[ \sum_{ \substack{1<m < (x+y)/x^{0.1} \\ p|m \implies p>z}} \Big| \sum_{\substack{x/m\le n \le (x+y)/m \\ p|n \implies p\le z  }  } f(n) \Big|^{2} \Big] &= \sum_{ \substack{1<m < (x+y)/{x^{0.1}} \\ p|m \implies p>z}} \sum_{\substack{x/m\le n \le (x+y)/m \\ p|n \implies p\le z  }  } 1 \\ 
&\ll \sum_{ \substack{1<m < (x+y)/{x^{0.1}} \\ p|m \implies p>z}} \frac{y}{m  (\log \log x)^{100}} 
\ll \frac{y \log\log\log x}{(\log \log x)^{100}} .
\end{align*}
By Markov's inequality it follows that with probability at least $1-O({(\log \log x)^{-10}})$, the contribution from $m < (x+y)/x^{0.1}$ is at most $O(y/(\log \log x)^{60})$.

We next focus on those $m\ge (x+y)/x^{0.1}$, and we group these terms in the ranges $(x+y)/(r+1) < m \le (x+y)/r$ where $r$ is an integer in the range $1\le r\le x^{0.1}$.  Thus the contribution of $m \ge (x+y)/x^{0.1}$ to $V_f(x,y)$ is 
\begin{equation*}
 \sum_{1\le r \le x^{0.1}}     \sum_{ \substack{\frac{x+y}{r+1} < m \le  \frac{x+y}{r} \\ p|m \implies p>z}} \Big| \sum_{\substack{x/m\le n \le (x+y)/m \\ p|n \implies p\le z  }  } f(n) \Big|^{2}. 
\end{equation*}
Given $r$, for all $(x+y)/(r+1) < m \le (x+y)/r$ the inner sum over $n$ above varies by at most $O(1)$.   Therefore the above equals 
\begin{equation}\label{eqn: short}
 \sum_{1\le r \le x^{0.1}}     \sum_{ \substack{\frac{x+y}{r+1} < m \le \frac{x+y}{r} \\ p|m \implies p>z}} \Big(\Big| \sum_{\substack{ \frac{x}{x+y} \cdot r
\le n \le r \\ p|n \implies p\le z  }  } f(n) \Big|^{2} +O\Big(1 +\Big| \sum_{\substack{ \frac{x}{x+y} \cdot r
\le n \le r \\ p|n \implies p\le z  }  } f(n) \Big|\Big) \Big).
\end{equation}
This style of argument also arises in (for example) the proof of Proposition 1 of Harper~\cite{Harperlargevalue}, and Lemma 1.2 of Gorodetsky and Wong~\cite{GW-Short}. 

The inner sum is now independent of $m$, so we can use the following standard sieve result to get rid of the $z$-roughness condition in the sum over $m$. 

\begin{lemma}[See Theorem 3 of Iwaniec~\cite{Iwaniec}]\label{lem: sieve}
    Let $\mathcal{M}$ be a set of $M$ integers such that for any positive integer $d$, the number of multiples of $d$ in $\mathcal{M}$ lies between $\frac{M}{d}-1$ and $\frac{M}{d} +1$. Let $s:= \frac{\log M}{\log z}$, and suppose $s < \frac{\log M}{(\log\log 3M)^6}$. Then 
    \[\#\{ n\in \mathcal{M}: p|n \implies p>z \} = M \prod_{p\le z} \Big(1-\frac{1}{p}\Big) \Big(1+ O\Big(\Big(\frac{e}{s\log s}\Big)^{s}\Big)\Big).   \]
\end{lemma}

Apply the lemma to $\mathcal{M} := \{\frac{x+y}{r+1} < m \le \frac{x+y}{r}\}$, so that $M\asymp x /r^{2} \gg x^{0.8}$ and $s \asymp \frac{\log x}{\log z} \asymp \log \log \log x$, and the error term gives a saving factor of order $O((\log \log x)^{-100})$, say. Thus the quantity in \eqref{eqn: short} is
\begin{align*}
 \sum_{1\le r \le x^{0.1}} \Big( \int_{\frac{x+y}{r+1}}^{\frac{x+y}{r}} dw \Big) \Big(\Big| \sum_{\substack{ \frac{x}{x+y} \cdot r
\le n \le r \\ p|n \implies p\le z  }  } f(n) \Big|^{2} &+O\Big(1 + \Big| \sum_{\substack{ \frac{x}{x+y} \cdot r
\le n \le r \\ p|n \implies p\le z  }  } f(n) \Big|\Big)  \Big)\\
& \times \prod_{p\le z} \Big(1-\frac{1}{p}\Big) \Big(1 + O((\log \log x)^{-100}) \Big), 
\end{align*}
which equals, upon using Mertens's theorem, 
\begin{equation}\label{eqn: Aaftersieve}
    \int_{(x+y)/x^{0.1}}^{x+y} \Big( \Big| \sum_{\substack{ \frac{x}{w} 
\le n \le \frac{x+y}{w} \\ p|n \implies p\le z  }  } f(n) \Big|^{2} + O\Big(1+\Big| \sum_{\substack{ \frac{x}{w} 
\le n \le \frac{x+y}{w} \\ p|n \implies p\le z  }  } f(n) \Big|\Big) \Big) \frac{e^{-\gamma}}{\log z}  \Big(1 + O((\log \log x)^{-100}) \Big)  dw.
\end{equation}

The contribution of all the error terms above is 
$$ 
\ll (\log z)^{-1} \int_{(x+y)/x^{0.1}}^{x+y} \Big( 1+ \Big| \sum_{\substack{ \frac{x}{w} 
\le n \le \frac{x+y}{w} \\ p|n \implies p\le z  }  } f(n) \Big| + \Big| \sum_{\substack{ \frac{x}{w} 
\le n \le \frac{x+y}{w} \\ p|n \implies p\le z  }  } f(n) \Big|^2 (\log \log x)^{-100} \Big) dw. 
$$ 
Take the expectation of this quantity, which is (using Cauchy--Schwarz) 
\begin{align*} 
&\ll (\log z)^{-1} \ \int_{(x+y)/x^{0.1}}^{x+y} \Big( 1 + \Big(\E\Big[  \Big| \sum_{\substack{ \frac{x}{w} 
\le n \le \frac{x+y}{w} \\ p|n \implies p\le z  }  } f(n) \Big|^2 \Big]\Big)^{\frac 12} + \Big( \frac{y}{w} +1\Big) (\log \log x)^{-100} \Big) dw 
\\ 
&\ll \frac{x}{\log z} + \frac{y \log x}{\log z} (\log \log x)^{-100} + \frac{\sqrt{yx}}{\log z} \ll \frac{y}{(\log \log x)^{99}}.
\end{align*} 
Here we used our assumption that $y\ge x/(\log x)^{\frac 12}$ to simplify the contribution.  By Markov's inequality, we conclude that with probability at least $1-O((\log \log x)^{-10})$, 
the contribution of these error terms is $O({y}/{(\log \log x)^{60}})$.

Our work so far shows that with probability $1- O((\log \log x)^{-10})$, 
\begin{equation*}
    V_f(x, y) =  \frac{e^{-\gamma}}{\log z} \int_{(x+y)/{x^{0.1}}}^{x+y}  \Big| \sum_{\substack{ \frac{x}{w} 
\le n \le \frac{x+y}{w} \\ p|n \implies p\le z  }  } f(n) \Big|^{2} dw  + O\Big( \frac{y}{(\log \log x)^{60}} \Big).
\end{equation*}
Writing $\delta =y/x$, and with the change of variables $u=x/w$, the relation above becomes 
\begin{equation}\label{eqn: 28}
   V_f(x, y) =  \frac{e^{-\gamma} x}{\log z} \int_{\frac{x}{x+y}}^{\frac{x^{1.1}}{x+y}}  \Big| \sum_{\substack{ u
\le n \le (1+\delta)u \\ p|n \implies p\le z  }  } f(n) \Big|^{2} \frac{du}{u^{2}} + O\Big( \frac{y}{(\log \log x)^{60}} \Big).
\end{equation}

We wish to extend the integral in \eqref{eqn: 28} to infinity.  The expected error induced in doing so is 
$$
\ll \frac{x}{\log z} \int_{\frac{x^{1.1}}{x+y}}^{\infty} \E\Big [   \Big| \sum_{\substack{ u
\le n \le (1+\delta)u \\ p|n \implies p\le z  }  } f(n) \Big|^{2} \Big]\frac{du}{u^{2}} = \frac{x}{\log z} \int_{\frac{x^{1.1}}{x+y}}^{\infty} \sum_{\substack{ u
\le n \le (1+\delta)u \\ p|n \implies p\le z  }  }  1 \frac{du}{u^2} \ll \frac{x}{\log z} \sum_{\substack{ n> \frac{x^{1.1}}{x+y} \\ p|n \implies p\le z}} \frac{\delta}{n}. 
$$
Using Mertens's theorems (and here crucially exploiting the $z$-smoothness condition) it follows that   
\begin{align*}
\sum_{\substack{ n> \frac {x^{1.1}}{x+y} \\ p|n \implies p\le z}} \frac{1}{n} &\ll (x^{0.1})^{-1000/\log z} \sum_{\substack{ n\ge 1 \\ p|n \implies p\le z}} \frac{1}{n^{1-1000/\log z}} 
\le (\log \log x)^{-100} \prod_{p\le z} \Big(1 -\frac{1}{p^{1-1000/\log z}} \Big)^{-1}\\
 &\ll (\log \log x)^{-100} \log z.   
\end{align*} 
 By Markov's inequality, we conclude that with probability at least $1-O((\log \log x)^{-10})$ the error induced by extending the integral in \eqref{eqn: 28} to infinity is $\ll y/(\log \log x)^{60}$.   Thus with probability at least $1-O((\log \log x)^{-10})$ we have 
\begin{equation}\label{eqn: main}
V_f(x, y) = \frac{e^{-\gamma}x}{\log z}  \int_{\frac{1}{1+\delta}}^{+\infty} \Big|\sum_{\substack{u \le n \le (1+\delta)u \\ p|n \implies p\le z}} f(n)\Big|^{2} \frac{du}{u^{2}} +O\Big(\frac{y}{(\log \log x)^{60}}\Big).   
\end{equation}

Recall that $\delta= y/x$ so that $T=2/\log (1+\delta)$, and make the change of variables $u=e^{v-1/T}$ so that $u(1+\delta) = e^{v+1/T}$.  
Then the integral in \eqref{eqn: main} becomes  
  $$
   \frac{e^{-\gamma }x}{\log z} 
 \cdot e^{1/T} \int_{-\infty}^{+\infty}  |h(v)|^{2} dv, 
 \ \ \text{where} \ \  
   h(v):= \sum_{\substack{e^{v-1/T} \le n \le e^{v+ 1/T} \\ p|n \implies p\le z}} f(n) e^{-v/2}.
  $$
   The Fourier transform of $h(v)$ is (it is convenient to omit the $2\pi$ in the normalization here)  
\begin{align*}
    \hat{h}(t) = \int_{-\infty}^{+\infty} h(v)e^{-i vt } dv &= \sum_{\substack{n\ge 1 \\ p|n\implies p \leq z}} f(n) \int_{\log n - \frac{1}{T}}^{\log n + \frac{1}{T}} e^{-i v t -v/2} dv\\ &= \sum_{\substack{n\ge 1\\ p|n\implies p \leq z}} \frac{f(n)}{n^{\frac{1}{2}+ i t}} \Big( \frac{e^{\frac{1/2+it}{T}} - e^{-\frac{1/2+ i t }{T} }}{1/2+it}\Big). 
\end{align*}
By the Plancherel formula the integral in \eqref{eqn: main} becomes  
\[  
\frac{e^{-\gamma} x}{\log z}  \frac{e^{1/T}}{2\pi T^2}  \int_{-\infty}^{+\infty} |F_z(\tfrac 12+it)|^2 K_T(t)^2 dt, 
\] 
with $F_z(s)$ and $K_T(t)$ as in \eqref{eqn: F}.   
Since $e^{1/T}/T = \delta/2 + O(1/T^2)$ and $K_T(t)^2 \ll \min (1, T^2/(1+t^2))$, we may write the above as 
\begin{equation} 
\label{2.8} 
\frac{e^{-\gamma} y}{\log z} \frac{1}{4\pi T} \int_{-\TT}^{\TT} |F_z(\tfrac 12+it)|^2 K_T(t)^2 dt + O\Big( \frac{y}{\log z} E \Big), 
\end{equation} 
where 
\begin{equation} 
\label{2.9} 
E = \frac{1}{T^2} \int_{|t|\le \TT} |F_z(\tfrac 12+it)|^2 dt + T \int_{|t| > \TT} |F_z(\tfrac 12+it)|^2 \frac{dt}{t^2}. 
\end{equation} 

We now estimate the expected value of $E^q$ for a suitable $0<q<1$, which will enable us to show that with high probability $E$ is suitably small.   Divide the integrals in $E$ into intervals of length $1$, from $n$ to $n+1$ for $n\in {\mathbb Z}$.   For non-negative real numbers $a_n$ and $0<q <1$, note that  $(\sum_n a_n)^q \le \sum_n a_n^q$.   Thus we find that $\E[E^q]$ is 
$$ 
\ll  \frac{1}{T^{2q}} \sum_{|n| \le \TT+1} \E \Big[ \Big( \int_{n}^{n+1} |F_z(\tfrac 12+it)|^2 dt \Big)^q \Big] 
+   \sum_{|n| > \TT} \frac{T^q}{|n|^{2q}}  \E \Big[ \Big( \int_{n}^{n+1} |F_z(\tfrac 12+it)|^2 dt \Big)^q \Big].
$$ 
Now using the translation invariance in law of the random Euler product $F_z(\frac{1}{2} +it)$, and a result of Harper~\cite[Section 4]{HarperLow}, we find that uniformly for $0<q<1$ 
$$ 
\E\Big[ \Big( \int_{n}^{n+1} |F_z(\tfrac 12+it)|^2 dt \Big)^q \Big]  = \E\Big[ \Big( \int_{0}^{1} |F_z(\tfrac 12+it)|^2 dt \Big)^q \Big] \ll  \Big(\frac{\log z}{1 + (1-q)\sqrt{\log \log z}}\Big)^{q}.   
$$
Taking $q=1-1/\log T$,  and noting $\TT =T(\log T)^{100}$,  it follows that 
 \begin{align*}
 \E[E^{1-1/\log T}] &\ll \Big( \frac{(\log T)^{100}}{T} + \frac{1}{(\log T)^{100}}\Big) \Big( \frac{\log z \log T}{\sqrt{\log \log z}}\Big)^{1-1/\log T} 
\\
& \ll \frac{1}{(\log T)^{99}} \Big(\frac{\log z}{\sqrt{\log \log z}}\Big)^{1-1/\log T}. 
\end{align*}
Notice it is crucial here to take $q$ close to 1 (to capture the decay of ${dt}/{t^2}$), but strictly less than 1 (to pick up the better than square-root cancellation factor $\sqrt{\log\log z}$ in the denominator). Markov's inequality now implies that with probability at least $1-O((\log T)^{-49})$ one has $E \ll (\log T)^{-50} (\log z)/\sqrt{\log \log z}$.   
 
Putting everything together,  with probability at least $1-O((\log T)^{-10})$, say (over all realizations of the $(f(p))_{p \leq z}$) we have
\begin{equation}\label{eqn: Vtrunc}
V_f(x, y) = \frac{e^{-\gamma}}{4\pi} \frac{y}{\log z} \frac{1}{T} \int_{-\TT}^{\TT} \Big|F_z(\tfrac{1}{2}+it)\Big|^{2}|K_T(t)|^{2} dt +O\big(\frac{y}{\sqrt{\log \log x}(\log T)^{50}}\big).
\end{equation} 
This completes the proof of Proposition~\ref{prop: asmpy}. 
\end{proof}

  \section{Concentration of the conditional variance}\label{SEC: CONCEN}
 In this section, we prove that with asymptotic probability 1 over realizations of the $(f(p))_{p \leq z}$, the conditional variance $V_f(x, y)$ is concentrated around a deterministic quantity $V(x,y)$ with the properties claimed in \eqref{eqn: V} and \eqref{eqn: hVexplicit}. In the next section, this will allow us to establish an unconditional central limit theorem for $\sum_{\substack{x \le n \le x+y \\ P(n) > z  }  } f(n)$ (and thus for $\sum_{x\le n \le x+y}f(n)$).

 \begin{theorem}[Concentration of conditional variance]\label{thm: concentration} Let $x$ be large, and let $y$ be in the range $x/(\log x)^{\frac 12}\le y\le x$.  Put 
 $z:= x^{\frac{1}{\log \log \log x}}$ and $T=2/\log (1+y/x)$.   Let $f$ be a Steinhaus random multiplicative function and let $V_f(x, y)$ be the conditional variance defined as in \eqref{eqn: Vf}.  Then there exists a deterministic quantity $\mu= \mu(x,y)$ such that with probability $1- O( (\log T)^{-\frac 15})$, we have
     \[
     V_f(x, y) = e^{-\gamma } \frac{y}{\log z} \mu \Big(1  + O\Big(\frac 1{(\log T)^{\frac 15}} \Big)\Big),
     \]
     where the quantity $\mu$ satisfies 
     \[
     \mu \asymp \log z \cdot  \min\Big\{1, \frac{\log T}{\sqrt{\log \log x}}\Big\}. 
     \]
  Further, as $T \rightarrow \infty$ we have the asymptotic 
  $$
  \mu \sim (e^{\gamma} \log z) \frac{1}{\sqrt{2\pi}} \int_{-\kappa}^{\kappa}  e^{-u^{2}/2} du, 
  \qquad \text{where} \qquad \kappa = \kappa(x,y):= \frac{\log T}{\sqrt{2\log \log z}}. 
  $$ 
  \end{theorem}

In view of Proposition \ref{prop: asmpy}, to prove Theorem \ref{thm: concentration} it will suffice to show that with probability $1- O((\log T)^{-\frac 15})$ we have (with $\TT=T(\log T)^{100}$ as before)
$$
\frac{1}{4\pi T} \int_{-\TT}^{\TT } |F_z(\tfrac{1}{2}+it)|^{2} K_T(t)^{2} dt = \mu \Big(1 + O\Big(\frac{1}{(\log T)^{\frac 15}}\Big)\Big),
$$
where $\mu$ has the properties claimed in the theorem.  The integral only depends on $x$ and $y$ via $z = x^{\frac{1}{\log \log \log x}}$ and $T=2/\log(1+y/x) = 2x/y + O(1)$. Likewise, $\mu$ will only depend on $x$, $y$ via $z$ and $T$, see Proposition \ref{prop: mu} below.

\subsection*{Overall strategy for establishing Theorem \ref{thm: concentration}}  Our proof will proceed in three steps: 
\begin{enumerate}
    \item We set up an initial barrier event $\mathcal{G}(t)$ holding with high probability for all $|t| \leq \TT$, such that with high probability  
$$
\int_{-\TT}^{\TT} |F_z(\tfrac{1}{2}+it)|^{2} K_T(t)^{2} dt = \int_{-\TT}^{\TT} |F_z(\tfrac{1}{2}+it)|^{2} K_T(t)^{2} \mathbbm{1}_{\mathcal{G}(t)} dt. 
$$ 
\item With high probability, we replace $\mathcal{G}(t)$ by a stronger (more restrictive) barrier event $\mathcal{H}(t)$ with little loss.  That is, with high probability one has 
$$
\int_{-\TT}^{\TT} |F_z(\tfrac{1}{2}+it)|^{2} K_T(t)^{2} \mathbbm{1}_{\mathcal{G}(t)} dt \approx \int_{-\TT}^{\TT} |F_z(\tfrac{1}{2}+it)|^{2} K_T(t)^{2} \mathbbm{1}_{\mathcal{H}(t)} dt.
$$
Unlike the event $\mathcal{G}(t)$ which holds with high probability for {\em all} $|t|\le T(\log T)^{100}$, the event $\mathcal{H}(t)$ is not required to hold for {\em all} $t$ with high probability.  Instead, the total contribution from points $t$ where it fails (but $\mathcal{G}(t)$ holds) is adequately small with high probability.
\item We establish the desired concentration estimate (in fact a much stronger one) for $\int_{-\TT}^{\TT } |F_z(\frac{1}{2}+it) |^{2} K_T(t)^{2} \mathbbm{1}_{\mathcal{H}(t)} dt$, via a variance calculation heavily relying on the presence of $\mathbbm{1}_{\mathcal{H}(t)}$. 
\end{enumerate}

\vspace{12pt}
Before carrying out this strategy, we first prove a much simpler concentration result (a variance type calculation without any barriers) that performs well when only rather small primes (compared with $T$) are involved. 
\begin{lemma}\label{lem: concentration}
    Let $f$ be a Steinhaus random multiplicative function. Let $T$ be large and $K_T(t)$ be defined as in \eqref{eqn: F}. Then for all $H\ge T$ and $w \ge 2$,  
    \[
    \E\Big[\Big|\int_{-H}^{H} |F_w(\tfrac{1}{2}+it)|^{2} K_T(t)^{2}dt - \sum_{\substack{n\ge 1\\  p|n \implies p\le w}} \frac{1}{n} \int_{-H}^{H} K_T(t)^{2}  dt  \Big|^{2}\Big]\ll T (\log w)^{4} , 
    \]
where $F_w(s):= \prod_{p \leq w} (1 - {f(p)}{p^{-s}} )^{-1}$.
\end{lemma}

Note that 
$$
\sum_{\substack{n\ge 1\\  p|n \implies p\le w}} \frac{1}{n} \int_{-H}^{H} K_T(t)^{2}  dt  = \prod_{p \leq w} \Big(1 - \frac{1}{p}\Big)^{-1} \int_{-H}^{H} K_T(t)^{2}  dt \asymp T\log w,
$$
 so the lemma provides genuine concentration if the right hand side is smaller than $(T\log w)^2$; that is,  if $\log w$ is smaller than $\sqrt{T}$.  This lemma will be useful in the proof of Proposition~\ref{prop: concentration}, for an appropriate  $w$ chosen in terms of $T$, to handle the contribution from the small primes while a different argument takes care of larger primes. 

\begin{proof}[Proof of Lemma~\ref{lem: concentration}]  Let ${\mathcal S}(w)$ denote the set of natural numbers all of whose prime factors are below $w$.   The left side of the lemma is 
$$ 
\E\Big[\Big|\int_{-H}^{H} \sum_{\substack{n,m \in {\mathcal S}(w) \\ n \neq m}} \frac{f(n) \overline{f(m)}}{n^{\frac 12+it} m^{\frac 12 - it}} K_T(t)^{2}dt \Big|^{2}\Big].
 $$
Expanding the square and using the orthogonality of the $f(n)$, this equals 
\begin{eqnarray*}
& &\int_{-H}^{H}\int_{-H}^{H} \sum_{\substack{n_1, n_2, m_1, m_2 \in {\mathcal S}(w) \\ n_1 n_2 = m_1 m_2 \\ n_1 \neq m_1}} \frac{1}{n_1^{\frac 12+it_1} m_1^{\frac 12 - it_1}} \frac{1}{n_2^{\frac 12+it_2} m_2^{\frac 12 - it_2}} K_T(t_1)^{2} K_T(t_2)^{2} dt_1 dt_2 \\
& = &  \sum_{\substack{n_1, n_2, m_1, m_2 \in {\mathcal S}(w) \\ n_1 n_2 = m_1 m_2 \\ n_1 \neq m_1}}  \frac{1}{m_1m_2} \int_{-H}^{H}\int_{-H}^{H} \Big(\frac{n_1}{m_1}\Big)^{i(t_2-t_1)} K_T(t_1)^{2} K_T(t_2)^{2} dt_1 dt_2 . 
\end{eqnarray*}

The solutions to $n_1n_2= m_1m_2$ may be parametrized by writing $g=(n_1,m_1)$ and setting $n_1= gr$, $m_1= gs$ where $r$ and $s$ are coprime, with $rs>1$.   It then follows that $n_2= sh$ and $m_2= rh$ for some integer $h$.  Thus our desired sum equals 
\begin{equation} 
\label{3.1} 
\sum_{g, h \in{\mathcal S}(w)} \sum_{\substack{ r, s\in {\mathcal S}(w) \\ (r,s)=1 \\ rs >1}} \frac{1}{ghrs} \Big|\int_{-H}^{H} \Big(\frac rs\Big)^{it} K_T(t)^2 dt \Big|^2 \ll 
(\log w)^2  \sum_{\substack{ r, s\in {\mathcal S}(w) \\ (r,s)=1 \\ rs >1}} \frac{1}{rs} \Big|\int_{-H}^{H} \Big(\frac rs\Big)^{it} K_T(t)^2 dt \Big|^2. 
\end{equation}

We now claim that for any $\xi \neq 0$ 
 \begin{equation}\label{eqn: kernel}
        \begin{split}
    \int_{-H}^{H} e^{it\xi}  K_T(t)^{2} dt 
     \ll \min \Big( T, \frac{1}{|\xi|} \Big) . 
        \end{split}
    \end{equation}
 Recall the definition of $K_T(t)$ given in \eqref{eqn: F},  from which it follows that $K_T(t)^2 \ll \min(1, T^2/(1+t^2))$.  This estimate readily implies the bound of $T$ given in \eqref{eqn: kernel}.    To obtain the other bound, we integrate by parts to see that the integral is 
 $$
 \frac{e^{it\xi}}{i\xi} K_T(t)^2 \Big|_{-H}^{H} - \frac{1}{i\xi} \int_{-H}^{H} e^{i\xi t}\frac{d}{dt} K_T(t)^2  dt \ll 
 \frac{1}{|\xi|} + \frac{1}{|\xi|} \int_{-H}^{H} \Big| \frac{d}{dt} K_T(t)^2 \Big| dt. 
 $$ 
A small calculation using the definition of $K_T(t)$ shows that the derivative of $K_T(t)^2$ is bounded in magnitude by $\ll (1+|t|)/T^2$ for $|t| \le T$ and by $\ll T/t^2$ for $|t| >T$.  The second bound claimed in \eqref{eqn: kernel} follows.  
    
 Using the bound \eqref{eqn: kernel} in \eqref{3.1} we see that our desired  variance is 
 $$ 
 \ll   (\log w)^2 \sum_{\substack{ r, s\in {\mathcal S}(w) \\ (r,s)=1 \\ rs >1}} \frac{1}{rs} \frac{T}{|\log (r/s)|} \ll T (\log w)^2 \sum_{r \in {\mathcal S}(w)} \sum_{s<r} \frac{1}{rs} \frac{1}{\log (r/s)}, 
 $$ 
 where in the last step we assumed by symmetry that $r$ is larger than $s$, and dropped the coprimality condition as well as the requirement that $s$ is $w$-smooth.    Using $\log (1/t) \ge (1-t)$ for $\tfrac 12 <t <1$, we find  
 $$ 
 \sum_{s<r} \frac{1}{s\log (r/s)} \ll \sum_{s\le r/2} \frac 1s + \sum_{r/2 < s <r} \frac{1}{(r-s)} \ll \log r.  
 $$
 Thus our variance is 
 $$ 
 \ll T(\log w)^2 \sum_{r\in {\mathcal S}(w)} \frac{\log r}{r} = T(\log w)^2 \sum_{r\in {\mathcal S}(w)} \frac{1}{r} \sum_{\ell |r} \Lambda(\ell) 
 = T (\log w)^2 \sum_{\ell \in {\mathcal S}(w)} \frac{\Lambda(\ell)}{\ell} \prod_{p\le w} \Big(1-\frac 1p\Big)^{-1}, 
 $$ 
which is $\ll T(\log w)^4$,  proving the proposition.  
 \end{proof}

\vspace{12pt}

We now begin work on executing the proof strategy described earlier.   Let $\tau$ denote the smallest natural number such that 
$$ 
z^{e^{-\tau}} \le \exp( \exp( (\log T)^{1/200})), 
$$ 
so that (by our choice of $z$ and as $T \ll \sqrt{\log x}$) 
$$
\tau = \log \log z - (\log T)^{\frac{1}{200}} +O(1) \sim \log \log z \sim \log \log x. 
$$ 
 Let $B$ be a fixed large constant, chosen sufficiently large to ensure that the probabilistic results invoked below hold.
 
 For each integer $0\le j\le \tau$, it will be convenient to define the partial random Euler product 
 \begin{equation} 
 \label{3.3} 
 {\mathcal F}_j(s) = \prod_{z^{e^{-\tau}} < p \le z^{e^{-j}}} \Big( 1- \frac{f(p)}{p^s}\Big)^{-1}; 
 \end{equation} 
 note that the final Euler product ${\mathcal F}_\tau$ is empty (and thus equals $1$). 
 Note that for any $t\in \R$ (and large $T$) 
 \begin{equation} 
 \label{3.4} 
 \E[ |{\mathcal F}_j(1/2+it)|^2 ] = \prod_{z^{e^{-\tau}} < p \le z^{e^{-j}}} \Big( 1- \frac{1}{p}\Big)^{-1} \sim e^{\tau-j}.
 \end{equation} 

For each $t \in \R$, let $\mathcal{G}(t)$ denote the event that for all $0\le j\le \tau$ 
\begin{equation}\label{eqn: g(t)}
     (\sqrt{T} e^{\tau-j})^{-B}  \le |{\mathcal F}_j(\tfrac 1 2 +it)|  \le \sqrt{T} e^{\tau-j}\exp((\log T)^{\frac 1{100}}) (\log \log z - j).   
\end{equation}
The stronger barrier event $\mathcal{H}(t)$ demands that the above holds and in addition (for all $0\le j\le \tau$), 
\begin{equation}\label{eqn: AJH h(t)}
|{\mathcal F}_j(\tfrac 12+it)| \le \sqrt{T} e^{\tau-j}  \frac{\exp(-(\log T)^{\frac 1{100}})}{ (\log \log z - j)^{5}}.  
\end{equation}
Note that when $j=\tau$ the Euler product ${\mathcal F}_j$ is empty, and the conditions \eqref{eqn: g(t)} and \eqref{eqn: AJH h(t)} hold automatically. 

The lower bound in the definition of $\mathcal{G}(t)$ is purely technical, the upper bound constraint is the important part. Then $\mathcal{H}(t)$ will give us a small, but crucial, extra saving in terms of both $T$ and $(\log\log z - j)$ (which, as the reader may check, cannot be incorporated directly into our proof that $\mathcal{G}(t)$ holds with high probability).

More specifically, the factor $(\log \log z - j)^{-5}$ ultimately leads to the uniform boundedness of an integral appearing in the proof of Proposition~\ref{prop: concentration}, below. The factor $\exp(-(\log T)^{\frac 1{100}})$ is needed both to overcome losses because our barrier conditions only start at $z^{e^{-\tau}}$, and in producing the final saving in Proposition~\ref{prop: concentration} (we note in passing that we could afford to save much less there and would still obtain Theorem \ref{thm: concentration}, the bounds flowing from Proposition~\ref{prop: Hfail} below are anyway weaker). There is quite a lot of flexibility in the choice of the factor $\exp(-(\log T)^{\frac 1{100}})$ and of the start point $z^{e^{-\tau}}$, but this must grow with $T$ at a certain rate so that the probabilistic results we shall invoke are valid (this ultimately corresponds to having good enough error terms for prime number sums of length $z^{e^{-\tau}}$).

\vspace{12pt}

We now establish our first step, showing that the barrier event ${\mathcal G}(t)$ holds for all $|t|\le \TT$ with high probability.   The proof combines second moment calculations together with the union bound and a discretization modeled after Harper~\cite{HarperLow} and (especially) Soundararajan and Zaman~\cite{SZ}.

\begin{proposition} \label{prop3.3}  Recalling that $\TT = T(\log T)^{100}$, we have  
\begin{equation}\label{eqn: exceptprob}
\P\Big( \mathcal{G}(t)~\text{holds for all $|t|\le \TT$} \Big)  \ge 1 -\exp (-(\log T)^{\frac 1{100}}). 
\end{equation}
\end{proposition} 
\begin{proof}   Given $0\le j\le \tau -1$  we shall show that the probability that \eqref{eqn: g(t)} fails for some $|t| \le \TT$ is  
$$ 
\ll \exp(-(\log T)^{\frac 1{100}}) (\log \log z -j)^{-2}.  
$$
Summing this over all the possibilities for $j$ yields the proposition.  

Consider a mesh of points ${\mathcal T}_j = \{ {\widehat t} = e^{j} n/\log z :\ \ n\in {\mathbb Z}, \ |{\widehat t}| \le \TT\}$.  The mesh ${\mathcal T}_j$ contains $\ll \TT e^{-j} \log z$ points, and for each $t$ with $|t|\le \TT$ we may find ${\widehat t} \in {\mathcal T}_j$ with $|t-{\widehat t}| \le e^{j}/\log z$. It will turn out that this places ${\widehat t}$ sufficiently close to $t$ (relative to the length of the Euler product ${\mathcal F}_{j}$) that the behaviour of ${\mathcal F}_j(\tfrac 12 +it)$ is essentially controlled by that of ${\mathcal F}_j(\tfrac 12 +i{\widehat t})$.

Indeed, if \eqref{eqn: g(t)} fails at $t$, then with ${\widehat t}$ denoting the nearest point to $t$ in ${\mathcal T}_j$ we must have one of the following four possibilities: 
\begin{equation} 
\label{3.8} 
|{\mathcal F}_j(\tfrac 12 +i{\widehat t})| \ge \tfrac 12 \sqrt{T} e^{\tau-j} \exp((\log T)^{\frac 1{100}}) (\log \log z - j), 
\end{equation} 
or 
\begin{equation} 
\label{3.9} |{\mathcal F}_j(\tfrac 12 + i{\widehat t})|^{-1} \ge \tfrac 12 (\sqrt{T} e^{\tau-j})^B, 
\end{equation} 
or 
\begin{equation} 
\label{3.10} 
\int_{-e^{j}/\log z}^{e^{j}/\log z} |{\mathcal F}_j^{\prime}(\tfrac 12+i{\widehat t} +ih)| dh 
\ge |{\mathcal F}_j(\tfrac 12+it)-{\mathcal F}_j(\tfrac 12+i{\widehat t})| \ge \tfrac 12 \sqrt{T} e^{\tau-j} \exp((\log T)^{\frac 1{100}}) (\log \log z - j), 
\end{equation}
or 
\begin{equation} 
\label{3.11} 
\int_{-e^{j}/\log z}^{e^{j}/\log z} |({\mathcal F}_j(\tfrac 12+i{\widehat t} +ih)^{-1})^{\prime}| dh 
\ge  |{\mathcal F}_j(\tfrac 12+it)^{-1} - {\mathcal F}_j(\tfrac 12 +i{\widehat t})^{-1} |  > \tfrac 12  (\sqrt{T}e^{\tau-j})^{B}. 
\end{equation} 

Given ${\widehat t}$, using Markov's inequality with \eqref{3.4}, we see that the probability that \eqref{3.8} holds is
\begin{align} 
\label{3.12} 
&\ll e^{\tau-j} (\tfrac 12 \sqrt{T} e^{\tau-j} \exp((\log T)^{\frac 1{100}}) (\log \log z - j))^{-2} \nonumber\\
& \ll T^{-1} e^{j-\tau} \exp(-2(\log T)^{\frac 1{100}}) (\log \log z - j)^{-2}.
\end{align} 
Since 
$$
\E[|{\mathcal F}_j(\tfrac 12+i {\widehat t})|^{-2}] = \prod_{z^{e^{-\tau}} < p \le z^{e^{-j}}} \Big(1 +\frac 1p \Big) \sim e^{\tau -j}, 
$$ 
an even stronger bound applies for the probability that \eqref{3.9} holds.   Next, note that 
$$ 
\E[ |{\mathcal F}_j^{\prime}(\tfrac 12 +i{\widehat t} +ih)|^2 ] = \E\Big[ \Big| \sum_{\substack{ n=1 \\ p|n \implies z^{e^{-\tau}}<p\le z^{e^{-j}}}} \frac{f(n) \log n}{n^{\frac 12+i {\widehat t}+ih}}\Big|^2\Big] = \sum_{\substack{ n=1 \\ p|n \implies z^{e^{-\tau}}<p\le z^{e^{-j}}}} \frac{(\log n)^2}{n},  
$$ 
and using $(\log n)/n^{\alpha} \le 1/(e\alpha)$ for all $\alpha>0$ and $n\ge 1$ we may bound this by 
\begin{align*} 
&\ll \Big( \frac{\log z}{e^j} \Big)^2 \sum_{\substack{ n=1 \\ p|n \implies z^{e^{-\tau}}<p\le z^{e^{-j}}}} \frac{1}{n^{1-2 e^{j}/\log z}} 
\\
&\le \Big( \frac{\log z}{e^j} \Big)^2  \prod_{z^{e^{-\tau}}<p\le z^{e^{-j}}} \Big(1- \frac 1{p^{1-2e^{j}/\log z}}\Big)^{-1} \ll \Big( \frac{\log z}{e^j} \Big)^2 e^{\tau-j}. 
\end{align*} 
Using these estimates and Cauchy--Schwarz we conclude that 
$$ 
\E \Big[ \Big| \int_{-e^{j}/\log z}^{e^{j}/\log z} |{\mathcal F}_j^{\prime}(\tfrac 12+i{\widehat t} +ih)| dh \Big|^2 \Big] 
\ll \frac{e^{j}}{\log z} \int_{-e^{j}/\log z}^{e^{j}/\log z}  \E[ |{\mathcal F}_j^{\prime}(\tfrac 12 +i{\widehat t} +ih)|^2 ]  dh \ll e^{\tau- j}. 
$$ 
Therefore by Markov's inequality, the probability that \eqref{3.10} holds is also bounded by the quantity in \eqref{3.12}.   An entirely analogous argument shows that the same estimate also holds for the probability with which \eqref{3.11} holds.   

Since there are $\ll {\widetilde T}e^{-j} \log z$ possible points ${\widehat t}$ in ${\mathcal T}_j$, we conclude that the probability that one of the four possibilities in \eqref{3.8}, 
\eqref{3.9},  \eqref{3.10}, or \eqref{3.11} holds for some ${\widehat t} \in {\mathcal T}_j$ is 
\begin{align*}
&\ll ({\widetilde T} e^{-j} \log z) T^{-1} e^{j-\tau} \exp(-2(\log T)^{\frac 1{100}}) (\log \log z - j)^{-2}\\
& \ll (\log T)^{100} (e^{-\tau} \log z) \exp(-2(\log T)^{\frac 1{100}}) (\log \log z - j)^{-2} \\
&\ll \exp(-(\log T)^{\frac 1{100}}) (\log \log z - j)^{-2}, 
\end{align*}
upon recalling the definition of $\tau$.   This bounds the probability that \eqref{eqn: g(t)} fails for this particular $j$ and some $|t| \le \TT$, and completes our proof. 
 \end{proof} 


Our next proposition will compute the expectation of $|F_z(\frac{1}{2} + it)|^{2}$ restricted to the barrier event $\mathcal{H}(t)$. Here we see the appearance and properties of the quantity $\mu$, which is related to the value to which our variance will concentrate. In order to obtain sufficiently strong information about $\mu$ to deduce \eqref{eqn: hVexplicit}, we will need the following Gaussian random walk lemma, which (although not hard to prove) is more precise than usually required in random multiplicative function problems.

\begin{lemma}\label{lemma: rwasymp}
Let $G_1 , ..., G_n$ be independent Gaussian random variables, each having mean zero and variance $0 < \E [G_m^2] \leq 20$ (say). Set $N:= \sum_{m=1}^{n} \E [G_m^2]$, and assume this is large. Then, for $a$ larger than a suitable absolute constant, 
$$ 
\P\Big(\sum_{m=1}^{j} G_m \leq a  \text{  for all } 1 \leq j \leq n\Big) = \Big(1 + O\Big( \frac{\log a}{a} \Big) \Big) \frac{1}{\sqrt{2\pi}} \int_{-\frac{a}{\sqrt{N}}}^{\frac{a}{\sqrt{N}}} e^{-u^{2}/2} du . 
$$
\end{lemma}

\begin{proof}  Consider a standard Brownian motion $W_t$.  If $W_t$ stays below $a$ for all continuous times $0\le t\le N$, then it certainly does so at the discrete time points $t_j := \sum_{m=1}^{j} \E [G_m^2]$ corresponding to $\sum_{m=1}^{j} G_m$.  Thus the probability desired in the lemma is bounded below by 
 $$
 \P(W_t \leq a  \text{  for all  } 0 \leq t \leq N) = \frac{1}{\sqrt{2\pi}} \int_{-\frac {a}{\sqrt{N}}}^{\frac{a}{\sqrt{N}}} e^{-u^{2}/2} du, 
 $$
upon using the known exact formula for this Brownian motion probability (see e.g. section 13.4 of Grimmett and Stirzaker~\cite{gs}).

Let $\epsilon>0$ be a small parameter that will be fixed shortly.   The desired probability can be upper bounded by 
$$
 \P(W_t \leq (1+\epsilon)a \text{ for  all }  0 \leq t \leq N) + \P(\sup_{0 \leq t \leq N} W_t \geq (1+\epsilon)a, \text{  but  }  W_{t_j} \leq a \text{ for all } 1 \leq j \leq n). 
 $$ 
As before, the first term above has an exact formula: 
\begin{align*}
\frac{1}{\sqrt{2\pi}} \int_{-(1+\epsilon)\frac{a}{\sqrt{N}}}^{(1+\epsilon)\frac{a}{\sqrt{N}}} e^{-u^{2}/2} du & =  \frac{1}{\sqrt{2\pi}} \int_{-\frac{a}{\sqrt{N}}}^{\frac{a}{\sqrt{N}}} e^{-u^{2}/2} du + O\Big(\epsilon \min\Big\{ \frac{a}{\sqrt{N}}, 1\Big\} \Big) \\
& = (1 + O(\epsilon) ) \frac{1}{\sqrt{2\pi}} \int_{-\frac{a}{\sqrt{N}}}^{\frac{a}{\sqrt{N}}} e^{-u^{2}/2} du .  
\end{align*}
We shall proceed to bound the second probability. Let $\tilde{t}$ denote the infimum of all $0 \leq t \leq N$ for which $W_t = (1+\epsilon)a$. If $\sup_{0 \leq t \leq N} W_t \geq (1+\epsilon)a$ then the (random) point $\tilde{t}$ is well defined, and we shall divide into cases according to whether $0 \leq \tilde{t} \leq N/2$ or $N/2 < \tilde{t} \leq N$. Also let $t^{*}$ be the (random) smallest point, {\em of the form} $t_j$, which is $\geq \tilde{t}$, and let $j^{*}$ be the index for which $t_{j^{*}} = t^{*}$.

Note that 
\begin{eqnarray}
&& \P(N/2 < \tilde{t} \leq N , \text{  but  }  W_{t_j} \leq a \text{ for all } 1 \leq j \leq n) \nonumber \\
& \leq & \P(W_{t^{*}} - W_{\tilde{t}} \leq -\epsilon a,  \;\; \text{and} \;\; N/2 < \tilde{t} \leq N) . \nonumber
\end{eqnarray}
Now $W_{t^{*}} - W_{\tilde{t}}$ is independent of $(W_{t})_{t \leq \tilde{t}}$, and is a mean zero Gaussian with variance $\leq 20$ (by definition of $t^{*}$ and our hypothesis about the variances of the $G_j$). Therefore the above probability is 
\begin{align*} 
&\ll e^{-(\epsilon a)^{2}/40} \P(N/2 < \tilde{t} \leq N) \leq e^{-(\epsilon a)^{2}/40} \P(W_t \leq (1+\epsilon)a \;\; \text{  for all } \; 0 \leq t \leq N/2)\\
& \ll e^{-(\epsilon a)^{2}/40} \min\Big\{1, \frac{a}{\sqrt{N}}\Big\} . 
\end{align*}

One can bound $\P(\tilde{t} \leq N/2 , \text{  but  }  W_{t_j} \leq a \text{ for all } 1 \leq j \leq n)$ in a similar way (and achieving the same bound). Let us split this probability further based on the size of $W_{t^*}$, as
$$ \sum_{k=0}^{\infty} \P(\tilde{t} \leq N/2 , \;\; -ka \leq W_{t^{*}} \leq (1-k)a , \text{  but  }  W_{t_j} \leq a \text{ for all } 1 \leq j \leq n) . $$
When the $k$-th event here occurs, we must in particular have $W_{t^{*}} - W_{\tilde{t}} \leq -(\epsilon + k)a$, and $W_{t_j} - W_{t^{*}} \leq a - W_{t^{*}} \leq (k+1)a$ for all $j^{*} < j \leq n$. Using these conditions, noting that $(W_{t_j} - W_{t^{*}})_{j^{*} < j \leq n}$ is independent of $(W_{t})_{t \leq t^{*}}$, and that $\sum_{j^{*} < m \leq n} \E [G_m^2 ] \asymp N$ when $\tilde{t} \leq N/2$ (given our assumptions that $\E [G_m^2 ] \leq 20$ and that $N$ is large), we find the sum is indeed
$$ \ll \sum_{k=0}^{\infty} e^{-((\epsilon + k) a)^{2}/40} \min\Big\{1, \frac{(k+1)a}{\sqrt{N}}\Big\} \ll e^{-(\epsilon a)^{2}/40} \min\Big\{1, \frac{a}{\sqrt{N}}\Big\} . $$
(Here the upper bound for the random walk probability that $W_{t_j} - W_{t^{*}} \leq (k+1)a$ for all $j^{*} < j \leq n$, which turns out to be the same as the corresponding upper bound if this were Brownian motion on the full interval, follows from e.g. Probability Result 1 of Harper~\cite{HarperLow}.)  

Taking $\epsilon = (\log a)/a$, say, then yields the claimed result.
\end{proof}

 \begin{proposition}\label{prop: mu}
     Let $f(n)$ be a Steinhaus random multiplicative function and $F_z(s)$ be the partial random Euler product over primes up to $z$. Let $T \le (\log z)^{10}$ be large. Let $\mathcal{H}(t)$ be defined as in \eqref{eqn: AJH h(t)} above. Set $\mu = \mu_{z, T}: = \E[|F_z(\tfrac{1}{2} + it)|^{2} \mathbbm{1}_{\mathcal{H}(t)}]$, which is independent of $t$.   
    Then    
    \begin{equation}\label{eqn: mu}
    \mu  \asymp (\log z) \cdot  \min \Big( 1, \frac{\log T}{\sqrt{\log \log z}} \Big).    
    \end{equation}
More precisely, as $T \rightarrow \infty$ we have 
$$
\mu  \sim \prod_{p \le z}\Big(1 - \frac{1}{p}\Big)^{-1} \cdot \frac{1}{\sqrt{2\pi}} \int_{-\kappa}^{\kappa} e^{-u^{2}/2} du, \qquad \text{ with }\qquad \kappa = \frac{\log T}{\sqrt{2\log \log z}}.
$$
  \end{proposition}

 \begin{proof}
The fact that the value of $\mu$ is independent of $t$ follows from translation invariance in law of the random Euler product (i.e. the fact that the law of the sequence $f(p)p^{-it}$ 
for primes $p$ is the same for all $t \in \R$). For simplicity of writing, in what follows we shall take $t=0$. 

The statement \eqref{eqn: mu} can be readily deduced from an approximate Girsanov type result of Harper~\cite[Lemma 4]{HarperLow}, followed by a suitable probabilistic estimate. First note that $\E[|F_z(\frac 12)|^{2}] = \prod_{p \le z}(1 - \frac{1}{p})^{-1} \asymp \log z$. Then we apply \cite[Lemma 4]{HarperLow}, with $t_j \equiv 0$ and $\sigma = 0$, to deduce that $\frac{\E[|F_z(\frac 12)|^2 \mathbbm{1}_{\mathcal{H}(0)}]}{\E[|F_z(\frac12)|^{2}]} \asymp $
\begin{equation}\label{eqn: prob}
   \P \Big( -\frac{B}{2}\log T - (B+1)j -O(1) \le \sum_{m=1}^{j}G_m \le \frac{\log T}{2} - (\log T)^{\frac{1}{100}} + h(j), \text{ for all } j\le \tau \Big), 
\end{equation}
where $h(j) = -5\log ((\log T)^{\frac 1{200}} + j) + O(1)$  and (we recall) $\tau \asymp \log \log z$, and $G_m$ are independent Gaussians with mean zero and variance $\frac 12 + o(1)$ (as $T \rightarrow \infty$). Note that the applicability of \cite[Lemma 4]{HarperLow} relies upon the fact that $\log(z^{e^{-\tau}})$ is large enough compared with $\log T$. The Probability Results 1, 2 in \cite{HarperLow} show that the probability in \eqref{eqn: prob} is $\asymp \min (1, \frac{\log T}{\sqrt{\tau}}) \asymp \min (1, \frac{\log T}{\sqrt{\log \log z}})$.

We next prove the final part of the proposition. 
 Since the smallest primes involved in the definition of $\mathcal{H}(0)$ are of size $z^{e^{-\tau}}$, which tends to infinity with $T$, the proof of \cite[Lemma 4]{HarperLow} actually implies that $\frac{\E[|F_z(\frac 12)|^2 \mathbbm{1}_{\mathcal{H}(0)}]}{\E[|F_z(\frac 12)|^{2}]} $ is asymptotic to the probability of the ballot event in \eqref{eqn: prob}. We estimate the probability. Again, since $T \rightarrow \infty$, the same calculations leading to Probability Results 1, 2 in \cite{HarperLow} (see the Appendix there) show not just that the lower bound $ -\frac{B}{2}\log T - (B+1)j -O(1)$ and the terms $- (\log T)^{\frac 1{100}} + h(j)$ in the upper bound do not alter the order of magnitude of the probability \eqref{eqn: prob}, but in fact that as $T \rightarrow \infty$ it is asymptotic to
$$
\P \Big( \sum_{m=1}^{j}G_m \le \frac{\log T}{2}, \text{  for all  }  j\le \tau \Big) . 
$$
Here we can apply Lemma \ref{lemma: rwasymp}, with $a = \frac 12\log T$ and $N = (\frac 12 + o(1))\tau = (\frac 12 + o(1))\log\log z$, so that $a/\sqrt{N} = (1+o(1))(\log T)/\sqrt{2\log\log z}$. This concludes the proof.
  \end{proof}

 We next show that the expected contribution of $|F_z(\tfrac 12+it)|^2$ when $\mathcal{G}(t)$ holds but $\mathcal{H}(t)$ fails is small when compared with $\mu$. In combination with Markov's inequality, this will later allow us to replace $\mathcal{G}(t)$ with $\mathcal{H}(t)$, as per Step 2 of the strategy outlined above.  Note that since ${\mathcal H}(t)$ imposes stronger restrictions than ${\mathcal G}(t)$, the expression ${\mathbbm 1}_{{\mathcal G}(t)}- {\mathbbm 1}_{{\mathcal H}(t)}$ is the indicator function of the event that ${\mathcal G}(t)$ holds but ${\mathcal H}(t)$ fails.  

 \begin{proposition}\label{prop: Hfail}
 Let the situation be the same as in Proposition~\ref{prop: mu}. Then 
 \[\E[(\mathbbm{1}_{\mathcal{G}(t)}-\mathbbm{1}_{\mathcal{H}(t)})|F_z(\tfrac 12+it)|^{2} ]  \ll \frac{\log z}{\sqrt{\log \log z}} (\log T)^{\frac 1{25}}. 
 \]
 \end{proposition}
\begin{proof}
The argument will broadly follow the proof of Multiplicative Chaos Result 2 of Harper~\cite{Harperlargevalue}, but with various changes to reflect the different sizes of barrier (involving $T$) that we are working with here, and the simplification that we have no need to insert a ``middle" barrier event $\mathcal{D}^{*}(t)$ (because our initial barrier $\mathcal{G}(t)$ already holds exactly at the point $t$, rather than at an approximating point).

Suppose $\mathcal{G}(t)$ holds but $\mathcal{H}(t)$ fails. Then there must exist some $0\le J \le \tau-1$ such that 
\begin{equation}\label{eqn: smallp1}
(  \sqrt{T}e^{\tau-j})^{-B} \le |{\mathcal F}_j(\tfrac 12+it)|  \le \sqrt{T} e^{\tau-j}  \frac{\exp(-(\log T)^{\frac 1{100}})}{(\log \log z - j)^{5}}  
\end{equation}
for all $J+1\le j \le \tau-1$, and 
\begin{equation}\label{eqn: smallp2}
   \begin{split}
       \sqrt{T} e^{\tau-J} \frac{\exp(-(\log T)^{\frac 1{100}})}{(\log \log z - J)^{5}}  <  |{\mathcal F}_J(\tfrac 12+it)| & \le \sqrt{T} e^{\tau-J} \exp((\log T)^{\frac 1{100}})  (\log \log z - J).
   \end{split}  
\end{equation}
We use ${\mathcal A}_J(t)$ to denote that all the inequalities \eqref{eqn: smallp1} and \eqref{eqn: smallp2} hold. 

   Combining \eqref{eqn: smallp2} with the definition of $\mathcal{G}(t)$ gives that for all $0\le  j \le J-1$, we have 
   \begin{align}\label{eqn: bigp}
 \Big(  \sqrt{T}e^{\tau-j}\Big)^{-B-1}  \frac{\exp(- (\log  T)^{\frac 1{100}}) }{\log \log z -J} 
 &\le \prod_{z^{e^{-J}} <p \le z^{e^{-j}} } \Big |1-\frac{f(p)}{p^{\frac 12+ it}}\Big|^{-1} \nonumber  \\
 & \le \exp(J-j+2 (\log T)^{\frac 1{100}}) (\log \log z - j)^{6}.
   \end{align}
 Let ${\mathcal B}_J(t)$ denote the event that all inequalities \eqref{eqn: bigp} hold. Note that ${\mathcal A}_J(t)$ only depends on $f(p)$ for primes $p\le z^{e^{-J}}$ while ${\mathcal B}_J(t)$ only involves $f(p)$ for primes $z^{e^{-J}} < p\le z$, so that ${\mathcal A}_J(t)$ and ${\mathcal B}_J(t)$ are independent. 
This leads to 
\begin{align*}
\E[(\mathbbm{1}_{\mathcal{G}(t)}-\mathbbm{1}_{\mathcal{H}(t)}) &|F_z(\tfrac 12+it)|^{2}] \le \sum_{J=0}^{\tau-1}\E[\mathbbm{1}_{\mathcal{A}_J(t)}\mathbbm{1}_{\mathcal{B}_J(t)}|F_z(\tfrac 12+it)|^{2} ]  \\ 
& = \sum_{J=0}^{\tau-1}\E\Big[\mathbbm{1}_{\mathcal{A}_J(t)} \prod_{p\le z^{e^{-J}}} \Big|1-\frac{f(p)}{p^{\frac 12+it}}\Big|^{-2}\Big] \E\Big[\mathbbm{1}_{\mathcal{B}_J(t)} \prod_{ z^{e^{-J}}<p \le z} \Big|1-\frac{f(p)}{p^{\frac 12+it}}\Big|^{-2}\Big].
\end{align*}
We next use the probability results in \cite{HarperLow}. In particular, we use \cite[Lemma 4]{HarperLow} to derive that 
\[\frac{\E[\mathbbm{1}_{\mathcal{B}_J(t)} \prod_{ z^{e^{-J}}<p \le z} |1-\frac{f(p)}{p^{\frac 12+it}}|^{-2}]}{\E[ \prod_{ z^{e^{-J}}<p \le z} |1-\frac{f(p)}{p^{\frac 12+it}}|^{-2}]} \ll \P \Big ( \sum_{m=1}^{j}G_m \le  2 (\log T)^{\frac 1{100}} + h(j),~~\forall \, 1\le j \le J  \Big),   \]
where $h(j) = 6\log (\log \log z -J+ j) +O(1)\le 6 \log (\log \log z - J) +6\log j +O(1)$
 and $G_m$ are independent Gaussian random variables with mean zero and variance $\frac 12+o(1)$ (as $T \rightarrow \infty$). We invoke \cite[Probability Result 1]{HarperLow} and see that this probability is 
  \begin{equation}\label{eqn: prob1}
 \ll  \frac{(\log T)^{\frac 1{100}}  + \log (\log \log z - J) }{1+\sqrt{J}} \ll  \frac{(\log T)^{\frac 1{100}} + \log (\tau - J) }{1+\sqrt{J}}  ,     
 \end{equation}
 where in the last step we used that $\log \log z - J = \tau - J + (\log T)^{\frac 1{200}} + O(1)$.  
 
 In the same way, we may bound 
\[
\frac{\E[\mathbbm{1}_{\mathcal{A}_J(t)} \prod_{ p\le z^{e^{-J}}} |1-\frac{f(p)}{p^{\frac 12+it}}|^{-2}]}{\E[ \prod_{  p\le z^{e^{-J}}} |1-\frac{f(p)}{p^{\frac 12+it}}|^{-2} ]}
\] 
by the probability that the following two events hold (with $G_m$ independent Gaussians with mean zero and variance $\tfrac 12+o(1)$ as above) 
\[ 
 \max_{j\le \tau-J} \sum_{m=1}^{j} G_m\le \tfrac{1}{2}\log T + (\log T)^{\frac 1{100}}  +  \log (\log \log z - J) +O(1),
 \] 
 and  
 \[ 
 \sum_{m=1}^{\tau-J} G_m \ge  \tfrac{1}{2}\log T - (\log T)^{\frac 1{100}}  -5 \log (\log \log z - J) -O(1).
 \]
 Applying the Ballot Theorem (see e.g. Probability Result 1 of Harper~\cite{harper2019partition} for a suitable version), this probability is 
\begin{align} \label{eqn: prob2}
& \ll \min\Big\{1, \frac{\log T + \log(\log \log z - J)}{\sqrt{\tau-J}}\Big\} \cdot \frac{((\log T)^{\frac 1{100}} + \log (\log \log z - J ) )^{2}}{\tau - J} \nonumber \\ 
 & \ll \min\Big\{1, \frac{\log T + \log(\tau - J)}{\sqrt{\tau-J}}\Big\} \cdot \frac{((\log T)^{\frac 1{100}} + \log (\tau - J ) )^{2}}{\tau - J} .   
\end{align}

 Combining \eqref{eqn: prob1} and \eqref{eqn: prob2} together, it follows that  $\frac{\E[(\mathbbm{1}_{\mathcal{G}(t)}-\mathbbm{1}_{\mathcal{H}(t)}) |F_z(\frac{1}{2}+it)|^{2}]}{\E[|F_z(\frac{1}{2}+it)|^{2}]} $ is
 \[
   \ll \sum_{J=0}^{\tau -1} \min\Big\{1, \frac{\log T + \log(\tau - J)}{\sqrt{\tau-J}}\Big\}  \cdot \frac{((\log T)^{\frac{1}{100}} + \log (\tau - J ) )^{2}}{\tau - J} \cdot \frac{(\log T)^{\frac{1}{100}} + \log (\tau - J) }{1+\sqrt{J}} .
 \]
The sum over $J$ can be bounded straightforwardly.   For those $J$ with $\tau-J\le (\log T)^{10}$, we upper bound the first factor  by $1$ and get a contribution that is 
\[ 
\ll (\log T)^{\frac{3}{100}}   \sum_{\tau- (\log T)^{10} \le J \le \tau-1 } \frac{1}{(\tau - J)(1+\sqrt{J})} \ll (\log T)^{\frac 3{100}} \frac{\log \log T}{\sqrt{\log \log z}} \ll \frac{(\log T)^{\frac 1{25}}}{\sqrt{\log \log z}},
\]
upon distinguishing whether $J\le \tau/2$ or not.   
For those $J$ with $\tau-J> (\log  T)^{10}$ (this can only occur if $\log T \leq \tau^{1/10} \leq (\log\log z)^{1/10}$), we can bound the contribution crudely by 
  \[    (\log T)^{1+\frac{3}{100}} \sum_{J\le \tau - (\log T)^{10}} \frac{(\log (\tau -J))^4}{(\tau -J)^{3/2}} \frac{1}{1 + \sqrt{J}}  \ll \frac{(\log T)^{1+\frac{3}{100}} (\log\log T)^4}{\sqrt{\tau} (\log T)^5} \ll \frac{1}{\sqrt{\log \log z}} .     \]
Recalling that $\E[|F_z(\frac{1}{2}+it)|^{2}] \asymp \log z$, we conclude that $\E[(\mathbbm{1}_{\mathcal{G}(t)}-\mathbbm{1}_{\mathcal{H}(t)}) |F_z(\frac{1}{2}+it)|^{2}]$ satisfies the claimed bound.
\end{proof}

The final (vital) ingredient in proving Theorem~\ref{thm: concentration} will be the following ``barrier-modified" variance estimate, crucially exploiting the presence of $\mathcal{H}(t)$ (to control the near diagonal contribution when everything is expanded out), whose proof we postpone for a moment.
\begin{proposition}\label{prop: concentration}
Let $f$ be a Steinhaus random multiplicative function. Let $F_z(s)$ be the random Euler product of $f$ over primes up to $z$. Let $T\le (\log z)^{10}$ be large and $K_T(t)$, $\mathcal{H}(t)$ and $\mu$ be as before. 
Then, with $\TT = T(\log T)^{100}$ as before, 
\begin{equation}\label{eqn: AJHdisp}
    \begin{split}
& \E\Big[ \Big|\int_{-\TT}^{\TT}|F_z(1/2+it)|^{2} K_T(t)^{2} \mathbbm{1}_{\mathcal{H}(t)} dt- \mu \int_{-\TT}^{\TT}K_T(t)^{2}  dt\Big|^{2}\Big] \ll 
 ( \mu T)^{2} \exp(-( \log T)^{\frac 1{100}}).       
    \end{split}
\end{equation}
\end{proposition}

From the definition of $K_T(t)$ (see \eqref{eqn: F}) note that 
$$ 
K_T(t)^2 = \frac{T^2 (e^{1/T} +e^{-1/T}-2)}{1/4+t^2} + 4 \frac{T^2 (\sin (t/T))^2}{1/4+t^2}, 
$$ 
from which it follows that 
\begin{align} 
\label{3.21} 
\int_{-\TT}^{\TT} K_T(t)^2 dt &= \int_{-\infty}^{\infty} K_T(t)^2 dt +O\Big(\int_{\TT}^{\infty} \frac{T^2}{t^2}dt\Big) = \int_{-\infty}^{\infty} \frac{4T^2 (\sin (t/T))^2}{1/4+t^2} dt +O\Big(\frac{T}{(\log T)^{100}}\Big) \nonumber \\
&= T \int_{-\infty}^{\infty} \frac{4 (\sin u)^2}{1/(4T^2) +u^2} du + O\Big(\frac{T}{(\log T)^{100}}\Big) = 4\pi T + O\Big(\frac{T}{(\log T)^{100}}\Big). 
\end{align} 
Thus the bound in Proposition \ref{prop: concentration} shows that the variance computed there is indeed small, compared with $( \mu \int_{-\TT}^{\TT}K_T(t)^{2}  dt )^2$.   We postpone the proof of this proposition, showing first how Theorem~\ref{thm: concentration} may be derived from it.   

\begin{proof}[Proof of Theorem~\ref{thm: concentration}, assuming Proposition~\ref{prop: concentration}]
Proposition \ref{prop3.3} shows that with probability at least $ 1- \exp(-(\log T)^{\frac 1{100}})$, 
\[
\int_{-\TT}^{\TT} |F_z(\tfrac 12+it)|^{2}K_T(t)^{2} dt = \int_{-\TT}^{\TT} |F_z(\tfrac 12+it)|^{2}K_T(t)^{2} \mathbbm{1}_{\mathcal{G}(t)}dt.  
\]
Now we use Proposition~\ref{prop: Hfail} together with Markov's inequality to get that with probability $1- O((\log T)^{-\frac 15})$, the right hand side is 
\[
\int_{-\TT}^{\TT} |F_z(\tfrac 12+it)|^{2}K_T(t)^{2}\mathbbm{1}_{\mathcal{H}(t)} dt +  O\Big(\int_{-\TT}^{\TT}  \log z \cdot \frac{(\log T)^{\frac 15+\frac{1}{25}}}{\sqrt{\log \log z}}  K_T(t)^{2} dt\Big). 
\]
The lower bound part of the estimate \eqref{eqn: mu} for $\mu$ shows that the error term above may be bounded by (distinguish the cases $\log T <\sqrt{\log \log z}$ and $\log \log z \gg \log T \ge \sqrt{\log \log z}$) 
\[ 
\ll \frac{1}{(\log T)^{\frac 15}} \int_{-\TT}^{\TT} \mu \cdot  K_T(t)^{2}  dt \ll \frac{T \mu}{(\log T)^{\frac 15}}.
\]
The proof is completed by assuming  Proposition~\ref{prop: concentration}, which together with \eqref{3.21} implies that with probability at least $1 - \exp(- \frac 13(\log T)^{\frac 1{100}})$ we have
\begin{align*}
\int_{-\TT}^{\TT} |F_z(1/2+it)|^{2}K_T(t)^{2}\mathbbm{1}_{\mathcal{H}(t)} dt &= \mu \int_{-\TT}^{\TT} K_T(t)^2 dt  + O\Big( T \mu\exp(-\tfrac 13 (\log T)^{\frac 1{100}}) \Big) 
\\
&= 4\pi \mu T + O\Big( \frac{T\mu}{(\log T)^{100}}\Big).
\end{align*}
\end{proof}

We now begin the proof of Proposition~\ref{prop: concentration}, which will take up the rest of this section.

\begin{proof}[Proof of Proposition~\ref{prop: concentration}] Expanding out the square and using the definition of $\mu$ (recall Proposition~\ref{prop: mu}), the left hand side of \eqref{eqn: AJHdisp} equals
\begin{equation}\label{eqn: double}
  \int_{-\TT}^{\TT}\int_{-\TT}^{\TT}  \E[|F_z(\tfrac 12+it)|^{2} |F_z(\tfrac 12 +iu)|^{2} \mathbbm{1}_{\mathcal{H}(t)}\mathbbm{1}_{\mathcal{H}(u)}]  K_{T}(t)^{2}K_{T}(u)^{2} dt du 
  - \Big( \mu \int_{-\TT}^{\TT} K_T(t)^2 dt \Big)^2.  
\end{equation}
We factor the expectation above as 
\[
\E[ |F_z(\tfrac 12+it)|^{2} |F_z(\tfrac 12+iu)|^{2} \mathbbm{1}_{\mathcal{H}(t)}\mathbbm{1}_{\mathcal{H}(u)}] = P_1(t, u) \cdot P_2(t, u), 
\]
where
\begin{equation}\label{eqn: P_1}
  P_1(t, u) := \E[ |F_{z^{e^{-\tau}}}(\tfrac 12+it)|^{2} |F_{z^{e^{-\tau}}}(\tfrac 12+iu)|^{2} ]  ,
\end{equation}
and (since $\mathcal{H}(t), \mathcal{H}(u)$ only depend on the $f(p)$ with $p > z^{e^{-\tau}}$)
\[
P_2(t, u): = \E[\mathbbm{1}_{\mathcal{H}(t)} |{\mathcal F}_0(\tfrac 12+it)|^2 \mathbbm{1}_{\mathcal{H}(u)} |{\mathcal F}_0(\tfrac 12+iu)|^2]. 
\]
Note that 
$$
\mu = \E[ |F_{z^{e^{-\tau}}}(\tfrac 12+it)|^{2}] \E[\mathbbm{1}_{\mathcal{H}(t)} |{\mathcal F}_0(\tfrac 12+it)|^2], 
$$
and so if it were the case that $P_1(t, u)$ and $P_2(t, u)$ each simply factored into the piece corresponding to $t$ and the piece corresponding to $u$, then \eqref{eqn: AJHdisp} would be identically zero.

The next lemma takes a step towards such a de-correlation when $|t-u|$ is large.  

\begin{lemma}\label{lem: slice} Suppose that $C\exp(-\tfrac 12 (\log T)^{\frac 1{200}}) \le |t-u|\le 2\TT$, for a suitable large constant $C$. Then 
\begin{equation}\label{eqn: p2}
   P_2(t,u) =  \Big(1 + O\Big(\frac{\exp(-\tfrac 12 (\log T)^{\frac 1{200}})}{|t-u| }\Big)\Big) \E[ \mathbbm{1}_{\mathcal{H}(t)} |{\mathcal F}_0(\tfrac 12+it)|^2 ]\E[ \mathbbm{1}_{\mathcal{H}(u)} |{\mathcal F}_0(\tfrac 12+iu)|^2].  
   \end{equation}
A similar but cruder decorrelation estimate for $P_1(t, u)$ is that 
\begin{equation}\label{eqn: p1}
    P_1(t, u)\ll \Big((\log (2+ |t-u|))^4 + \frac{1}{|t-u|^{2}}\Big) \E[|F_{z^{e^{-\tau}}}(\tfrac 12+it)|^{2} ] \E[|F_{z^{e^{-\tau}}}(\tfrac 12+iu)|^{2} ] .
\end{equation}
\end{lemma}
\begin{proof}[Sketch proof of Lemma~\ref{lem: slice}]
The proof of the first statement 
 follows from the second part of \cite[Lemma 7]{HarperLow}, and a slicing argument. We sketch the small modifications and changes needed, and verify that the conditions required there hold.  In the notation of \cite[Lemma 7]{HarperLow}, we take there $\sigma = 0$, and $t$ there is our $t-u$ (using translation invariance in law).  The parameters $x_j$ there correspond to $z^{e^{-\tau + j}}$, and the condition $z^{e^{-\tau + 1}} \geq e^{C/|t-u|^2}$ (needed at the end of the proof, to ensure that $|t-u|\sqrt{\log (z^{e^{-\tau}})}$ is large) is indeed satisfied under our hypotheses on $|t-u|$. In \cite[Lemma 7]{HarperLow}
 there is a condition $|t|\le 1$, which in our case says that we need $|t-u|\le 1$.  But this assumption is only needed for the cancellation in certain prime number sums (see \cite[the end of page 62, published version]{HarperLow}); namely,
\begin{equation}\label{eqn: primesum}
 \sum_{x_j^{1/e} < p \le x_j} \frac{\cos(|t-u| \log p)}{p} \ll \frac{1}{|t-u|\log x_j} ,  
\end{equation}
while in our case, since we are working with primes $p > z^{e^{-\tau}} = \exp(\exp((\log T)^{\frac 1{200}} +O(1) ))$ and $|t-u| \leq 2T (\log T)^{100}$, the relevant prime number sum estimate still holds by using a strong form of the prime number theorem (see e.g. \cite[Number Theory Result 2]{Harperlargevalue}).
The multiplier in the asymptotic formula of \cite[Lemma 7]{HarperLow} is stated as $1+O(1/\sqrt{C})$, but in fact this is $1+O(1/(|t|\sqrt{\log x_1}))$, which in our case is 
$1+O(\exp(-\frac 12 (\log T)^{\frac 1{200}})/{|t-u|})$. Finally, we remark that at the end of the proof there, the same calculations that replace the covariance by 0 also allow the means of $N_{j}^{1}, N_{j}^{2}$ to be replaced by $\sum_{x_j^{1/e} < p \le x_j} \frac{1}{p}$ (rather than $\sum_{x_j^{1/e} < p \le x_j} \frac{1 + \cos(|t-u| \log p)}{p}$), completely decoupling them and removing all dependence on $t,u$.

For the second part, we apply e.g. \cite[display (6)]{HarperLow} and conclude that 
\[
P_1(t, u)\ll \exp\Big(\sum_{p\le z^{e^{-\tau}}}\frac{2\cos((t-u)\log p)}{p}\Big) \E[ |F_{z^{e^{-\tau}}}(1/2+it)|^{2} ] \E[ |F_{z^{e^{-\tau}}}(1/2+iu)|^{2}  ].
 \]
The prime number sum above is estimated by splitting into small primes and big primes. Let us call the threshold $p_0$.
In the case that $|t-u|\le 1$, choose $p_0 = e^{1/|t-u|}$. When $p<p_0$, we bound $\cos((t-u)\log p)$ trivially by $1$ and use Mertens' estimate to get the bound 
$\sum_{p\le e^{1/|t-u|}}\frac{2\cos((t-u)\log p)}{p} \leq 2\log \frac{1}{|t-u|} + O(1)$. When $p>p_0$, we use the prime number theorem (to get a bound like \eqref{eqn: primesum}) and conclude that the contribution is $\ll 1$.  For the case that $|t-u|>1$, we choose $p_0 = e^{A(\log |t-u|)^{2}}$ for a suitable large constant $A$. When $p\le p_0$, we again bound $\cos((t-u) \log p)$ trivially by 1 which leads to the bound $4\log \log |t-u| + O(1)$ for the prime number sum. When $p>p_0$, we apply the prime number theorem with classical error term $x\exp(-d\sqrt{\log x})$ (for some $d>0$) and derive that the contribution in this case is $\ll 1$ (again see \cite[Number Theory Result 2]{Harperlargevalue} for details). This concludes the proof.
\end{proof}

Split the double integral in \eqref{eqn: double} as $I_1 + I_2$, where $I_1$ is the integral over pairs $t$, $u$ with $|t-u| \ge C\exp(-\tfrac 12 (\log T)^{\frac 1{200}})$ (which is the typical case), and $I_2$ is the integral over nearby points $t$, $u$ with $|t-u| \le C\exp(-\tfrac 12 (\log T)^{\frac 1{200}})$.   We now use 
 the decorrelation estimate from Lemma~\ref{lem: slice} to evaluate $I_1$, showing that it cancels out the term $(\mu \int_{-\TT}^{\TT} K_T(t)^2 dt)^2$ appearing in 
 \eqref{eqn: double} up to a negligible error term.  Following that, we shall show that the contribution of $I_2$ is negligible, completing the proof of Proposition \ref{prop: concentration}.

 We begin by using \eqref{eqn: p2}, and first consider the effect of the error term there.  Using the crude bounds $K_T(t)^2$, $K_T(u)^2\ll 1$ together with \eqref{eqn: p1}, and writing $h = |t-u|$ we see that this contribution to $I_1$ is 
 \begin{align*} 
& \ll \mu^2 \TT \exp(-\tfrac 12 (\log T)^{\frac{1}{200}}) \int_{C\exp(-\frac 12(\log T)^{\frac 1{200}})}^{2\TT} \Big( \frac{(\log (2+h))^4}{h} +\frac{1}{h^3}\Big) dh 
 \\
 &\ll \mu^2 T (\log T)^{100}  \exp(\tfrac 12 (\log T)^{\frac{1}{200}}). 
 \end{align*}
 
 Now consider the main term for $P_2(t,u)$ from \eqref{eqn: p2}, which is (importantly using translation invariance, as noted in the Introduction) 
 $$ 
 \E[\mathbbm{1}_{\mathcal{H}(t)} |{\mathcal F}_0(\tfrac 12+it)|^2 ]\E[ \mathbbm{1}_{\mathcal{H}(u)} |{\mathcal F}_0(\tfrac 12+iu)|^2] = 
 \Big( \E[ \mathbbm{1}_{\mathcal{H}(0)} |{\mathcal F}_0(\tfrac 12)|^2 ]\Big)^2. 
 $$ 
  Its contribution to $I_1$ is 
  $$ 
   \Big( \E[ \mathbbm{1}_{\mathcal{H}(0)} |{\mathcal F}_0(\tfrac 12)|^2 ]\Big)^2 \int_{\substack{|t|, |u| \le \TT \\ |t-u|  \ge C\exp(-\frac 12 (\log T)^{\frac 1{200}})}}  P_1(t,u) K_T(t)^2 K_T(u)^2 dt du. 
   $$ 
 We now show that the condition on $|t-u|$ above may be dropped with negligible error.   First note that by a simple fourth moment computation (see e.g. \cite[Euler Product Result 1]{Harperlargevalue}, with $k=1$ and $\alpha_1 = 2$ and $\sigma = 0$):
\begin{equation}\label{eqn: small}
P_1(t, u) \leq \E[|F_{z^{e^{-\tau}}}(1/2)|^{4}] \asymp (\log z^{e^{-\tau}})^{4} \asymp \exp( 4 (\log T)^{\frac 1{200}}).   
\end{equation}
Therefore, recalling the definition of $\mu$ and using $K_T(u)^2 \ll 1$, the error induced by dropping the condition on $|t-u|$ is 
$$ 
\ll \mu^2 \exp(4(\log T)^{\frac 1{200}}) \int_{-\TT}^{\TT} K_T(t)^2 dt \ll \mu^2 T \exp(4(\log T)^{\frac 1{200}}). 
$$ 
Thus the contribution of $I_1$ is
\begin{equation} 
\label{3.28} 
 \Big( \E[ \mathbbm{1}_{\mathcal{H}(0)} |{\mathcal F}_0(\tfrac 12)|^2 ]\Big)^2 \int_{-\TT}^{\TT}\int_{-\TT}^{\TT} P_1(t,u) K_T(t)^2 K_T(u)^2 dt du +O(\mu^2 T\exp(4(\log T)^{\frac{1}{200}})). 
 \end{equation} 
 
 At this stage we invoke  Lemma~\ref{lem: concentration}; taking there $w=z^{e^{-\tau}}$ and $H=\TT$, we obtain
\[ 
\E\Big[\Big| \int_{-\TT}^{\TT}  |F_{z^{e^{-\tau}}}(1/2+it)|^{2}  K_T(t)^{2}dt  - \int_{-\TT}^{\TT} \sum_{\substack{n \geq 1 \\ p|n \implies p\le  z^{e^{-\tau}}}} \frac{K_T(t)^{2}}{n}  dt      \Big|^{2}\Big]
\ll T\exp(4 (\log T)^{\frac 1{200}}). 
\]
 Expanding out the left side (which is a variance) and rearranging, it follows that 
 $$
 \int_{-\TT}^{\TT}\int_{-\TT}^{\TT} P_1(t,u) K_T(t)^2 K_T(u)^2 dt du\! = \!\Big( \int_{-\TT}^{\TT} \E[|F_{z^{e^{-\tau}}}(\tfrac 12)|^2] K_T(t)^2 dt \Big)^2 + O(T\exp(4 (\log T)^{\frac 1{200}})). 
 $$ 
  Combining this with \eqref{3.28}, we conclude that 
  \begin{equation} 
  \label{3.29} 
  I_1 = \Big( \mu \int_{-\TT}^{\TT} K_T(t)^2 dt \Big)^2 + O(\mu^2 T \exp(4 (\log T)^{\frac 1{200}})).
  \end{equation}

\vspace{12pt}

We now turn to the problem of estimating $I_2$, where $t$ and $u$ are close together, and it is here that the barrier events $\mathcal{H}(t)$ and ${\mathcal H}(u)$  play a crucial role.  Our task is to bound
\begin{equation*}
  \int \int _{\substack{ |t|, |u| \le \TT \\ |t-u|\le C \exp(-\frac 12(\log T)^{\frac 1{200} })  } }   \E [\mathbbm{1}_{\mathcal{H}(t)} |F_z(\tfrac 12+it)|^{2} \mathbbm{1}_{\mathcal{H}(u)} |F_z(\tfrac 12+iu)|^{2}] K_T(t)^{2}dt K_T(u)^{2}du.        
\end{equation*}
Write $h=t-u$, and use translation invariance to bound the above by 
\begin{align*} 
&\int_{|h| \le C \exp(-\frac 12 (\log T)^{\frac 1{200} })     } \E[ \mathbbm{1}_{\mathcal{H}(0)} |F_z(\tfrac 12)|^{2} \mathbbm{1}_{\mathcal{H}(h)} |F_z(\tfrac 12+ih)|^{2}] dh \int_{-\TT}^{\TT} K_T(u)^2 K_T(u+h)^2 du \nonumber \\ 
&\ll T \int_{|h| \le C \exp(-\frac 12 (\log T)^{\frac 1{200} })     } \E[ \mathbbm{1}_{\mathcal{H}(0)} |F_z(\tfrac 12)|^{2} \mathbbm{1}_{\mathcal{H}(h)} |F_z(\tfrac 12+ih)|^{2}] dh.
\end{align*} 
Using the fourth moment estimate \eqref{eqn: small} to handle the small primes $p\le z^{e^{-\tau}}$, we obtain the further reduction 
\begin{equation} 
\label{3.30} 
I_2 \ll T \exp(4(\log T)^{\frac 1{200}})  \int_{|h| \le C \exp(-\frac 12 (\log T)^{\frac 1{200} })    } \E[ \mathbbm{1}_{\mathcal{H}(0)} |{\mathcal F}_0(\tfrac 12)|^2 \mathbbm{1}_{\mathcal{H}(h)} |{\mathcal F}_0(\tfrac 12+ih)|^2 ] dh. 
\end{equation}

We now focus on bounding the expectation in \eqref{3.30} for a given value of $h$.   Define $J$ to be the smallest non-negative integer with $z^{e^{-J}} \le \exp((C/|h|)^2)$, so that $J\le \tau$ in our range for $h$.  Define $J_0$ to be the smallest non-negative integer at most $\tau$ with $z^{e^{-J_0}} \le \exp(C/|h|)$, setting $J_0=\tau$ if no such integer exists.   Thus $0\le J \le J_0 \le \tau$.  We factor the Euler product ${\mathcal F}_0(\tfrac 12+it)$ into three pieces depending on whether $z^{e^{-\tau}} < p \le z^{e^{-J_0}}$, or $z^{e^{-J_0}}< p \le z^{e^{-J}}$, or $z^{e^{-J}} < p \le z$, and call the three corresponding Euler products ${\mathcal P}_1(\tfrac 12+it)$, ${\mathcal P}_2(\tfrac 12+it)$, and ${\mathcal P}_3(\tfrac 12+it)$.   

Next we examine the conditions ${\mathcal H}(0)$ and ${\mathcal H}(h)$, extracting from them constraints on the primes in these three ranges.  Recall that the conditions ${\mathcal H}(0)$ and ${\mathcal H}(h)$ are given by the upper barrier constraints in \eqref{eqn: AJH h(t)}, together with the lower barrier constraint in \eqref{eqn: g(t)}. The idea is that the primes $z^{e^{-J_0}} < p$ are large enough that their contributions to the products at 0 and at $h$ behave quite independently. So we will invoke ${\mathcal H}(0)$ to bound ${\mathcal P}_1(\tfrac 12)$, then all our remaining Euler product factors at 0 and at $h$ will be essentially uncorrelated. When $z^{e^{-J}} < p$ is even larger, we will also be able to show that the barrier conditions at 0 and at $h$ provide quite independent constraints, giving some further saving.

More precisely, on the lower range $z^{e^{-\tau}} < p \le z^{e^{-J_0}}$ we keep only the condition (arising from \eqref{eqn: AJH h(t)} with $j=J_0$ and $t=0$) 
\begin{equation} 
\label{3.31} 
|{\mathcal F}_{J_0}(\tfrac 12)| = |{\mathcal P}_1(\tfrac 12)| \le \sqrt{T} e^{\tau -J_0} \frac{\exp(-(\log T)^{\frac 1{100}})}{(\log \log z -J_0)^5}. 
\end{equation} 
In the upper range $z^{e^{-J}}< p \le z$, using the upper bound in \eqref{eqn: AJH h(t)} together with the lower bound in \eqref{eqn: g(t)}, we find that for all $0\le j\le J-1$  and with $t$ being $0$ or $h$ 
\begin{equation} 
 \label{3.32} 
(\sqrt{T} e^{\tau-j})^{-(B+1)} \le  \prod_{z^{e^{-J}}< p \le z^{e^{-j}}} \Big|1- \frac{f(p)}{p^{\frac 12+it}}\Big|^{-1}\le T^{(B+1)/2} \exp((B+1)(\tau-J) + (J-j)). 
\end{equation}
In the middle range $z^{e^{-J_0}}<p\le z^{e^{-J}}$ we ignore the constraints imposed by ${\mathcal H}(0)$ and ${\mathcal H}(h)$.  

We now bound the expected value of $|{\mathcal F}_0(\tfrac 12) {\mathcal F}_0(\tfrac 12+ih)|^2$ keeping only the constraints given by \eqref{3.31} and \eqref{3.32}.   Since these constraints are independent over the three ranges for primes, we may factor the expectation correspondingly into three parts.  For the contribution of $|{\mathcal P}_1(\tfrac 12){\mathcal P}_1(\tfrac 12+ih)|^2$, we use \eqref{3.31} to obtain that this expectation is 
\begin{align}
\label{3.33} 
&\ll T e^{2(\tau-J_0)} \frac{\exp(-2(\log T)^{\frac 1{100}})}{(\log \log z-J_0)^{10}} \E[|{\mathcal P}_1(\tfrac 12+ih)|^2] \nonumber \\
&=    T e^{2(\tau-J_0)} \frac{\exp(-2(\log T)^{\frac 1{100}})}{(\log \log z-J_0)^{10}} \prod_{z^{e^{-\tau}}<p \le z^{e^{-J_0}}} \Big(1-\frac 1p\Big)^{-1}  \nonumber\\
&\ll T e^{3(\tau-J_0)}  \frac{\exp(-2(\log T)^{\frac 1{100}})}{(\log \log z-J_0)^{10}}.
\end{align} 
We see that a multiplier $T$ has emerged here, but (thanks to our strengthened barrier ${\mathcal H}(0)$) accompanied by the saving factor $\frac{\exp(-2(\log T)^{\frac 1{100}})}{(\log \log z-J_0)^{10}}$.

For the middle range, using Lemma 6 from \cite{HarperLow} we find 
\begin{equation} 
\label{3.34} 
\E[|{\mathcal P}_2(\tfrac 12) {\mathcal P}_2(\tfrac 12+ih)|^2] \ll 
\exp\Big( \sum_{z^{e^{-J_0}} < p \le z^{e^{-J}}} \frac{2+2\cos(h\log p)}{p} \Big) \ll \exp( 2 (J_0-J)), 
\end{equation} 
upon applying the prime number theorem.  Finally for the large primes the expected value of $|{\mathcal P}_3(\tfrac 12){\mathcal P}_3(\tfrac 12 +ih)|^2$ subject to the constraints in \eqref{3.32} is (upon using Proposition 7 from \cite{HarperLow}) 
\begin{equation} 
\label{3.35} 
\ll \prod_{z^{e^{-J}}<p\le z} \Big(1-\frac 1p\Big)^{-2} \Big(\frac{(\log T + (\tau - J))}{1+\sqrt{J}}\Big)^2 
\ll e^{2J} \Big( \frac{(\log T + (\tau-J))^2}{1+J} \Big). 
\end{equation} 
Putting \eqref{3.33}, \eqref{3.34} and \eqref{3.35} together we conclude that 
\begin{equation} 
\label{3.36} 
\E[\mathbbm{1}_{\mathcal{H}(0)}\mathbbm{1}_{\mathcal{H}(h)}|{\mathcal F}_0(\tfrac 12){\mathcal F}_0(\tfrac 12+ih)|^2] 
\ll T e^{3\tau - J_0}  \frac{\exp(-2(\log T)^{\frac 1{100}})}{(\log \log z-J_0)^{10}} \Big( \frac{(\log T + (\tau-J))^2}{1+J} \Big).
\end{equation} 

We simplify the bound in \eqref{3.36} by distinguishing the cases when $|h| \le (\log z)^{-\frac 14}$ and when $(\log z)^{-\frac 14} \le |h| \le C\exp(-\frac 12(\log T)^{\frac{1}{200}})$.  In the first case note that $\log \log z -J_0 \gg \log \log z$, and we may replace the bound in \eqref{3.36} by 
\begin{align*}
&\ll T (\log z)^3 e^{-J_0} \frac{\exp(-2(\log T)^{\frac 1{100}})}{(\log \log z)^{10}} (\log T)^2 (\log \log z)^2 
\\
&\ll T (\log z)^3 \frac{\exp(-\frac 32(\log T)^{\frac 1{100}})}{(\log \log z)^{8}} \min \Big( 1, \frac{1}{|h| \log z}\Big).
\end{align*}
 In the second case, note that $\log \log z -J_0$ and $\log \log z -J$ are both $\asymp \log (1/|h|)$, and $J \gg \log \log z$.  This enables us to simplify the bound in \eqref{3.36} to 
 $$ 
 \ll T (\log z)^3 e^{-J_0} \frac{ \exp(-2(\log T)^{\frac 1{100}})}{(\log (1/|h|))^{10}}\frac{ (\log T)^2 (\log 1/|h|)^2 }{\log \log z} 
 \ll T \frac{(\log z)^2}{|h| \log \log z}  \frac{\exp(-\frac 32(\log T)^{\frac 1{100}})}{(\log 1/|h|)^{8}}.  
 $$ 
 To see the last estimate above, split into the cases $J_0 <\tau$ where $e^{-J_0} \asymp 1/(|h|\log z)$ and the case $J_0=\tau$ where $|h| \ge C\exp(-(\log T)^{\frac 1{200}})$.  
 Using these two bounds in \eqref{3.30}, we conclude that 
 \begin{align*}
 I_2\ll T^2 &\exp(-(\log T)^{\frac 1{100}}) \Big( \frac{(\log z)^3}{(\log \log z)^8}  \int_{|h|\le (\log z)^{-\frac 14}} \min\Big(1, \frac{1}{|h|\log z}\Big) dh \\
 &+   \frac{(\log z)^2}{\log \log z} \int_{(\log z)^{-\frac 14} < |h| \le C\exp(-\frac 12(\log T)^{\frac{1}{200}})} \frac{1}{|h| (\log (1/|h|))^8} dh \Big) \\
 &\ll T^2 \frac{(\log z)^2}{\log \log z} \exp(-(\log T)^{\frac{1}{100}}) \ll T^2 \mu^2  \exp(-(\log T)^{\frac{1}{100}}). 
 \end{align*}
Combining this with \eqref{3.29}, we have completed the proof of Proposition~\ref{prop: concentration}.  
\end{proof} 
 
 \section{Conditional Gaussian limiting distribution}\label{Sec: CLT}
 In this section, we show a conditional central limit theorem (conditioning on all the values $(f(p))_{p\le z}$) for the sum $\sum_{\substack{x\le n \le x+y \\ P(n)>z}  } f(n)$. More precisely, we will show that with high probability over all realizations of the $(f(p))_{p\le z}$, the conditional characteristic function (and therefore the conditional distribution) of $\sum_{\substack{x\le n \le x+y \\ P(n)>z}  } f(n)$ is close to complex Gaussian.
 
 Recall the decomposition 
 \[
 \sum_{\substack{x\le n \le x+y \\ P(n)>z}  } f(n)  =   \sum_{\substack{1<m\le x+y\\ p|m \implies p>z}} f(m)    \sum_{\substack{\frac{x}{m} \le n \le \frac{x+y}{m} \\ P(n)\le z}} f(n).   
  \]
 By using the orthogonality of $f(n)$, the conditional variance is 
 \begin{equation*}
     V_f(x, y) = \sum_{\substack{1<m\le x+y\\ p|m \implies p>z}} \Big| \sum_{\substack{\frac{x}{m} \le n \le \frac{x+y}{m} \\ P(n)\le z}} f(n) \Big|^{2}.
 \end{equation*}
The goal of this section is to show the following.
\begin{theorem}[Conditional CLT]\label{thm: conditionCLT}
     Let $f(n)$ be a Steinhaus random multiplicative function. Let $x$ be large and $x/(\log x)^{1/2}\le y =o(x)$, and $z=x^{1/\log \log \log x}$. Let $V_f(x, y)$ be defined as above.  Let $\widetilde{\P}$ denote the conditional probability, conditioning on all the values $(f(p))_{p\le z}$. Then for asymptotically almost all realizations of $(f(p))_{p\le z}$, we have as $x\to +\infty$,
\[ \widetilde{\P} \Big( \frac{1}{\sqrt{V_f(x, y)}} \sum_{\substack{x\le n \le x+y\\ P(n) > z}}f(n) \in R \Big) \rightarrow \P (Z \in R), \]
where $Z$ is a standard complex Gaussian random variable with mean $0$ and variance $1$ and $R$ is any nice region in the complex plane. 
 \end{theorem}

We state this a little loosely (e.g. we do not bother to specify what constitutes a ``nice region''), as noted above we will actually prove a quantitative statement about the proximity of probabilistic characteristic functions, which would imply a quantitative version of Theorem \ref{thm: conditionCLT}.   To establish this, we shall apply the following complex-valued version of the martingale central limit theorem \cite{McLeish}, established in \cite{SoundXu}.

\begin{theorem}[Special case of Theorem 3.1 of Soundararajan--Xu~\cite{SoundXu}]\label{thm: SX} \label{thm: a_n} Let $f$ denote a Steinhaus random multiplicative function, and let $a_m$ denote a sequence of complex numbers.
Put 
$$ V= \sum_{1< m\le M} |a_m|^2,
$$ 
and define the complex valued random variable 
$$ 
Z:= \frac{1}{\sqrt{V}} \sum_{1<m\le M} a_m f(m). 
$$ 
Suppose that for some $1 \ge \epsilon >0$ the following two conditions hold:

(1).  We have 
$$ 
\Big| \sum_{\substack{ 1<m_1, m_2, m_3, m_4 \le M \\ m_1 m_2= m_3 m_4 \\ m_1 \neq m_3, m_2\neq m_4 \\ P(m_1)=P(m_3) \\ P(m_2)=P(m_4)}} a_{m_1} a_{m_2} 
\overline{a_{m_3}a_{m_4}} \Big| \le \epsilon^2 V^2.
$$ 

(2).  We have 
$$ 
\Big| \sum_{\substack{ 1<m_1, m_2, m_3, m_4 \le M  \\ m_1 m_2= m_3 m_4  \\ P(m_1)=P(m_2) =P(m_3)=P(m_4)}} a_{m_1} a_{m_2} 
\overline{a_{m_3}a_{m_4}} \Big| \le \epsilon^4 V^2.
$$ 

Then for any real numbers $t_1$ and $t_2$ we have, with $t^2 =(t_1^2+ t_2^2)/2$,
$$ 
\E[ e^{it_1 \text{Re}(Z) + it_2 \text{Im}(Z)} ]= e^{-t^2/2} + O(e^{t^2} \epsilon). 
$$ 
\end{theorem} 

Let $\CA:=\{1<m \le x+y: p|m \implies p>z\}$. 
We apply the above criteria to our setting with the following:
\[M=x+y, \quad a_m = \mathbbm{1}_{\CA}(m) \cdot \sum_{\substack{\frac{x}{m} \le n \le \frac{x+y}{m} \\ P(n)\le z}} f(n), \quad V= V_f(x, y). \]
Note that once we condition on the $(f(p))_{p\le z}$, the $a_m$ and $V$ become fixed complex numbers.

The goal is to show that with probability $1- O((\log T)^{-\frac 15}) -O((\log \log z)^{-\frac{1}{10}})$ over realizations of $(f(p))_{p\le z}$, the following 
conditions hold:
\begin{equation}\label{eqn: cross}
 \Big| \sum_{\substack{ m_1, m_2, m_3, m_4 \in \CA \\ m_1 m_2= m_3 m_4 \\ m_1 \neq m_3, m_2\neq m_4 \\ P(m_1)=P(m_3) \\ P(m_2)=P(m_4)}} \sum_{\substack{n_1, n_2, n_3, n_4 \\ \frac{x}{m_j} \le n_j \le \frac{x+y}{m_j} \\ P(n_j)\le z}} f(n_1)f(n_2)\overline{f(n_3) f(n_4)}  \Big| \ll \frac{V_f(x, y)^2}{(\log\log z)^{\frac 15}} ,
\end{equation}
and 
\begin{equation} \label{eqn: linden}
\Big| \sum_{\substack{ m_1, m_2, m_3, m_4 \in \CA \\ m_1 m_2= m_3 m_4  \\ P(m_1)=P(m_2) =P(m_3)=P(m_4)}} \sum_{\substack{n_1, n_2, n_3, n_4 \\ \frac{x}{m_j} \le n_j \le \frac{x+y}{m_j} \\ P(n_j)\le z}} f(n_1)f(n_2)\overline{f(n_3) f(n_4)}  \Big| \ll \frac{V_f(x, y)^2}{\sqrt{z}} .
\end{equation}
Verifying these conditions, especially the difficult \eqref{eqn: cross}, will occupy most of this section.   Before establishing them, we first show how \eqref{eqn: cross} and \eqref{eqn: linden} may be used with Theorem \ref{thm: SX} to obtain our main result Theorem \ref{thm: main}.  

\begin{proof}[Proof of Theorem~\ref{thm: main}, assuming \eqref{eqn: cross} and  \eqref{eqn: linden}]
Invoking Theorem \ref{thm: concentration}, with probability $1- O((\log T)^{-\frac 15})$ over realizations of $(f(p))_{p\le z}$ we have 
\begin{equation} 
\label{4.3} 
V_f(x, y) = \Big(1+O\Big(\frac{1}{(\log T)^{\frac 15}}\Big)\Big) V(x, y) \asymp y \min \Big( 1, \frac{\log (x/y)}{\sqrt{\log \log x}}\Big), 
\end{equation} 
where we set $V(x,y) = e^{-\gamma } \frac{y}{\log z} \cdot \mu(x,y)$. Note that $V(x,y)$ obeys the estimates claimed in Theorem~\ref{thm: main}.

Let us temporarily write $\tilde{\E}$ to denote expectation conditional on the values $(f(p))_{p\le z}$.  Further, write temporarily 
$$ 
Z = \frac{1}{\sqrt{V_f(x,y)}} \sum_{\substack{x\le n \le x+y \\ P(n)>z}  } f(n), \text{   and   }  {\widetilde Z} = \frac{1}{\sqrt{V(x,y)}} \sum_{\substack{x\le n \le x+y \\ P(n)>z}  } f(n).
$$
Then using \eqref{eqn: cross}, \eqref{eqn: linden} and Theorem \ref{thm: SX} with $\epsilon = O((\log\log z)^{-\frac{1}{10}})$, we see that with probability $1- O((\log T)^{-\frac 15}) -O((\log \log z)^{-\frac 1{10}})$ over $(f(p))_{p\le z}$ we get
$$ 
\tilde{\E}\Big[\exp(it_1 \text{Re} Z  + it_2 \text{Im}Z )\Big] = e^{-t^2/2} + O\Big(\frac{e^{t^2}}{(\log\log z)^{\frac 1{10}}}\Big) . 
$$
With the same high probability, we may also assume that \eqref{4.3} holds.  Note that \eqref{4.3} gives, 
\begin{align*}
\tilde{\E}[|Z-\tilde{Z}|] &= \Big| \frac{1}{\sqrt{V(x,y)}} - \frac{1}{\sqrt{V_f(x,y)}}\Big| \tilde{\E}\Big[ \Big| \sum_{\substack{x\le n \le x+y \\ P(n)>z}  } f(n)\Big|\Big] 
\\
&\ll \frac{1}{(\log T)^{\frac 15}} \frac{1}{\sqrt{V_f(x,y)}} \Big(\tilde{\E}\Big[ \Big| \sum_{\substack{x\le n \le x+y \\ P(n)>z}  } f(n)\Big|^2 \Big]\Big)^{\frac 12} 
\ll \frac{1}{(\log T)^{\frac 15}}, 
\end{align*} 
where the last steps follow by Cauchy's inequality and recalling that $V_f(x,y)$ is the conditional variance.  Therefore 
$$ 
\tilde{\E}[|\exp(it_1 \text{Re} Z+it_2 \text{Im} Z) - \exp(it_1 \text{Re }\tilde{Z} + it_2 \text{Im}\tilde{Z})|]  \ll \tilde{\E}[ (|t_1| + |t_2|) |Z-\tilde{Z}|] \ll \frac{|t_1|+|t_2|}{(\log T)^{\frac 15}}, 
$$ 
and we conclude that with probability  $1- O((\log T)^{-\frac 15}) -O((\log \log z)^{-\frac 1{10}})$ over $(f(p))_{p\le z}$ there holds 
$$ 
\tilde{\E}\Big[\exp(it_1 \text{Re} \tilde{Z}  + it_2 \text{Im}\tilde{Z} )\Big] = e^{-t^2/2} + O\Big(\frac{e^{t^2}}{(\log\log z)^{\frac 1{10}}} + \frac{|t_1|+|t_2|}{(\log T)^{\frac 15}}\Big) . 
$$
 
Finally, averaging over $(f(p))_{p\le z}$ as well (using the Tower Property of conditional expectation, and the fact that the conditional characteristic function is always bounded by 1) we find the characteristic function of $\frac{1}{\sqrt{V(x,y)}} \sum_{\substack{x\le n \le x+y \\ P(n)>z}  } f(n)$ ({\em without any conditioning}) is
$$ 
= e^{-t^2/2} + O\Big(\frac{e^{t^2}}{(\log\log z)^{\frac{1}{10}}} + \frac{1 + |t_1| + |t_2|}{\log^{1/5}T} \Big) . 
$$
As $T \asymp x/y \rightarrow \infty$, this is $= e^{-t^2/2} + o(1)$ (i.e. converging to the characteristic function $e^{-t^2/2}$ of the standard complex Gaussian), implying convergence in distribution.
\end{proof}

Thus it remains only to establish the estimates in \eqref{eqn: cross} and \eqref{eqn: linden} with suitably high probability.  We begin with a lemma which will quickly lead to a proof of \eqref{eqn: linden}.  

\begin{lemma} \label{lem4.3}  Suppose $w \ge 1$ and ${\mathcal I}$ is a set of integers in $[1,w]$.   Then, for any complex numbers $a(n)$ with $|a(n)|\le 1$,  
$$ 
\E\Big[ \Big| \sum_{n\in {\mathcal I}} a(n) f(n) \Big|^4 \Big] \ll w |{\mathcal I}| (\log 2w)^3. 
$$
\end{lemma} 
 \begin{proof}  Expanding out the fourth moment and using orthogonality, the desired quantity is
 $$ 
 \sum_{\substack{ n_1, n_2, n_3, n_4 \in {\mathcal I}\\ n_1 n_2= n_3 n_4}} a(n_1) a(n_2) \overline{a(n_3) a(n_4)} \le \sum_{n_1, n_2 \in {\mathcal I}} d(n_1 n_2) \le \Big(\sum_{n\in {\mathcal I}} d(n) \Big)^2, 
 $$ 
 where $d(n)$ denotes the usual divisor function, and the last estimate follows since $d(n_1 n_2) \le d(n_1) d(n_2)$.  Now Cauchy--Schwarz shows that the above is 
 $$ 
 \le |{\mathcal I}| \sum_{n\le w} d(n)^2 \ll w |{\mathcal I}| (\log 2w)^3, 
 $$
 completing our proof.
 \end{proof} 
 
 \begin{proof}[Proof of \eqref{eqn: linden}] Using the triangle inequality, the left side of \eqref{eqn: linden} may be bounded by 
 $$ 
\sum_{\substack{ m_1, m_2, m_3, m_4 \in \CA \\ m_1 m_2= m_3 m_4  \\ P(m_1)=P(m_2) =P(m_3)=P(m_4)}} \Big|\sum_{\substack{n_1, n_2, n_3, n_4 
\\ \frac{x}{m_j} \le n_j \le \frac{x+y}{m_j} \\ P(n_j)\le z}} f(n_1)f(n_2)\overline{f(n_3) f(n_4)}  \Big|.         
$$ 
Take the expectation of this quantity over all $(f(p))_{p\le z}$.  Using H{\" o}lder's inequality followed by Lemma \ref{lem4.3}, we may bound this expectation by 
\begin{align*}
&\le 
 \sum_{\substack{ m_1, m_2, m_3, m_4 \in \CA \\ m_1 m_2= m_3 m_4  \\ P(m_1)=P(m_2) =P(m_3)=P(m_4)}}  \prod_{j=1}^{4} 
 \Big( \E \Big[ \Big| \sum_{\substack{\frac x{m_j} \le n_j \le \frac{x+y}{m_j} \\ P(n_j) \le z} } f(n)\Big|^4 \Big]\Big)^{\frac 14} 
\\
&\ll  \sum_{\substack{ m_1, m_2, m_3, m_4 \in \CA \\ m_1 m_2= m_3 m_4  \\ P(m_1)=P(m_2) =P(m_3)=P(m_4)}}  \prod_{j=1}^{4} 
\Big( \frac{(x+y)}{m_j} \left(\frac{y}{m_j} + 1 \right) (\log (2(x+y)))^3 \Big)^{\frac 14}  \\ 
&\ll x^2 (\log x)^3  \sum_{\substack{ m_1, m_2, m_3, m_4 \in \CA \\ m_1 m_2= m_3 m_4  \\ P(m_1)=P(m_2) =P(m_3)=P(m_4)}}  \frac{1}{m_1 m_2}
\ll  x^2 (\log x)^3 \sum_{\substack{m_1, m_2 \in {\mathcal A} \\ P(m_1) = P(m_2)}} \frac{d(m_1 m_2)}{m_1 m_2}. 
\end{align*}
If we write $p=P(m_1)=P(m_2)$, so that $p > z$, then using $d(m_1m_2) \le d(m_1) d(m_2)$ we may bound the above by 
$$ 
\ll x^2 (\log x)^3 \sum_{p>z} \Big( \sum_{m \le (x+y)/p} \frac{d(mp)}{mp}\Big)^2 \ll x^2 (\log x)^3 \sum_{p> z} \frac{1}{p^2} (\log x)^4 \ll \frac{ x^2 (\log x)^7 }{z}. 
$$  
Since $y \geq x/(\log x)^{1/2}$ by assumption, our bound is $\ll y^2 (\log x)^{8}/{z}$. Markov's inequality allows us to deduce that, with probability at least $1- O(z^{-\frac 14})$ (over realizations of the $(f(p))_{p\le z}$), the left side of \eqref{eqn: linden} is  $\ll y^{2} (\log x)^{8}/{z^{3/4}}$.   In view of \eqref{4.3}, which holds with suitably high probability, we conclude that \eqref{eqn: linden} holds with the desired high probability.  
 \end{proof}

It remains to verify the more demanding condition \eqref{eqn: cross}. The argument is a bit lengthy, but breaks into a few distinct phases. We shall bound the left hand side of \eqref{eqn: cross} by a (smoothed) quadruple integral of our random Euler products. We will then show that (with high probability) the ranges of integration may be significantly truncated to a very near diagonal portion. Finally, a barrier analysis (similar as in section \ref{SEC: CONCEN}, but less delicate) will give an acceptable bound for this portion of the quadruple integral.

\noindent {\bf Phase One: Pass to contour integrals.}  Put 
\begin{equation} 
\label{4.4} 
\delta= \log (1 +y/x) \asymp y/x, \qquad \text{ and } \qquad \delta_1 = (\log x)^{-40}. 
\end{equation} 
For positive real numbers $u$, $v$ define two functions $a(u;v)$ and $b(u;v)$ by setting them both to be $0$ if $u<ve^{-\delta_1}$ or if $u> ve^{\delta}$.   In the range $ve^{-\delta_1} \le u \le v e^{\delta}$ define 
\begin{equation} 
\label{4.5} 
a(u;v) = \begin{cases} 
\delta_1^{-1} \log(ue^{\delta_1}/v) &\text{ if } ve^{-\delta_1} \le u \le v \\ 
1 & \text{ if } v \le u \le ve^{\delta-\delta_1}\\ 
\delta_1^{-1} \log (e^{\delta}v/u) &\text{ if } ve^{\delta-\delta_1} \le u\le ve^{\delta} .\\ 
\end{cases}
\end{equation} 
The function $b(u;v)$ is defined by setting 
\begin{equation} 
\label{4.6} 
b(u;v) = \begin{cases} 
1- a(u;v) &\text{ if } v\le u \le e^{\delta}v\\ 
-a(u;v) &\text{ if } ve^{-\delta_1} \le u\le v .\\
\end{cases}
\end{equation} 
Thus $a(u;v)+b(u;v)$ is the indicator function of the condition $u \in [v,ve^{\delta}]$;  the function $a(u;v)$ is a smoothed approximation to this indicator function, and $b(u;v)$ is the error incurred in the smoothing.  

Examining the left side of \eqref{eqn: cross}, we use the above notation to write the inner sum over $n_1$ as 
$$ 
\sum_{P(n_1) \le z} a(n_1;x/m_1) f(n_1) + \sum_{P(n_1) \le z} b(n_1; x/m_1) f(n_1).
$$ 
Similar expressions hold for the sums over $n_2$, $n_3$, $n_4$, taking care to replace $f$ by its conjugate in the cases of $n_3$ and $n_4$.   Thus the left side of \eqref{eqn: cross} may be bounded in terms of $16$ sums, depending on which combination of $a(n_j;x/m_j)$ or $b(n_j; x/m_j)$ arises.   Of these $16$ terms, we now show that only the term with $a(n_j;x/m_j)$ appearing for all $1\le j\le 4$ is significant, and the remaining $15$ terms may be bounded easily.   

Consider one of the $15$ terms where $b(n_j;x/m_j)$ appears at least once.  Suppose $b(n_1;x/m_1)$ appears, and for $j=2$, $3$, $4$ we are indifferent to whether $a(n_j;x/m_j)$ or $b(n_j;x/m_j)$ occurs, and denote by $c(n_j;x/m_j)$ either of these possibilities.  We now bound the expected value of the contribution of such terms to \eqref{eqn: cross}: namely, 
$$
\sum_{\substack{ m_1, m_2, m_3, m_4 \in \CA \\ m_1 m_2= m_3 m_4 \\ m_1 \neq m_3, m_2\neq m_4 \\ P(m_1)=P(m_3) \\ P(m_2)=P(m_4)}}  
\E \Big[ \Big| \sum_{\substack{ n_1, n_2, n_3, n_4 \\ P(n_j) \le z}}  b(n_1;x/m_1) f(n_1) c(n_2;x/m_2) f(n_2) \prod_{j=3}^{4} c(n_j; x/m_j) \overline{f(n_j)}\Big|\Big]. 
$$ 
By applications of H{\" o}lder's inequality followed by Lemma \ref{lem4.3} we may bound the inner expectation by 
\begin{align*}
&\Big( \E \Big[ \Big| \sum_{P(n_1)\le z} b(n_1;x/m_1) f(n_1)\Big|^4\Big]\Big)^{\frac 14} \prod_{j=2}^{4} \Big(\E \Big[ \Big| \sum_{P(n_j) \le z} c(n_j;x/m_j) f(n_j)\Big|^4 \Big] \Big)^{\frac 14} \\
&\ll \Big( \E \Big[ \Big| \sum_{P(n_1)\le z} b(n_1;x/m_1) f(n_1)\Big|^4\Big]\Big)^{\frac 14} \prod_{j=2}^{4} \Big( \frac{x^2}{m_j^2} (\log x)^3 \Big)^{\frac 14}
\end{align*} 
We sum this over $m_1$, $m_2$, $m_3$, $m_4$, keeping only the constraints that these variables are $\le x+y$ and satisfy $m_1 m_2 = m_3 m_4$.  Given $m_1$, $m_2$, there are 
at most $d(m_1 m_2) \le d(m_1) d(m_2)$ choices for $m_3$ and $m_4$.   Thus the total contribution to \eqref{eqn: cross} from this case is 
\begin{align*}
&\ll \sum_{m_1, m_2  \le x+y} \Big( \E \Big[ \Big| \sum_{P(n_1)\le z} b(n_1;x/m_1) f(n_1)\Big|^4\Big]\Big)^{\frac 14} x^{\frac 32} (\log x)^{\frac 94} \frac{d(m_1)d(m_2)}{\sqrt{m_1} m_2}  \\
&\ll x^{\frac 32} (\log x)^{\frac {17}{4}} \sum_{m_1 \le x+y} \frac{d(m_1)}{\sqrt{m_1}} \Big( \E \Big[ \Big| \sum_{P(n_1)\le z} b(n_1;x/m_1) f(n_1)\Big|^4\Big]\Big)^{\frac 14}.
\end{align*}
To estimate this sum, we distinguish two cases: (i) when $m_1 \le x/(\log x)^{40}$ and (ii) when $x/(\log x)^{40} < m_1 \le x+y$.   In the first case, note that (when $b(n_1;x/m_1)$ is non-zero) $n_1$ lies in two intervals near $x/m_1$ and $(x+y)/m_1$ of length $\asymp \delta_1 (x/m_1) = (\log x)^{-40}(x/m_1)$.   Therefore using Lemma \ref{lem4.3} 
$$
 \E \Big[ \Big| \sum_{P(n_1)\le z} b(n_1;x/m_1) f(n_1)\Big|^4\Big] \ll \frac{x}{m_1} \frac{x}{m_1(\log x)^{40} }(\log x)^3 \ll \frac{x^2}{m_1^2 (\log x)^{37}}. 
 $$ 
 Thus the contribution of this range of $m_1$ to the expectation is 
 $$ 
 \ll x^2 (\log x)^{-5}  \sum_{m_1 \le x/(\log x)^{40}} \frac{d(m_1)}{m_1} \ll x^2 (\log x)^{-3}.
 $$ 
 In the second case, note that the sum over $n_1$ again runs over integers in two intervals around $x/m_1$ and $(x+y)/m_1$ of length $\asymp \delta_1 x/m_1$, and now there are at most a bounded number of integers in these intervals.   Thus the expectation of the sum over $n_1$ is bounded.   Moreover, usually there are no integers $n_1$ so close to $x/m_1$ or $(x+y)/m_1$, and the sum is non-zero only if for some integer $k\le 2(\log x)^{40}$ one has $m_1$ lying in an interval of length $O(x(\log x)^{-40}/k)$ around $x/k$ or $(x+y)/k$.  Thus the contribution from the second case is 
 \begin{align*} 
& \ll x^{\frac 32} (\log x)^{\frac{17}{4}} \sum_{k\le 2(\log x)^{40}} \Big( \sum_{|m_1 -x/k| \ll x(\log x)^{-40}/k} \frac{d(m_1)}{\sqrt{m_1}} + \sum_{|m_1-(x+y)/k| \ll x(\log x)^{-40}/k} 
 \frac{d(m_1)}{\sqrt{m_1}} \Big) 
 \\
 &\ll  x^{\frac 32} (\log x)^{\frac{17}{4}} \sum_{k\le 2(\log x)^{40}} \frac{\sqrt{x}}{\sqrt{k}} (\log x)^{-39} \ll \frac{x^2}{(\log x)^{10}}. 
 \end{align*}
 We conclude that the expected value of the contribution of the $15$ remainder terms to \eqref{eqn: cross} is $\ll x^2/(\log x)^3 \ll y^2/(\log x)^2$, since $y \ge x/(\log x)^{\frac 12}$.  
 By Markov's inequality we conclude that with probability $1- O((\log x)^{-1})$ the contribution of these terms to \eqref{eqn: cross} is $\ll y^2/(\log x)$, which is more than satisfactory.  
 
 Thus our goal now is to understand 
 \begin{equation} 
 \label{4.7} 
 \sum_{\substack{ m_1, m_2, m_3, m_4 \in \CA \\ m_1 m_2= m_3 m_4 \\ m_1 \neq m_3, m_2\neq m_4 \\ P(m_1)=P(m_3) \\ P(m_2)=P(m_4)}}  
\sum_{\substack{ n_1, n_2, n_3, n_4 \\ P(n_j) \le z}}  a(n_1;x/m_1) f(n_1) a(n_2;x/m_2) f(n_2) a(n_3; x/m_3) \overline{f(n_3)} a(n_4;x/m_4) \overline{f(n_4)}. 
 \end{equation} 
 The function $a(u;v)$ may be expressed as a contour integral 
 $$ 
 a(u;v) = \frac{1}{2\pi i} \int_{(c)} \Big(\frac{v}{u}\Big)^s \Big( \frac{e^{\delta s}-1}{s}\Big) \Big( \frac{1-e^{-\delta_1 s}}{\delta_1 s} \Big) ds, 
 $$ 
 where the integral is taken over the line Re$(s)=c >0$.   We take the line of integration to be $c=\tfrac 12$, and writing $s=\tfrac 12+it$ express the above as 
 $$ 
 a(u;v) = \frac{1}{2\pi} \int_{-\infty}^{\infty} \Big(\frac vu\Big)^{\frac 12+it} W(t) dt, 
 $$ 
 where the kernel $W(t)=W(t;\delta,\delta_1)$ is defined by 
 \begin{equation} 
 \label{4.8} 
 W(t)= \Big(\frac{e^{\delta(\frac 12+it)} -1}{\frac 12+it}\Big) \Big(\frac{1-e^{-\delta_1 (\frac 12+it)}}{\delta_1 (\frac 12+it)}\Big).  
 \end{equation} 
 We may readily check that $W(t)$ satisfies the bound
 \begin{equation} 
 \label{4.9} 
 W(t) \ll \min\Big( \delta, \frac{1}{1+|t|}\Big) \min\Big( 1, \frac{1}{\delta_1 (1+|t|)}\Big) = \min \Big( \delta, \frac{1}{1+|t|}, \frac{1}{\delta_1 (1+|t|)^2}\Big). 
 \end{equation}
 With this notation, it follows that 
 $$
 \sum_{P(n_1)\le z} a(n_1;x/m_1) f(n_1) = \frac{1}{2\pi} \int_{-\infty}^{\infty} F_z(\tfrac 12+it_1) \Big(\frac x{m_1}\Big)^{\frac 12+it_1} W(t_1) dt_1.  
 $$
 An identical expression holds for the sum over $n_2$, while for $n_3$ and $n_4$ we have closely analogous expressions after taking into account complex conjugation: for instance, 
 $$ 
 \sum_{P(n_3) \le z} a(n_3;x/m_3) \overline{f(n_3) }= 
 \frac{1}{2\pi} \int_{-\infty}^{\infty} \overline{F_z(\tfrac 12+it_3)} \Big(\frac{x}{m_3}\Big)^{\frac 12-it_3} \overline{W(t_3)} dt_3. 
 $$ 
 
We use these expressions for the sums over $n_j$ in \eqref{4.7}, and then bring in the outer sums over the $m_j$, replacing the condition $m_j \in {\mathcal A}$ 
by just requiring $m_j >1$ and having prime factors in the range $(z,x+y]$ (since the terms with $m_j>x+y$ have $a(n_j;x/m_j) = 0$).  For ease of notation, define 
\begin{equation} 
\label{4.10} 
G(t_1, t_2, t_3, t_4) = \rsum_{m_1, m_2, m_3, m_4 > 1} \frac{1}{m_1^{\frac 12+it_1}} \frac{1}{m_2^{\frac 12+it_2}} \frac{1}{m_3^{\frac 12-it_3}}\frac{1}{m_4^{\frac 12-it_4}}, 
\end{equation} 
where the $\star$ expresses the constraints 
\begin{align} 
\label{4.11} 
p|m_j \implies p\in (z,x+y]; \ \ &m_1 m_2 =m_3m_4; \ \ m_1 \neq m_3, m_2 \neq m_4; \nonumber\\
&\ \ P(m_1) = P(m_3), P(m_2)=P(m_4). 
\end{align}
Gathering our argument so far, we conclude that the desired quantity in \eqref{4.7} may be bounded by 
\begin{equation} 
\label{4.12} 
\ll x^2 \int_{t_1, t_2, t_3, t_4 \in {\mathbb R}} \prod_{j=1}^{4} |F_z(\tfrac 12+it_j) W(t_j)| |G(t_1,t_2,t_3,t_4)| dt_1 dt_2 dt_3 dt_4. 
\end{equation} 

\vspace{12pt}
\noindent {\bf Phase Two: Truncate the integrals.}
Our next step is to show that in \eqref{4.12} we may restrict attention to the range when all $|t_j|$ are at most $(\log x)^{100}$.  To see this, we first record a simple bound on $G(t_1,t_2,t_3,t_4)$.  Using the triangle inequality, and keeping only the conditions that $m_1m_2 =m_3 m_4$ (so that given $m_1$, $m_2$ there are at most $d(m_1 m_2) \le d(m_1)d(m_2)$ choices for $m_3$ and $m_4$) and that all prime factors of $m_j$ are in $(z,x+y]$,  we see that 
\begin{align*}
|G(t_1, t_3, t_3, t_4)| &\le \sum_{\substack{m_1, m_2 \\ p|m_j \implies p\in (z,x+y]}} \frac{d(m_1 m_2)}{m_1 m_2} \le \Big( \sum_{p|m_1 \implies p\in (z,x+y]} \frac{d(m)}{m}\Big)^2 \\
&= \prod_{z < p \le x+y} \Big(1-\frac 1p\Big)^{-4} \ll \Big(\frac{\log x}{\log z} \Big)^4. 
\end{align*}
Further, note that 
$$ 
\E\Big[ \prod_{j=1}^{4} |F_z(\tfrac 12+it_j)|\Big] \le \frac 14 \sum_{j=1}^{4} \E[ |F_z(\tfrac 12+it_j)|^4] = \E[ |F_z(\tfrac 12)|^4] \ll (\log z)^4, 
$$
where the third estimate follows by translation invariance, and the fourth moment bound follows by a straightforward calculation (or see \cite[Euler Product Result 1]{Harperlargevalue}).
Thus the expected value of the terms where one of the $|t_j|$ exceeds $(\log x)^{100}$ is 
$$ 
\ll x^2 \Big(\frac{\log x}{\log z}\Big)^4 (\log z)^4 \Big( \int_{|t| \ge (\log x)^{100}} |W(t)| dt \Big) \Big( \int_{-\infty}^{\infty} |W(t)| dt \Big)^3 \ll 
\frac{x^2}{(\log x)^{50}}, 
$$ 
upon using \eqref{4.9}.   Thus by Markov's inequality, with probability $1-O((\log x)^{-1})$ the contribution of terms to \eqref{4.12} with one of the variables $|t_j| > (\log x)^{100}$ is $\le y^2/(\log x)^{40}$ which is satisfactory.  

Henceforth we restrict attention to the case where all $|t_j|$ are below $(\log x)^{100}$.   Here a key feature will be that $|G(t_1, t_2, t_3, t_4)|$ provides a saving unless the variables $t_j$ are all close to each other. For prime number theoretic reasons, we will need such a restriction on the maximum size of the $t_j$ to show the bound we want for $|G(t_1, t_2, t_3, t_4)|$. Indeed, it follows easily from the prime number theorem and partial summation (or consult  \cite[Number Theory Result 2]{Harperlargevalue}) that for any $z \le u \le x+y$, and $|t| \le 2(\log x)^{100}$ one has 
\begin{equation} 
\label{4.13} 
\Big| \sum_{z < p\le u} \frac{1}{p^{1+it}} \Big| \le \min \Big( \log \frac{\log u}{\log z} + O(1), O\Big( \frac{1}{|t| \log z} \Big) \Big) . 
\end{equation} 
With this estimate in hand, we proceed to bounding $|G(t_1,t_2,t_3,t_4)|$.  

\begin{lemma} \label{lem4.4}  For $t\in {\mathbb R}$ define (interpreting $\gamma(0)$ as $1$)
\begin{equation}\label{eqn: gamma}
    \gamma(t) := \min\Big( 1, \frac{1}{|t|\log x} \Big). 
\end{equation}
In the range $|t_j| \le (\log x)^{100}$ for all $1\le j\le 4$, we have 
\begin{equation} \label{eqn: G upper}
    G(t_1, t_2, t_3, t_4) \ll \Big( \frac{\log x}{\log z}\Big)^4   \min\{ \gamma(t_1-t_3)\gamma(t_2-t_4),  \gamma(t_1-t_4) +\gamma(t_2-t_3) \} .
\end{equation}
\end{lemma}

 As discussed earlier, the key feature of the above bound is that unless all four points $t_1, t_2, t_3, t_4$ are close to each other (on a scale of $\approx 1/\log x$), we must have a saving.   When we apply the bound, it will be very important that the extra factor $(\log x/\log z)^4$ is not too large (a very small power of $\log\log x$ at most), and here it is crucial that $z$ was chosen fairly close to $x$ (on a logarithmic scale).

\begin{proof}[Proof of Lemma \ref{lem4.4}]  We parametrize the solutions to $m_1 m_2 = m_3 m_4$ by setting $g= (m_1, m_3)$ and $h=(m_2,m_4)$.  Then writing $m_1 = ga$ and $m_3= gb$, we find that $m_2=hb$ and $m_4=ha$.   With this parametrization, we may write 
$$ 
G(t_1, t_2, t_3, t_4) = \rsum_{a, b, g, h} \frac{1}{g^{1+i(t_1-t_3)}} \frac{1}{h^{1+i(t_2-t_4)}} \frac{1}{a^{1+i(t_1-t_4)}} \frac{1}{b^{1+i(t_2-t_3)}}, 
$$ 
where the $\star$ indicates the conditions (obtained by rewriting the conditions in \eqref{4.11}) 
$$
p|(ghab) \implies z < p\le x+y; \ \ (a,b)=1; \ \ ab >1; \ \ g, h >1; \ \ P(ab) \le \min(P(g), P(h)). 
$$ 
 
 One way to bound $G$ is by keeping $a$ and $b$ on the outside, and exploiting cancellation in the $g$ and $h$ sums.  Thus 
 $$ 
 G(t_1, t_2, t_3, t_4) \ll \sum_{\substack{ (a,b)=1 \\ ab >1 \\ p|ab \implies z < p\le x+y}} \frac {1}{ab} \Big| \sum_{\substack{ g> 1 \\ p|g \implies p\in (z,x+y] \\ P(ab) \le P(g)}} \frac{1}{g^{1+i(t_1-t_3)}} 
 \Big| \Big| \sum_{\substack{h> 1 \\ p|h \implies p\in (z,x+y] \\ P(ab) \le P(h)}} \frac{1}{h^{1+i(t_2-t_4)}}\Big|. 
 $$ 
Now the sum over $g$ may be rewritten as
$$ \sum_{\substack{ g> 1 \\ p|g \implies p\in (z,x+y]}} \frac{1}{g^{1+i(t_1-t_3)}} - \sum_{\substack{ g> 1 \\ p|g \implies p\in (z,P(ab)]}} \frac{1}{g^{1+i(t_1-t_3)}} , $$
where we have 
 $$ 
\Big| \sum_{\substack{ g> 1 \\ p|g \implies p\in (z,x+y]}} \frac{1}{g^{1+i(t_1-t_3)}} \Big| = \Big| \exp\Big( \sum_{z < p\le x+y} \Big(\frac{1}{p^{1+i(t_1-t_3)}} +O\Big(\frac{1}{p^2}\Big)\Big)\Big) -1 \Big| \ll \Big(\frac{\log x}{\log z}\Big) \gamma(t_1-t_3) ,
 $$ 
and the same bound for the sum over $g$ with all prime factors on the range $(z,P(ab)]$. (To see the estimate above, note that if $|t_1 - t_3| \geq 1/\log z$ then \eqref{4.13} implies the prime sum is $\ll \frac{1}{|t_1 - t_3| \log z}$, and use the fact that $e^w-1=O(w)$ when $w=O(1)$. If $|t_1 - t_3| \leq 1/\log x$, then \eqref{4.13} implies the prime sum has absolute value $\le \log \frac{\log x}{\log z} + O(1)$, and this gives an acceptable bound. Finally, if $1/\log x < |t_1 - t_3| < 1/\log z$ then splitting the prime sum at $e^{1/|t_1 - t_3|}$ shows it has absolute value $\le \log \frac{1}{|t_1 - t_3|\log z} + O(1)$, and this also gives an acceptable bound.)
 A similar bound holds for the sum over $h$ and we conclude that 
 $$ 
 G(t_1, t_2, t_3, t_4) \ll \Big(\frac{\log x}{\log z}\Big)^2 \gamma(t_1-t_3)\gamma(t_2-t_4) \prod_{z < p\le x+y} \Big(1-\frac 1p\Big)^{-2} 
 \ll  \Big(\frac{\log x}{\log z}\Big)^4 \gamma(t_1-t_3)\gamma(t_2-t_4).
 $$ 
 This establishes one of the bounds claimed in the lemma. 
 
 To establish the other bound, we keep the sums over $g$ and $h$ on the outside and look for cancellation in the sums over $a$ and $b$.  Thus 
 $$ 
 G(t_1, t_2, t_3, t_4) \ll \sum_{\substack{g, h >1 \\ p|gh \implies z < p\le x+y}} \frac{1}{gh} \Big| \sum_{\substack{ab > 1 \\ (a,b)=1 \\ p|ab \implies z < p\le x+y \\ P(ab) \le \min (P(g), P(h))}} \frac{1}{a^{1+i(t_1-t_4)} b^{1+i(t_2-t_3)}}\Big|. 
 $$ 
 Now the sum over $a$, $b$ equals 
 \begin{align*} 
 &\prod_{z < p \le \min(P(g),P(h))} \Big(1 + \sum_{k=1}^{\infty} 
 \frac{1}{p^{k(1+i(t_1-t_4))} }  + \sum_{\ell=1}^{\infty} \frac{1}{p^{\ell(1+i(t_2-t_3))}} \Big) - 1 
 \\
 &= \exp\Big( \sum_{z < p\le \min(P(g),P(h))} \Big(\frac{1}{p^{1+i(t_1-t_4)}} +\frac{1}{p^{1+i(t_2-t_3)}} +O\Big(\frac 1{p^2}\Big)\Big) \Big)-1 
 \end{align*}
 and using \eqref{4.13} as above we see that this is $\ll (\frac{\log x}{\log z})^{2} (\gamma(t_1-t_4) + \gamma(t_2-t_3))$.   Thus  
 \begin{align*}
 G(t_1, t_2, t_3, t_4) &\ll \Big(\frac{\log x}{\log z} \Big)^2 (\gamma(t_1-t_4) + \gamma(t_2-t_3)) \prod_{z < p \le x+y} \Big(1-\frac 1p\Big)^{-2}  \\
& \ll \Big(\frac{\log x}{\log z} \Big)^4  (\gamma(t_1-t_4) + \gamma(t_2-t_3)) , 
 \end{align*} 
as desired. 
 \end{proof}

We return to the problem of estimating \eqref{4.12} with the $t_j$ restricted to $|t_j| \le (\log x)^{100}$.   Lemma \ref{lem4.4} will allow us to make a further important truncation, to the situation where the four variables $t_j$ are all very nearly equal.   Define 
   \begin{equation*}
        \mathcal{R}: = \Big\{(t_1, t_2, t_3, t_4): |t_j| \le (\log x)^{100}; \ \  |t_1-t_3|, |t_2-t_4|\le \frac{(\log \log z)^{3}}{\log z}; \ \   |t_1-t_4|\le \frac{(\log \log z)^{9}}{\log z} \Big\}, 
    \end{equation*}
    and let ${\mathcal R}^{c}$ denote the complement of ${\mathcal R}$ in $[-(\log x)^{100}, (\log x)^{100}]^4$.

\begin{lemma}\label{lem: R} Keep notations as above.  With probability $1-O((\log \log z)^{-\frac{1}{10}})$, we have
 \begin{equation}\label{eqn: Abadint}
 x^2\int_{ \mathcal{R}^c}   \prod_{j=1}^{4} |F_z(\tfrac{1}{2}+it_j)W(t_j)| |G(t_1,t_2,t_3,t_4)|  dt_1 ... dt_4  \ll \frac{V(x,y)^2}{(\log \log z)^{\frac{1}{5}}}. 
 \end{equation}
\end{lemma}

 \begin{proof}  If $(t_1,t_2,t_3,t_4) \in {\mathcal R}^c$ then we must have either (i) $|t_1-t_3|$ or $|t_2-t_4| \ge (\log \log z)^3/\log z$, or (ii) $|t_1-t_3|$ and $|t_2-t_4|$ are 
 $\le (\log \log z)^3/\log z$ but $|t_1-t_4|$ exceeds $(\log \log z)^9/\log z$.   
 
 We begin with case (i), supposing that $|t_1-t_3| \ge (\log \log z)^3/\log z$ with an identical argument applying when $|t_2-t_4|$ is large.  Here we use the bound of $|G(t_1,t_2,t_3,t_4)|$ by $(\log x/\log z)^4 \gamma(t_1-t_3) \gamma(t_2-t_4)$ furnished by Lemma \ref{lem4.4}, together with the bound \eqref{4.9} for the weights $W(t_j)$.   Thus we seek a bound for 
\begin{align} \label{4.17} 
x^2 \Big(\frac{\log x}{\log z}\Big)^4 &\Big( \int_{\substack{|t_1|, |t_3| \leq (\log x)^{100} \\ |t_1-t_3|>(\log \log z)^{3}/\log z}} \gamma(t_1 - t_3) \prod_{j=1,3} |F_z(\tfrac{1}{2}+it_j)| \min\Big\{
\delta, \frac{1}{1+|t_j|}\Big\} dt_1 dt_3 \Big) \nonumber \\
& \cdot \Big( \int_{|t_2|, |t_4| \leq (\log x)^{100}} \gamma(t_2 - t_4) \prod_{j=2,4} |F_z(\tfrac{1}{2}+it_j)| \min\Big\{\delta, \frac{1}{1+|t_j|}\Big\} dt_2 dt_4 \Big).
\end{align}
 
 Consider the expectation of the double integral over $t_1$ and $t_3$ above.  Using \cite[Euler product result 1]{Harperlargevalue}, we obtain 
 \begin{align*}
 \E[|F_z(\tfrac 12+it_1) F_z(\tfrac 12+it_3)|] &\ll (\log z)^{\frac 12} \Big( \frac{1}{|t_1-t_3|} + (\log (2+|t_1-t_3|))^2 \Big)^{\frac 12}\\
 & \ll (\log z)^{\frac 12} \Big( \frac{1}{|t_1-t_3|^{\frac 12}} + \log \log x\Big).  
 \end{align*} 
 Since $\gamma(t_1-t_3) = (|t_1-t_3| \log x)^{-1}$ in this range, the expectation of this double integral is 
 $$ 
 \ll \frac{1}{(\log x)^{\frac 12}} \int_{\substack{|t_1|, |t_3| \leq (\log x)^{100} \\ |t_1-t_3|>(\log \log z)^{3}/\log z}} \frac{1}{|t_1-t_3|} \Big(\frac{1}{|t_1-t_3|^{\frac 12}} + \log \log x\Big) 
 \prod_{j=1, 3} \min\Big( \delta, \frac{1}{1+|t_j|}\Big) dt_1 dt_3. 
 $$ 
Bounding $\min(\delta, 1/(1+|t_1|)) \min (\delta, 1/(1+|t_3|))$ by $\min(\delta^2, 1/(1+|t_1|^2)) + \min(\delta^2, 1/(1+|t_3|^2))$, and using the symmetry of $t_1$ and $t_3$ we may bound the above by 
$$ 
\ll  \frac{1}{(\log x)^{\frac 12}} \int_{|t_1| \le (\log x)^{100}} \min\Big(\delta^2, \frac{1}{1+|t_1|^2}\Big)  \frac{(\log z)^{\frac 12}}{(\log \log z)^{\frac 32}} dt_1 
\ll \frac{\delta}{(\log \log z)^{\frac 32}}. 
$$ 
By Markov's inequality, we conclude that the double integral over $t_1$ and $t_3$ is at most $\delta (\log \log z)^{-\frac 75}$ with probability $1-O((\log \log z)^{-\frac 1{10}})$.

The double integral over $t_2$, $t_4$ may be handled exactly similarly, except now we need to include the contribution from points where $|t_2-t_4| \leq (\log \log z)^{3}/\log z$.
Using that  $\E[ |F_z(\frac 12+it_2) F_z(\tfrac 12+it_4)|] \ll \log z$, the expected value of this extra portion is 
$$ 
\ll \log z \int_{\substack{ |t_2|, |t_4| \le (\log x)^{100} \\ |t_2-t_4| \le (\log \log z)^3/\log z}} \min\Big(\delta, \frac{1}{1+|t_2|}\Big) \min\Big(\delta, \frac{1}{1+|t_4|}\Big) \min\Big(1, \frac{1}{|t_2-t_4|\log x}\Big) dt_2 dt_4,  
$$ 
 which by a small calculation is $\ll \delta$.  Once again by Markov's inequality, with probability $1-O((\log \log z)^{-\frac{1}{10}})$ the double integral over $t_2$ and $t_4$ in \eqref{4.17} is bounded by $\delta (\log \log z)^{\frac 1{10}}$.  
 
 Combining these two observations, and recalling that $\delta \asymp y/x$, we deduce that with probability $1-O((\log \log z)^{-\frac 1{10}})$ the quantity in \eqref{4.17} is 
 $$ 
 \ll x^2 \Big(\frac{\log x}{\log z}\Big)^{4} \frac{\delta}{(\log \log z)^{\frac 75}} \delta (\log \log z)^{\frac 1{10}} \ll \frac{y^2}{(\log \log z)^{\frac 65}} \ll \frac{V(x,y)^2}{(\log \log z)^{\frac 15}}. 
 $$ 

 This takes care of case (i), and we turn now to the second case when $|t_1-t_3|$ and $|t_2-t_4|$ are $\le (\log \log z)^3/\log z$ but $|t_1-t_4|$ exceeds $(\log \log z)^9/\log z$.  Note that $|t_2-t_3|$ is forced to be $\gg (\log \log z)^{9}/\log z$, so that $|G(t_1,t_2,t_3,t_4)| \ll (\log x/\log z)^4 (\gamma(t_1-t_4) + \gamma(t_2-t_3)) \ll (\log \log z)^{-8}$.  
Therefore the contribution of this case to \eqref{eqn: Abadint} is 
\begin{equation} 
\label{4.18} 
\ll \frac{x^2}{(\log \log x)^8} \Big( \int_{\substack{|t_1|, |t_3| \le (\log x)^{100} \\ |t_1-t_3| \le (\log \log z)^3/\log z} } \prod_{j=1,3} |F_z(\tfrac 12+it_j)| \min\Big(\delta^2, \frac{1}{1+|t_1|^2}\Big) dt_1 dt_3 \Big)^2,
\end{equation}
where we used that the integral over $t_2$, $t_4$ here is identical to the one over $t_1$, $t_3$.    Using  $\E[ |F_z(\frac 12+it_1) F_z(\tfrac 12+it_3)|] \ll \log z$, the expected value of the double integral above is 
$$ 
\ll \log z  \int_{\substack{|t_1|, |t_3| \le (\log x)^{100} \\ |t_1-t_3| \le (\log \log z)^3/\log z} } \min\Big(\delta^2, \frac{1}{1+|t_1|^2}\Big) dt_1 dt_3 \ll \delta (\log \log z)^3.
$$ 
Thus with probability at least $1-O((\log \log z)^{-\frac 1{10}})$ the double integral in \eqref{4.18} is $\le \delta (\log \log z)^{3+\frac {1}{10}}$, so (again recalling that $\delta \asymp y/x$) the quantity in \eqref{4.18} 
is $\ll y^2 (\log \log z)^{-\frac{9}{5}}$.   This is stronger than the stated bound and completes the proof of the lemma. 
 \end{proof}

\vspace{12pt}
\noindent {\bf Phase Three: Handle the near diagonal.}   In view of Lemma \ref{lem: R}, it remains now to estimate the integral in \eqref{4.12} when restricted to the region ${\mathcal R}$, where in particular all four variables $t_1$, $t_2$, $t_3$, $t_4$ are near each other (precisely, within $2(\log \log z)^9/\log z$ of each other).  We use the bound 
$|G(t_1,t_2,t_3,t_4)| \ll (\log x/\log z)^4 \gamma(t_1 -t_3) \gamma(t_2-t_4)$ and note that 
$$ 
\prod_{j=1}^{4} |F_z(\tfrac 12+it_j)| \le |F_z(\tfrac 12+it_1)F_z(\tfrac 12+it_2)|^2 + |F_z(\tfrac 12+it_3) F_z(\tfrac 12+it_4)|^2. 
$$ 
Thus the integral over the region ${\mathcal R}$ in \eqref{4.12} may be bounded by 
\begin{align*}
&\ll x^2 \Big( \frac{\log x}{\log z}\Big)^4 \int_{\substack{|t_1|, |t_2| \le (\log x)^{100}\\ |t_1 -t_2| \le 2 (\log \log z)^9/\log z}} |F_z(\tfrac 12+it_1)F_z(\tfrac 12+it_2)|^2 |W(t_1) W(t_2)| \\
&\times \int_{|t_1-t_3| \le 2(\log \log z)^9/\log z} \gamma(t_1-t_3) |W(t_3)| dt_3 \int_{|t_2-t_4|\le 2 (\log \log z)^9/\log z} \gamma(t_2-t_4) |W(t_4)| dt_4 dt_1 dt_2. 
\end{align*}
Here we have omitted the term corresponding to $|F_z(\tfrac 12+it_3)F_z(\tfrac 12+it_4)|^2$, which makes an identical contribution.

Now 
\begin{align*}
 \int_{|t_1-t_3| \le 2(\log \log z)^9/\log z} \gamma(t_1-t_3) |W(t_3)| dt_3 &\ll \min\Big(\delta, \frac{1}{1+|t_1|}\Big) \frac{\log \log \log x}{\log x}\\
 & \ll \frac{1}{\log z} \min \Big(\delta, \frac{1}{1+|t_1|}\Big), 
 \end{align*} 
 and a similar bound holds for the integral over $t_4$.   Thus we are left with the task of bounding
\begin{equation} 
\label{4.19} 
\frac{x^2}{(\log z)^2} \Big(\frac{\log x}{\log z}\Big)^4 \int_{\substack{|t_1|, |t_2| \le (\log x)^{100}\\ |t_1 -t_2| \le 2 (\log \log z)^9/\log z}} \prod_{j=1}^{2} 
|F_z(\tfrac 12+it_j)|^2 
\min\Big(\delta^2,\frac{1}{1+|t_j|^2}\Big) dt_1dt_2. 
\end{equation}

We use the barrier events method again to analyze the above double integral. The situation is similar to our work in section \ref{SEC: CONCEN}; while what we require now is less delicate than our work in that section, there is one further complication that the range of the $t_j$ here can be quite large compared with $T$.  Thus let $\tau^{*}$ be the smallest natural number for which $z^{e^{-\tau^*}} \leq e^{(\log\log z)^3}$, and $B$ be a suitably large fixed number. For $|t| \leq (\log x)^{100}$, let $\mathcal{G}^{*}(t)$ be the event that for all $0\le j\le \tau^* -1$ 
\begin{equation}\label{4.20} 
(\sqrt{T+|t|} e^{\tau^*-j}  \log\log z)^{-B}  \le \prod_{z^{e^{-\tau^*}} <p \le z^{e^{-j}} } \Big|1-\frac{f(p)}{p^{\frac 12+ it}}\Big|^{-1} \le \sqrt{T+|t|} e^{\tau^*-j} (\log \log z)^5 .
\end{equation}
Let $\mathcal{H}^{*}(t)$ denote the event that $\mathcal{G}^{*}(t)$ holds and in addition 
\begin{equation} 
\label{4.21} 
\prod_{z^{e^{-\tau^*}} <p \le z } \Big|1-\frac{f(p)}{p^{\frac 12+ it}}\Big|^{-1} \le \sqrt{T+|t|} \frac{e^{\tau^*}}{ (\log \log z )^{50}}.
\end{equation} 
Note that  the barriers here weaken when $|t| \geq T$, which is needed to show that $\mathcal{G}^{*}(t)$ holds for all $|t|\le (\log x)^{100}$ with high probability.  This weakening will be compensated by the decay of $|W(t)|$ for large $t$.  Another feature is that, compared to $\mathcal{H}(t)$ where there is an extra restriction at all scales $j$, the event $\mathcal{H}^{*}$ is simpler and only imposes an extra restriction on the full Euler product up to $z$. The very close proximity of $t_1$ and $t_2$ in \eqref{4.19} means that we will only use the barrier ${\mathcal H}^*(t)$ to factor out and bound a copy of the full Euler product. The extra saving $(\log\log z)^{50}$ in \eqref{4.21}, showing that (when ${\mathcal H}^*(t)$ holds) the Euler products cannot be {\em too} large, will then lead to an acceptable bound for \eqref{4.19}.

We first modify the argument of Proposition \ref{prop3.3} to show that the event ${\mathcal G}^{*}(t)$ holds for all $|t|\le (\log x)^{100}$ with high probability. 

\begin{proposition} \label{prop4.6}  With notations as above 
$$  
\P\left( \mathcal{G}^{*}(t)~\text{holds for all $|t|\le (\log x)^{100}$} \right) \geq 1 - O((\log\log x)^{-5}). 
$$ 
 \end{proposition}
\begin{proof}  The argument follows closely the pattern of Proposition \ref{prop3.3}.   For each $0\le j\le \tau^*-1$ we consider the probability that \eqref{4.20} fails for some $t$ with $|t|\le (\log x)^{100}$, and study this quantity by considering the mesh of points ${\mathcal T}_j = \{ {\hat t} = e^{j}n/\log z: \ \ n\in {\mathbb Z}, |{\hat t}|\le (\log x)^{100}\}$.  Since 
$$ 
\E\Big[ \prod_{z^{e^{-\tau^*}} <p \le z^{e^{-j}} } \Big|1-\frac{f(p)}{p^{1/2+ it}}\Big|^{-2}\Big] = \prod_{z^{e^{-\tau^*}} <p \le z^{e^{-j}} } \Big(1-\frac 1p\Big)^{-1} \sim e^{\tau^* -j}, 
$$ 
the probability that the Euler product in \eqref{4.20} gets as large as half the right side there is $\ll (T+|t|)^{-1} e^{j-\tau^*} (\log \log z)^{-10}$.  Thus the probability that this happens for some point in the mesh ${\mathcal T}_j$ is 
$$ 
\ll \sum_{|n|\le e^{-j} (\log z)(\log x)^{100}} (T+ e^{j}|n|/\log z)^{-1} e^{j-\tau^*} (\log \log z)^{-10} \ll \frac{e^{-\tau^*} (\log z)}{(\log \log z)^9} \ll (\log \log z)^{-6}. 
$$ 
Summing this over the $\ll \log \log z$ possibilities for $j$, gives a bound of $O( (\log \log z)^{-5})$ for the probability with which \eqref{4.20} can fail (in terms of the upper bound imposed there) for a point in our discretized sets.  This is the dominant contribution, and the other possibilities that arise (as in Proposition \ref{prop3.3}) may be treated similarly. 
\end{proof}

\begin{proposition}\label{prop: no t}  With notations as above, for all $|t| \leq (\log x)^{100}$, we have
   \[
   \E[\mathbbm{1}_{\mathcal{G}^{*}(t)}|F_z(\tfrac{1}{2} + it)|^{2} ] \asymp \log z \cdot \min \Big\{ 1, \frac{\log(T+|t|) + \log\log\log z}{\sqrt{\log \log z}} \Big\} 
   \]
and, noting that $(\mathbbm{1}_{\mathcal{G}^{*}(t)}-\mathbbm{1}_{\mathcal{H}^{*}(t)})$ is the indicator function of the event that ${\mathcal G}^*(t)$ holds but ${\mathcal H}^*(t)$ fails, 
\[
\E[(\mathbbm{1}_{\mathcal{G}^{*}(t)}-\mathbbm{1}_{\mathcal{H}^{*}(t)})|F_z(\tfrac{1}{2} + it)|^{2} ] \ll \log z \cdot \min \Big\{ 1, \frac{\log (T+|t|)}{\sqrt{\log \log z}} \Big\} \cdot \frac{(\log \log \log z)^{3}}{\log \log z}.
\]  
 \end{proposition}
 \begin{proof}
The proof of the first estimate is identical to the proof of Proposition~\ref{prop: mu}. By comparing the definition of $\mathcal{G}^{*}(t)$ here and $\mathcal{H}(t)$ in Proposition~\ref{prop: mu}, we see the relevant random walk event is now
\begin{align*}
 -\frac{B}{2}\log(T+|t|) - (B+1)j - B\log\log\log z - O(1)& \le \sum_{m=1}^{j}G_m \\
 &\le \frac{\log(T+|t|)}{2} + 5\log\log\log z + O(1), 
 \end{align*}
 for all $j\le \tau^{*}$.
Since $T \ll x/y \leq (\log x)^{\frac 12}$ and $|t| \leq (\log x)^{100}$, we see $\log(T+|t|) + \log\log\log z$ is small compared with $\sqrt{\log (z^{e^{-\tau^{*}}})} \asymp (\log\log z)^{\frac 32}$, so that \cite[Lemma 4]{HarperLow} applies.  The Probability Results 1, 2 in \cite{HarperLow} yield a probability estimate for the random walk 
$$
\asymp \min \Big(1, \frac{\log(T+|t|) + \log\log\log z}{\sqrt{\tau^{*}}}\Big) \asymp \min \Big(1, \frac{\log(T+|t|) + \log\log\log z}{\sqrt{\log \log z}}\Big),
$$ 
which proves the first estimate.  
     
     The second statement follows similarly to the proof of Proposition~\ref{prop: Hfail} and is indeed simpler since in the definition of $\mathcal{H}^{*}$, we only further constrain the Euler product for $z^{e^{-\tau^{*}}} < p \le z^{e^{-j}}$ for $j=0$ (in addition to the $\mathcal{G}^{*}$ constraints).  Modifying our earlier proof, we need only set $J=0$ and consider the analogue (with the barriers now corresponding to $\mathcal{G}^{*}, \mathcal{H}^{*}$) of the event $\mathcal{A}_0(t)$.  An application of the ballot theorem gives (analogously to \eqref{eqn: prob2}) that
  \begin{align*}
  \E[(\mathbbm{1}_{\mathcal{G}^{*}(t)}-\mathbbm{1}_{{\mathcal{H}^{*}(t)} } ) &|F_z(\tfrac{1}{2} + it)|^{2} ] \\
  &\ll \E[|F_z(\tfrac{1}{2} + it)|^{2} ] \min \Big\{ 1, \frac{\log (T+|t|) + \log \log \log z}{\sqrt{\tau^{*}}}\Big\} \frac{(\log \log \log z)^{2}}{\tau^{*}}. 
  \end{align*} 
Since $\tau^{*} \asymp \log \log z$, we obtain the second stated result.   
\end{proof}

We are now ready to resume the task of bounding \eqref{4.19}.   In view of Proposition \ref{prop4.6}, with high probability we may assume that in \eqref{4.19} both ${\mathcal G}^*(t_1)$ and ${\mathcal G}^*(t_2)$ hold: thus, we seek now to bound 
\begin{equation} 
\label{eqn: rest}
 \frac{x^2}{(\log z)^{2}} 
\Big(\frac{\log x}{\log z}\Big)^4 \int_{\substack{|t_1|, |t_2| \le (\log x)^{100}\\ |t_1 -t_2| \le 2 (\log \log z)^9/\log z}} \prod_{j=1}^{2} {\mathbbm{1}_{\mathcal{G}^{*}(t_j)}}
|F_z(\tfrac 12+it_j)|^2  \min\Big(\delta^2,\frac{1}{1+|t_j|^2}\Big) dt_1dt_2. 
\end{equation}

To bound \eqref{eqn: rest}  we first show that, up to acceptable errors, one can replace $\mathbbm{1}_{\mathcal{G}^{*}(t_2)}$ by the stronger condition $\mathbbm{1}_{\mathcal{H}^{*}(t_2)}$. 
Define
\begin{equation}
    I_1: =  \frac{1}{\log z}\int_{\substack{|t| \le (\log x)^{100} }} |F_z(\tfrac 12+it)|^{2} \mathbbm{1}_{\mathcal{G}^{*}(t)} \min\Big(\delta^2, \frac{1}{1+t^2} \Big)dt ,
\end{equation}
and
\begin{equation}
    I_2: = \frac{1}{\log z} \int_{\substack{|t| \le (\log x)^{100} }} |F_z(\tfrac 12+it)|^{2} (\mathbbm{1}_{\mathcal{G}^{*}(t)} - \mathbbm{1}_{\mathcal{H}^{*}(t)}) \min\Big(\delta^2, \frac{1}{1+t^2} \Big) dt.
\end{equation}
The first part of Proposition~\ref{prop: no t} shows (recalling $\delta \asymp y/x \asymp 1/T$) that 
\begin{align*}
\E[I_1] &\ll \int_{\substack{|t| \le (\log x)^{100} }} \min\Big(1, \frac{\log (T+|t|)+\log \log \log z}{\sqrt{\log \log z}}\Big) \min\Big(\delta^2, \frac{1}{1+t^2}\Big) dt \\
&\ll \delta \min\Big(1, \frac{\log T +\log \log \log z}{\sqrt{\log \log z}}\Big). 
\end{align*}
By Markov's inequality, it follows that with probability at least $1-O((\log \log z)^{-\frac 1{10}})$, we have
\begin{equation}
\label{4.25}
 I_1  \ll  (\log \log z)^{\frac 19} \delta \min\Big(1, \frac{\log T}{\sqrt{\log \log z}}\Big). 
\end{equation}
 The second part of Proposition~\ref{prop: no t} shows that 
 \begin{align*}
 \E [I_2] &\ll \int_{\substack{|t| \le (\log x)^{100} }} \min\Big(1, \frac{\log (T+|t|)}{\sqrt{\log \log z}}\Big) \frac{(\log \log \log z)^3}{\log \log z} \min\Big(\delta^2, \frac{1}{1+t^2}\Big) dt 
\\
& \ll \delta \min\Big(1, \frac{\log T}{\sqrt{\log \log z}}\Big) \frac{(\log \log \log z)^3}{\log \log z},
 \end{align*}
so that with probability $1-O((\log \log z)^{-\frac 1{10}})$ we have
\begin{equation}
\label{4.26}
 I_2  \ll   \delta \min\Big(1, \frac{\log T }{\sqrt{\log \log z}}  \Big) (\log \log z)^{-\frac 89}. 
\end{equation}
Thus with probability $1-O((\log \log z)^{-\frac 1{10}})$ both \eqref{4.25} and \eqref{4.26} hold, and the contribution to \eqref{eqn: rest} from the portion of the integral where ${\mathcal H}^*(t_2)$ fails is at most (temporarily dropping the condition that $t_1$ and $t_2$ are to be close) 
\[
\ll x^2\Big(\frac{\log x}{\log z}\Big)^4  I_1 \cdot I_2  \ll \frac{y^{2}}{(\log \log z)^{\frac 34}} \min\Big( 1, \frac{(\log T)^2}{\log \log z}\Big) \ll \frac{V(x,y)^2}{(\log \log z)^{\frac 34}},
\] 
which is satisfactory. {\em We remark that the small size of $(\log x)/\log z$ is very important at this step.}

It remains lastly to consider the contribution to \eqref{eqn: rest} where $\mathcal{H}^{*}(t_2)$ holds. Note that (using \eqref{4.21} at $t_2$ and dropping the condition ${\mathcal G}^*(t_1)$) 
\begin{align*}
\E\Big[ {\mathbbm 1}_{{\mathcal G}^*(t_1)} |F_z(\tfrac 12+it_1)|^2 {\mathbbm 1}_{{\mathcal H}^*(t_2)}& |F_z(\tfrac 12+it_2)|^2] 
\\
&\ll \frac{e^{2\tau^*}(T+|t_2|)}{(\log \log z)^{100}} \E\Big[ |F_z(\tfrac 12+it_1)|^2 \prod_{p\le z^{e^{-\tau^*}}} \Big| 1- \frac{f(p)}{p^{\frac 12+it_2}}\Big|^{-2}\Big].
\end{align*} 
Using the independence of $f$ on distinct primes, and Euler Product Result 1 in  \cite{Harperlargevalue}, we obtain 
\begin{align*}
 \E\Big[ &|F_z(\tfrac 12+it_1)|^2 \prod_{p\le z^{e^{-\tau^*}}} \Big| 1- \frac{f(p)}{p^{\frac 12+it_2}}\Big|^{-2}\Big] \\
 &=\E\Big[ \prod_{p\le z^{e^{-\tau^{*}}}} |1- \frac{f(p)}{p^{\frac{1}{2}+it_1}}|^{-2} \prod_{p\le z^{e^{-\tau^{*}}}} \Big|1- \frac{f(p)}{p^{\frac{1}{2}+it_2}}\Big|^{-2} \Big] \prod_{z^{e^{-\tau^*}} < p\le z}\Big(1-\frac 1p\Big)^{-1} \ll (e^{-\tau^*} \log z)^4 e^{\tau^*}.\\ 
 \end{align*}
Thus the expected contribution to \eqref{eqn: rest} from the terms where $\mathcal{H}^{*}(t_2)$ holds is (recall that $e^{-\tau^*} \log z$ is about $(\log \log z)^{3}$ by definition)
\begin{align*}
&\ll x^2 \frac{e^{-\tau^*} (\log z)^2}{(\log \log z)^{100}} \Big(\frac{\log x}{\log z}\Big)^4 \int_{\substack{ |t_1|, |t_2| \le (\log x)^{100} \\ |t_1-t_2| \le 2 (\log \log z)^9/\log z}} 
(T+|t_2|) \min\Big(\delta^2,\frac{1}{1+t_2^2}\Big)^2 dt_1 dt_2 
\\
&\ll \frac{x^2}{(\log \log z)^{87}} \int_{|t_2|\le (\log x)^{100}} (T+|t_2|)  \min\Big(\delta^2,\frac{1}{1+t_2^2}\Big)^2 dt_2 \ll \frac{y^2}{(\log \log z)^{87}}. 
\end{align*} 
{\em At this step, we crucially retained and used the restriction that $|t_1-t_2| \le 2 (\log \log z)^9/\log z$.} By Markov's inequality it follows that with probability $1-O((\log \log x)^{-1})$ this contribution to \eqref{eqn: rest} is $\ll y^2 (\log \log z)^{-86}$, which is better than needed.  

Tracking back through all our reductions, and keeping in mind \eqref{4.3}, at last we conclude that \eqref{eqn: cross} holds with the desired high probability.
\qed

  \section{Long sums}\label{SEC: LONG SUM}

\subsection{Proof of Proposition \ref{prop: tailsprop}}  The main theorems of Harper~\cite{HarperLow} imply the existence of 
absolute constants $c$, $C$ with  $0 < c \leq C$, such that uniformly for all large $x$ and $0 \leq q \leq 1$ we have
\begin{equation}\label{eqn: AJHmomcC}
c \left( \frac{x}{1 + (1-q)\sqrt{\log\log x}} \right)^q \leq \E[|\sum_{n \leq x} f(n)|^{2q}] \leq C \left( \frac{x}{1 + (1-q)\sqrt{\log\log x}} \right)^q .
\end{equation}
For ease of notation put temporarily $g= (\log \log x)^{\frac 14} x^{-\frac 12} \sum_{n\le x} f(n)$, so that uniformly in $q\le 1- 1/\sqrt{\log \log x}$ we find from \eqref{eqn: AJHmomcC} 
\begin{equation} 
\label{5.2} 
\frac {c}{3(1-q)} \le \E [|g|^{2q}] \le \frac{C}{(1-q)}. 
\end{equation} 
Here we used that $(2(1-q))^{-q} \ge 1/(3(1-q))$ and that $(1-q)^{-q} \le(1-q)^{-1}$ for all $0\le q<1$.  

Let $\lambda$ be large, and take 
$$ 
p = 1-\frac{1}{\log \lambda}, \qquad q = 1- \frac{c}{50 C \log \lambda}, \qquad r= 1- \frac{c}{100 C \log \lambda}, 
$$ 
so that $0< p < q < r \le 1- 1/\sqrt{\log \log x}$ for $x$ sufficiently large (in terms of $\lambda$).  We will apply \eqref{5.2} to analyze the $2p$, $2q$ and $2r$-th moments of $g$. The increased size of the $2q$-th moment compared with the $2p$-th moment will allow us to conclude that a large portion of the $2q$-th moment must be produced by values $|g| > \lambda$. On the other hand, the fact that the $2r$-th moment isn't too much bigger than the $2q$-th will imply that this size isn't all produced by values much larger than $\lambda$, and so $|g|$ must exceed $\lambda$ with reasonable probability.

Note that 
\begin{align*}
\E[ |g|^{2q} {\mathbbm 1}_{|g|>\lambda}] &= \E[|g|^{2q}] - \E[|g|^{2q} {\mathbbm 1}_{|g| \le \lambda}] \ge \E[|g|^{2q}] - \lambda^{2(q-p)} \E [ |g|^{2p}] \\
&\ge \frac{c}{3(1-q)} - e^2 \frac{C}{(1-p)} \ge 8C \log \lambda. 
\end{align*} 
On the other hand, by H{\" o}lder's inequality 
\begin{align*}
\E[ |g|^{2q} {\mathbbm 1}_{|g|>\lambda}] &\le (\P[|g| >\lambda])^{1- \frac{q}{r}} (\E[|g|^{2r}])^{\frac qr} \le (\P[|g| >\lambda])^{1- \frac{q}{r}} (100C^2 c^{-1} \log \lambda)^{\frac qr} 
\\
&\le  (\P[|g| >\lambda])^{1- \frac{q}{r}} (100C^2 c^{-1} \log \lambda) . 
\end{align*} 
We conclude that 
$$ 
\P[|g| >\lambda] \ge \Big( \frac{2c}{25C}\Big)^{r/(r-q)} \ge \Big( \frac{2c}{25C}\Big)^{1/(r-q)} = \lambda^{-A} 
$$
with $A = 100(C/c) \log (25C/2c)$, which proves the desired bound.  

\vspace{12pt}
Similarly for $\sum_{x \leq n \leq (1+\delta)x} f(n)$, there exist constants $0 < c(\delta) \leq C(\delta)$ such that
\begin{equation}
c(\delta) \Big( \frac{\delta x}{1 + (1-q)\sqrt{\log\log x}} \Big)^q \leq \E[\Big|\sum_{x \leq n \leq (1+\delta)x} f(n) \Big|^{2q}] \leq C(\delta) \Big( \frac{\delta x}{1 + (1-q)\sqrt{\log\log x}} \Big)^q
\end{equation}
for all large enough $x$ (depending on $\delta$) and $0 \leq q \leq 1$. The existence of $C(\delta)$ follows immediately from \eqref{eqn: AJHmomcC} and the triangle inequality, or from Theorem 1.1 of Caich~\cite{caichshort} (which would yield a much better dependence on $\delta$, in fact with $C(\delta)$ uniformly bounded if $q$ is close enough to 1 depending on $\delta$). The existence of $c(\delta)$ follows by slightly adapting the arguments of Harper~\cite{HarperLow} along the lines of section \ref{SEC: CONDVAR} here.  For example one could show that 
(omitting smaller order terms below) 
\begin{eqnarray}
\E[\Big|\sum_{x \leq n \leq (1+\delta)x} f(n) \Big|^{2q}] & \gg & (\delta x)^q \E\Big[\Big(\frac{\delta}{\log x} \int_{-1/(100\delta)}^{1/(100\delta)} \Big|F_{x^{3/4}}(\tfrac{1}{2} + \tfrac{B\log(1/\delta)}{\log x} +it)\Big|^{2} dt \Big)^{q}\Big]  \nonumber \\
& \geq & (\delta x)^q \E\Big[\Big(\frac{\delta}{\log x} \int_{-1/2}^{1/2} \Big|F_{x^{3/4}}(\tfrac{1}{2} + \tfrac{B\log(1/\delta)}{\log x} +it)\Big|^{2} dt \Big)^{q}\Big] \nonumber
\end{eqnarray}
for a suitable large constant $B$, and then the results in section 5 of Harper~\cite{HarperLow} show this is indeed $\geq c(\delta) \left( \frac{\delta x}{1 + (1-q)\sqrt{\log\log x}} \right)^q$. 
The claimed lower bound in Proposition \ref{prop: tailsprop} can then be obtained by exactly the same argument as for $\sum_{n \leq x} f(n)$.
\qed

\subsection{Long sums cannot converge to a non-degenerate Gaussian}
As mentioned in the Introduction, a particular consequence of Proposition \ref{prop: tailsprop} (and other known results) is that there is {\em no} normalizing factor $V(x)$ for which $\frac{1}{\sqrt{V(x)}} \sum_{n \leq x} f(n)$ can converge in distribution to a non-degenerate Gaussian. Indeed \eqref{eqn: Harper} gives 
$$
 \frac{1}{\sqrt{V(x)}} \E\Big[\Big|\sum_{n \leq x} f(n)\Big|\Big] \ll \frac{1}{\sqrt{V(x)}} \frac{\sqrt{x}}{(\log\log x)^{\frac 14}}, 
$$
so that $V(x)$ must necessarily be bounded by $L {x}/(\log \log x)^{\frac 12}$ for some constant $L$.  But in this case, Proposition \ref{prop: tailsprop} yields that
$$ 
\P\Big(\frac{1}{\sqrt{V(x)}} \Big|\sum_{n \leq x} f(n)\Big| \geq \lambda \Big) \geq \P\Big(\Big|\sum_{n \leq x} f(n)\Big| \geq \lambda \sqrt{L} \frac{\sqrt{x}}{(\log\log x)^{1/4}} \Big) \geq \frac{1}{L^{A/2} \lambda^A} . $$
For sufficiently large fixed $\lambda$, this greatly exceeds the Gaussian tail, preventing convergence in distribution.

\vspace{12pt}
In fact, there is an even softer argument for showing that $\frac{(\log \log x)^{1/4}}{\sqrt{x}} \sum_{n \leq x} f(n)$ cannot converge in distribution to a Gaussian (although the probability lower bound in Proposition \ref{prop: tailsprop} seems of independent interest). For any fixed $q_0 < 1$, the moment bounds \eqref{eqn: AJHmomcC} imply that $(\frac{(\log\log x)^{1/4}}{\sqrt{x}})^{2q_0} \E[|\sum_{n \leq x} f(n)|^{2q_0}]$ is uniformly bounded as $x \rightarrow \infty$. By general probability theory (see e.g. Theorem 4.2 of Gut~\cite{Gut}), this means that for any fixed $q < q_0$ (so for any fixed $q < 1$), the sequence of random variables $(\frac{(\log\log x)^{1/4}}{\sqrt{x}})^{2q} |\sum_{n \leq x} f(n)|^{2q}$ is uniformly integrable. Again by general theory (see e.g. Theorem 5.9 of Gut~\cite{Gut}), if $\frac{(\log \log x)^{1/4}}{\sqrt{x}} \sum_{n \leq x} f(n)$ were converging to a Gaussian, we would then get that all of the moments $(\frac{(\log\log x)^{1/4}}{\sqrt{x}})^{2q} \E[|\sum_{n \leq x} f(n)|^{2q}]$ with $q < 1$ would converge to the moments of that Gaussian limit. But the $2q$-th moments of any fixed Gaussian are uniformly bounded for all $q \leq 1$, whereas the lower bound part of \eqref{eqn: AJHmomcC} shows the moments $(\frac{(\log\log x)^{1/4}}{\sqrt{x}})^{2q} \E[|\sum_{n \leq x} f(n)|^{2q}]$ can be made arbitrarily large by taking $q$ close enough to 1.

For any fixed small $\delta > 0$, exactly the same considerations apply to $\sum_{x \leq n \leq (1+\delta)x} f(n)$.
  \bibliographystyle{plain}
\bibliography{RMFFinal}{}
\end{document}